\documentclass{amsart}
\usepackage[a4paper, margin=1in]{geometry}
\usepackage{amssymb}
\usepackage[hyphens]{url}
\usepackage{amsmath}
\usepackage{amsthm}
\usepackage{mathrsfs}
\usepackage[title]{appendix}
\usepackage[utf8]{inputenc}
\usepackage[english]{babel}
\usepackage{cite}
\usepackage{hyperref}
\usepackage{cleveref}
\usepackage[section]{placeins}
\RequirePackage{textgreek}
\usepackage{enumitem}
\usepackage{bm}
\usepackage{nicefrac}
\usepackage{mathabx}
\usepackage[all,cmtip]{xy}
\usepackage{caption}
\usepackage{float}
\usepackage{mathtools}
\usepackage{xcolor}
\usepackage{xspace}
\usepackage{stmaryrd}
\usepackage{dsfont}

\hypersetup{
    colorlinks=true,
    linkcolor=blue,
    citecolor=cyan,
    urlcolor=red
}

\hypersetup{linktoc=all}
\numberwithin{equation}{subsection}

\newtheorem{thm}{Theorem}[section]
\newtheorem*{thm*}{Theorem}
\newtheorem{lem}[thm]{Lemma}
\newtheorem{prop}[thm]{Proposition}
\newtheorem{cor}[thm]{Corollary}

\newtheorem{defn}[thm]{Definition}

\newtheorem{que}[thm]{Question}

\newtheorem{mainthm}{Theorem}

\theoremstyle{remark}
\newtheorem{rem}[thm]{Remark}
\newtheorem{exm}[thm]{Example}

\newcommand{\diag}{\textsubscript{\raisebox{0.4ex}{\scalebox{0.6}{\(\blacktriangle\)}}}}\xspace
\newcommand{\first}{\textsubscript{\raisebox{0.4ex}{\scalebox{0.4}{\(\blacksquare\)}}}}\xspace

\title{\scalebox{0.9}{Periodicity of Point Processes in Abelian Groups without Lattices}}

\date{}

\author{Nachi Avraham-Re'em}
\address{Department of Mathematics, Chalmers and University of Gothenburg, Gothenburg, Sweden}
\curraddr{Department of Mathematics, Technion---Israel Institute of Technology, Haifa, Israel}
\email{nachi.avraham@gmail.com}

\author{Michael Bj\"{o}rklund}
\address{Department of Mathematics, Chalmers and University of Gothenburg, Gothenburg, Sweden}
\email{micbjo@chalmers.se}

\author{Rickard Cullman}
\address{Department of Mathematics, Chalmers and University of Gothenburg, Gothenburg, Sweden}
\email{cullman@chalmers.se}

\thanks{The research was supported by the Knut and Alice Wallenberg Foundation (KAW 2021.0258).}



\begin{document}

\begin{abstract}
We investigate the intersection covolume of cross sections for probability preserving actions of a class of abelian groups without lattices, including \(p\)-adic groups and the group of finite adeles. We show that for cross sections with a uniformly discrete return time set, the intersection covolume is bounded below by twice the intensity, revealing a strict gap compared to the lattice case. Our main theorem asserts that cut-and-project systems uniquely attain the minimal intersection covolume. As an application, we characterize the generalized Farey fractions in the finite adeles via their Banach density.
\end{abstract}

\maketitle

\thispagestyle{empty}

\setcounter{tocdepth}{1}
\tableofcontents

\section{Introduction}

Stationary point processes can be studied from the perspective of ergodic theory by realizing them as probability preserving actions of locally compact second countable (lcsc) groups equipped with cross sections. In our previous work~\cite{AvBjCuI}, we introduced the {\em intersection covolume} for such systems, which is a numerical invariant that measures their degree of periodicity, and analyzed the case of lcsc groups that may admit lattices. In that setting, {\em induced actions} from lattices, with a singleton as cross section, serve as the model of complete periodicity.

As noted in~\cite[\S1.4]{AvBjCuI}, this leads naturally to the guiding question of the present paper:
\begin{quote}
\emph{How periodic can a stationary point process be when the acting group has no lattices?}
\end{quote}

We address this question for a broad family of lcsc abelian groups without lattices, referred to here as {\em class \(\mathcal{Q}\)}. This class includes, for example, the \(p\)-adic groups \(\mathbb{Q}_{p}\) and the group of finite adeles \(\mathbb{A}_{\mathrm{fin}}\).

\subsection{Intensity and intersection covolume}

Let \(G\) be an lcsc abelian group acting in a probability preserving way on a standard probability space \(\left(X,\mu\right)\). A {\bf cross section} is a Borel set \(Y\subset X\) that meets every \(G\)-orbit and whose {\bf return times sets} 
\[\left.Y_{x}\coloneqq\{g\in G: g.x\in Y\}, \quad x\in X,\right.\] 
are locally finite in \(G\). The {\bf return times set} of \(Y\) itself is the symmetric set 
\[\Lambda_{Y}\coloneqq\{g\in G: g.Y\cap Y\neq\emptyset\}.\] 
Every cross section \(Y\) carries a transverse measure \(\mu_{Y}\) on \(Y\), playing the role of the classical Palm measure for point processes. Its total mass is the {\bf intensity} 
\[\iota_{\mu}\left(Y\right)\coloneqq\mu_{Y}\left(Y\right)\in\left(0,+\infty\right].\] 
We say that \(Y\) is {\bf locally integrable} if \(\iota_{\mu}\left(Y\right)<+\infty\); in this case we refer to \(\left(X,\mu,Y\right)\) as a {\bf transverse} \(G\){\bf-space}, and we define its {\bf intersection covolume} 
\[I_{\mu}\left(Y\right)\in\left(0,+\infty\right],\] 
which quantifies the periodicity of \(Y\).

Following~\cite[\S5.2]{AvBjCuI}, a map \(\phi:\left(X,\mu,Y\right)\to\left(W,\nu,Z\right)\) between transverse \(G\)-spaces is a {\bf transverse} \(G\){\bf-factor} if \(\phi:\left(X,\mu\right)\to\left(W,\nu\right)\) is a \(G\)-factor map of probability preserving \(G\)-spaces and 
\[Y_{x}=Z_{\phi\left(x\right)} \quad\text{for }\mu\text{-a.e. }x\in X.\] 
By~\cite[Theorem~A]{AvBjCuI} one always has 
\begin{equation}\label{eq:genineq}
I_{\mu}\left(Y\right)\geq \iota_{\mu}\left(Y\right),
\end{equation}
with equality (i.e. complete periodicity) if and only if \(\left(X,\mu,Y\right)\) admits a transverse \(G\)-factor onto a homogeneous space \(\left(\Gamma\backslash G,m_{\Gamma\backslash G},\{\Gamma\}\right)\) for some lattice \(\Gamma<G\).

\subsection{The main inequality for class \(\mathcal{Q}\) groups}

We call {\bf class \(\mathcal{Q}\)} the collection of totally disconnected, non-discrete, non-compact lcsc abelian groups whose dual is torsion free; details are given in Section~\ref{sct:classQ}. Our first theorem strengthens \eqref{eq:genineq} for class~\(\mathcal{Q}\) groups under the additional assumption that \(\Lambda_{Y}\) is uniformly discrete—equivalently, that \(e_{G}\) is isolated in \(\Lambda_{Y}^{2}\).

\begin{mainthm}\label{mthm:mineq}
Let \(G\) be a \hyperlink{classQ}{class \(\mathcal{Q}\)} group. For every ergodic transverse \(G\)-space \(\left(X,\mu,Y\right)\) such that \(\Lambda_{Y}\) is uniformly discrete,
\begin{equation}\label{eq:mineq}
I_{\mu}\left(Y\right)\geq 2\cdot \iota_{\mu}\left(Y\right).
\end{equation}
\end{mainthm}

\subsection{Abelian cut--and--project systems and their intersection covolume}

Our main result (Theorem~\ref{mthm:mineqext}) gives a complete characterization of the equality cases in \eqref{eq:mineq} as those arising from cut--and--project systems. To prepare for this, we compute the intensity and intersection covolume of such systems in the abelian setting. By~\cite[Theorem~B]{AvBjCuI}, cut--and--project systems always have finite intersection covolume, but its precise value is generally unknown; restricting to abelian groups allows us to obtain explicit formulas.

An {\bf abelian cut--and--project scheme} is a triple \(\left(G,H;\Gamma\right)\) where \(G\) and \(H\) are lcsc abelian groups and \(\Gamma<G\times H\) is a lattice projecting injectively to \(G\) and densely to \(H\). It yields the ergodic probability preserving \(G\)-space 
\[\left(\Xi,\xi\right)\coloneqq\left(\Gamma\backslash\left(G\times H\right),m_{\Gamma\backslash\left(G\times H\right)}\right),\] 
with \(G\) acting on the first coordinate. For a relatively compact Jordan measurable {\em window} \(W\subset H\), we define the cross section 
\[Y_{W}\coloneqq\left\{\Gamma\left(e_{G},w\right):w\in W\right\}.\] 
Its return times set is the (symmetric) {\em cut--and--project set} 
\[\Lambda_{Y_{W}}=\operatorname{proj}_{G}\left(\Gamma\cap\left(G\times\left(W-W\right)\right)\right).\]

\begin{mainthm}\label{mthm:cap}
Let \(\left(G,H;\Gamma\right)\) be an abelian cut--and--project scheme and \(W\subset H\) a \hyperlink{window}{window}. For the associated cut--and--project \(G\)-space \(\left(\Xi,\xi\right)\) with cross section \(Y_{W}\):
\begin{enumerate}
    \item The intensity is 
    \[\iota_{\xi}\left(Y_{W}\right)=\mathrm{covol}_{G\times H}\left(\Gamma\right)^{-1}\cdot m_{H}\left(W\right).\] 
    In particular, if \(W\) is Jordan measurable, then \(\iota_{\xi}\left(Y_{W^{o}}\right)=\iota_{\xi}\left(Y_{W}\right)\).
    \item If \(W\) is open or \(W-W\) is Jordan measurable, then 
    \[\mathrm{covol}_{G\times H}\left(\Gamma\right)^{-1}\cdot m_{H}\left(W^{o}-W^{o}\right)\leq I_{\xi}\left(Y_{W}\right)\leq \mathrm{covol}_{G\times H}\left(\Gamma\right)^{-1}\cdot m_{H}\left(W-W\right).\]
\end{enumerate}
\end{mainthm}

The proof of Theorem~\ref{mthm:cap}, combined with the Brunn--Minkowski inequality, yields:

\begin{cor}\label{cor:cap}
Let \(\left(G,\mathbb{R}^{d};\Gamma\right)\) be an abelian cut--and--project scheme and \(W\subset\mathbb{R}^{d}\) a \hyperlink{window}{window}. Then 
\begin{equation}\label{eq:corcap}
I_{\xi}\left(Y_{W}\right)\geq 2^{d}\cdot\iota_{\xi}\left(Y_{W}\right).
\end{equation}
Moreover:
\begin{enumerate}
    \item If \(W=\mathrm{conv}\left(W\right)\) modulo \(m_{\mathbb{R}^{d}}\), then equality holds in \eqref{eq:corcap}.
    \item If equality holds in \eqref{eq:corcap}, then \(W=\mathrm{conv}\left(W^{o}\right)\) modulo \(m_{\mathbb{R}^{d}}\).
\end{enumerate}
\end{cor}

\begin{rem}
The equality statements in Corollary~\ref{cor:cap} imply that if \(W\subset\mathbb{R}^{d}\) satisfies \(W=\mathrm{conv}\left(W\right)\) modulo \(m_{\mathbb{R}^{d}}\), then also \(W=\mathrm{conv}\left(W^{o}\right)\) modulo \(m_{\mathbb{R}^{d}}\). This is because bounded convex sets are well-known to be Jordan measurable (see~\cite{szabo1997simple}), and convex hull and interior commute.
\end{rem}

\subsection{The extremes of the main inequality}

We call a class \(\mathcal{Q}\) group \(G\) {\bf cosolenoidal} if \(G\times\mathbb{R}\) admits a lattice. Examples include the \(p\)-adic groups \(\mathbb{Q}_{p}\) and the group \(\mathbb{A}_{\mathrm{fin}}\) of finite adeles. As will be discussed in Section~\ref{sct:coso}, there are also many non-cosolenoidal groups within class~\(\mathcal{Q}\).

If \(G\) is cosolenoidal, it admits a cut--and--project scheme \(\left(G,\mathbb{R};\Gamma\right)\), and choosing a compact interval as a window produces a cross section achieving equality in \eqref{eq:mineq} by Corollary~\ref{cor:cap}. Our main theorem shows that, essentially, equality in \eqref{eq:mineq} occurs only for cut--and--project \(G\)-spaces.

\begin{mainthm}\label{mthm:mineqext}
Let \(G\) be a \hyperlink{classQ}{class \(\mathcal{Q}\)} group. For every ergodic transverse \(G\)-space \(\left(X,\mu,Y\right)\), the following are equivalent:
\begin{enumerate}
    \item \(\Lambda_{Y}\) is uniformly discrete and \(I_{\mu}\left(Y\right)=2\cdot\iota_{\mu}\left(Y\right)\).
    \item There exists a compact interval \(W\subset\mathbb{R}\) and a transverse \(G\)-factor from \(\left(X,\mu,Y\right)\) onto a cut--and--project \(G\)-space \(\left(\Gamma\backslash\left(G\times\mathbb{R}\right),m_{\Gamma\backslash\left(G\times\mathbb{R}\right)}\right)\) for some lattice \(\Lambda<G\times\mathbb{R}\).
\end{enumerate}
In particular, in this case \(G\) is necessarily cosolenoidal.
\end{mainthm}

We also point out a potential higher-dimensional analogue of Theorems~\ref{mthm:mineq} and~\ref{mthm:mineqext}. For \(d\in\mathbb{N}\), call an lcsc abelian group \(G\) \(d\){\bf-cosolenoidal} if \(G\times\mathbb{R}^{d}\) admits a lattice but \(G\times\mathbb{R}^{d-1}\) does not.

\begin{que}
Let \(G\) be a class \(\mathcal{Q}\) group that is \(d\)-cosolenoidal. Is it true that for every ergodic transverse \(G\)-space \(\left(X,\mu,Y\right)\), where \(Y\) is \(s\)-separated for some sufficiently large \(s\in\mathbb{N}\), one has
\[I_{\mu}\left(Y\right)\geq 2^{d}\cdot\iota_{\mu}\left(Y\right)\,\,?\]
Is it true that equality \(I_{\mu}\left(Y\right)=2^{d}\cdot\iota_{\mu}\left(Y\right)\) implies the existence of a transverse \(G\)-factor from \(\left(X,\mu,Y\right)\) to a transverse cut--and--project \(G\)-space \(\left(\Gamma\backslash\left(G\times\mathbb{R}^{d}\right),m_{\Gamma\backslash\left(G\times\mathbb{R}^{d}\right)},Y_{W}\right)\) for some lattice \(\Gamma<G\times\mathbb{R}^{d}\) and a window \(W\) with \(W=\mathrm{conv}\left(W\right)\) modulo \(m_{\mathbb{R}^{d}}\)?
\end{que}

Not every lcsc abelian group is \(d\)-cosolenoidal for some \(d\in\mathbb{N}\). Inspired by Meyer~\cite[Ch.~II, \S11]{meyer1972algebraic}, we construct examples of class~\(\mathcal{Q}\) groups \(G\) that are not \(d\)-cosolenoidal for any \(d\in\mathbb{N}\), and indeed satisfy that \(G\times H\) admits no lattice for any lcsc abelian group \(H\) (see Example~\ref{exm:Meyer}). Whether such a \(G\) can admit an ergodic transverse \(G\)-space with finite intersection covolume remains unknown.

\subsection{A Banach density characterization of generalized Farey fractions in the finite adeles}

We now combine Theorems~\ref{mthm:mineq} and~\ref{mthm:mineqext} with a transverse version of the Furstenberg correspondence principle, due to the second author and Fish~\cite{BjFi2024}, to obtain a Banach-density characterization of {\em generalized Farey fractions}. This is a continuous analog of the discrete amenable case developed by the second author together with Fish~\cite{BjFi2019} and with Fish and Shkredov~\cite{BjFi2021}. At the same time, passing to the continuous setting necessitates significant technical innovations to identify and prove the correct generalization.

Recall the classical Farey fractions
\[\mathfrak{F}\coloneqq\left\{ q\in\mathbb{Q}: q\in\left[0,1\right]\right\}.\]
We recast \(\mathfrak{F}\) as a uniformly discrete subset of the finite adeles \(\mathbb{A}_{\mathrm{fin}}=\prod\nolimits_{p\text{ prime}}^{\prime}\left(\mathbb{Q}_{p},\mathbb{Z}_{p}\right)\). This is a cosolenoidal class~\(\mathcal{Q}\) group, since the diagonal embedding \(\mathbb{Q}<\mathbb{A}\coloneqq\mathbb{A}_{\mathrm{fin}}\times\mathbb{R}\) is a lattice. Consider the cut--and--project scheme \(\left(\mathbb{A}_{\mathrm{fin}},\mathbb{R};\mathbb{Q}\right)\). Under the diagonal embedding \(\iota:\mathbb{Q}\to\mathbb{A}_{\mathrm{fin}}\), \(\iota\left(q\right)=\left(q,q,\dotsc\right)\), the set \(\mathfrak{F}\) is exactly the cut--and--project set corresponding to the window \(W=\left[0,1\right]\):
\[\mathfrak{F}=\left\{\iota\left(q\right): q\in\mathbb{Q}\cap\left[0,1\right]\right\}=\operatorname{proj}_{\mathbb{A}_{\mathrm{fin}}}\left(\mathbb{Q}\cap\left(\mathbb{A}_{\mathrm{fin}}\times\left[0,1\right]\right)\right)\subset\mathbb{A}_{\mathrm{fin}}.\]
For an arbitrary compact interval \(W\subset\mathbb{R}\), we define the {\bf generalized Farey fractions} associated with \(W\) to be
\[\mathfrak{F}\left(W\right)\coloneqq\operatorname{proj}_{\mathbb{A}_{\mathrm{fin}}}\left(\mathbb{Q}\cap\left(\mathbb{A}_{\mathrm{fin}}\times W\right)\right)=\left\{\iota\left(q\right): q\in\mathbb{Q}\cap W\right\}\subset\mathbb{A}_{\mathrm{fin}}.\]
For \(u\in\mathbb{Z}_{\mathrm{fin}}\), a finite adelic integer, we write \(u\cdot\mathfrak{F}\left(W\right)\) for the {\bf dilated generalized Farey fractions}.

The basic facts underlying this perspective will be recalled in Section~\ref{sct:gff}. For a locally finite set \(P\subset\mathbb{A}_{\mathrm{fin}}\), let \(d^{\ast}\left(P\right)\) denote its upper Banach density with respect to strong F\o lner sequences (see Section~\ref{sct:Banach}). From Corollary~\ref{cor:cap} we deduce the formulas, for the adelic norm \(\left\Vert\cdot\right\Vert_{\mathrm{fin}}\) on \(\mathbb{Z}_{\mathrm{fin}}\),
\begin{equation}\label{eq:densff}
d^{\ast}\left(u\cdot\mathfrak{F}\left(W\right)\right)=\left\Vert u\right\Vert_{\mathrm{fin}}^{-1}\cdot m_{\mathbb{R}}\left(W\right)\quad\text{and}\quad d^{\ast}\left(u\cdot\mathfrak{F}\left(W\right)-u\cdot\mathfrak{F}\left(W\right)\right)=2\cdot\left\Vert u\right\Vert_{\mathrm{fin}}^{-1}\cdot m_{\mathbb{R}}\left(W\right).
\end{equation}

The next theorem is a continuous analogue of~\cite[Theorem~1.16]{BjFi2019} and~\cite[Theorem~1.8]{BjFi2021}.

\begin{mainthm}\label{mthm:farey}
Let \(P\subset\mathbb{A}_{\mathrm{fin}}\) be such that \(P-P\) is uniformly discrete. Then
\[d^{\ast}\left(P-P\right)\geq 2\cdot d^{\ast}\left(P\right).\]
Moreover, if equality occurs so that \(d^{\ast}\left(P-P\right)=2\cdot d^{\ast}\left(P\right)>0\), then there exist \(u\in\mathbb{Z}_{\mathrm{fin}}\) and a compact interval \(W\subset\mathbb{R}\) such that
\[P-P\supseteq u\cdot\mathfrak{F}\left(W\right)-u\cdot\mathfrak{F}\left(W\right)\]
and
\[d^{\ast}\left(P-P\right)=d^{\ast}\left(u\cdot\mathfrak{F}\left(W\right)-u\cdot\mathfrak{F}\left(W\right)\right)=2\cdot\left\Vert u\right\Vert_{\mathrm{fin}}^{-1}\cdot m_{\mathbb{R}}\left(W\right).\]
\end{mainthm}

\section{Class \(\mathcal{Q}\) and cosolenoidal groups}

In this section, we introduce two classes of lcsc abelian groups that will be central to our work, along with their basic properties and examples. Since this is a purely group theoretic topic, the proofs are deferred to Appendix~\ref{app:classcoss}.

\subsection{Class \(\mathcal{Q}\)}\label{sct:classQ}

\begin{defn}
\hypertarget{classQ}{}
The {\bf class \(\mathcal{Q}\)} consists of all lcsc abelian groups satisfying the following properties:
\begin{enumerate}
    \item The group is totally disconnected, non-discrete and non-compact.
    \item The (Pontryagin) dual of the group is torsion-free.\footnote{For instance, divisible group has torsion-free dual, but the converse is not always true; see~\cite[Theorem (24.23)]{HewittRoss}.}
\end{enumerate}
\end{defn}

\begin{exm}
For a prime \(p\), the \(p\)-adic numbers \(\mathbb{Q}_{p}\) is totally disconnected, self-dual and torsion-free, and hence is class \(\mathcal{Q}\).
\end{exm}

The following example, while may be less significant from a dynamical standpoint, is still noteworthy.

\begin{exm}
Consider the rationals \(\mathbb{Q}\) and the Pr\"{u}fer group \(\mathbb{Z}\left(p^{\infty}\right)\coloneqq\mathbb{Z}_{p}\backslash\mathbb{Q}_{p}\) for a prime \(p\), both are discrete divisible groups and hence have compact torsion-free duals~\cite[Theorem (24.25)]{HewittRoss}.\footnote{The only discrete abelian groups having torsion-free dual are direct sums of \(\mathbb{Q}\) and \(\mathbb{Z}\left(p^{\infty}\right)\) for primes \(p\)~\cite[Theorem A1.42, Corollary 8.5]{hofmann2006structure}.} Therefore, \(\mathbb{Q}_{p}\times \mathbb{Q}\) and \(\mathbb{Q}_{p}\times\mathbb{Z}\left(p^{\infty}\right)\) are both class \(\mathcal{Q}\).
\end{exm}

The following key properties of class \(\mathcal{Q}\) groups are the primary source of our interest in this class. We recall that a {\bf compactification} of an lcsc abelian group \(G\) is a pair \(\left(K,\tau\right)\), consisting of a compact abelian group \(K\) and a continuous homomorphism \(\tau:G\to K\), such that \(\tau\left(G\right)\) is dense in \(K\).

\begin{prop}\label{prop:classqgrp}
For every \hyperlink{classQ}{class \(\mathcal{Q}\)} group \(G\) the following properties hold:
\begin{enumerate}
    \item \(G\) admits a local base at the identity consisting of compact open subgroups.
    \item \(G\) admits no closed cocompact proper subgroup (and in particular no lattices).
    \item Every compactification of \(G\) is connected.
\end{enumerate}
\end{prop}

Recall that the {\bf restricted direct product} (also called \emph{local direct product}), of a sequence \(\left(G_{n},K_{n}\right)_{n\in\mathbb{N}}\) of lcsc abelian groups with compact open subgroups \(K_{n}<G_{n}\), is the group
\[\prod\nolimits_{n\in\mathbb{N}}^{\prime}\left(G_{n},K_{n}\right)\]
of sequences \(\left(g_{n}\right)_{n\in\mathbb{N}}\in \prod\nolimits_{n\in\mathbb{N}}G_{n}\) with \(g_{n}\in K_{n}\) for all but finitely many \(n\)'s. It becomes an lcsc group with the restricted product topology, whose local base at the identity is given by sets of the form \(\prod\nolimits_{n\in F}U_{n}\times\prod\nolimits_{n\in \mathbb{N}\backslash F}K_{n}\), for some finite set \(F\subset\mathbb{N}\) and some open sets \(U_{n}\subseteq K_{n}\) for \(n\in F\)~\cite[(6.16)]{HewittRoss}.

\begin{prop}\label{prop:classqclass}
\hyperlink{classQ}{Class \(\mathcal{Q}\)} is closed under finite products and, more generally, under restricted direct products with respect to arbitrary choice of compact open subgroups.
\end{prop}

The observation that class \(\mathcal{Q}\) is preserved under restricted direct products provides us with the following example that is important to our Theorem~\ref{mthm:farey}.

\begin{exm}\label{exm:adeles}
The group of finite adeles, given as the restricted direct product
\[\mathbb{A}_{\mathrm{fin}}=\prod\nolimits_{p\text{ prime}}^{\prime}\left(\mathbb{Q}_{p},\mathbb{Z}_{p}\right),\]
is totally disconnected, self-dual and torsion-free, and hence is class \(\mathcal{Q}\).
\end{exm}

\subsection{Cosolenoidal groups}\label{sct:coso}

Within class \(\mathcal{Q}\) groups, which admit no lattices, there is a subclass of groups that are closer to admit a lattice in a precise sense as follows. We define this for general abelian groups, but in practice our focus is on class \(\mathcal{Q}\). The term \emph{cosolenoidal} is motivated by Example~\ref{exm:qp} below.

\begin{defn}
Let \(G\) be an lcsc abelian group.  
We call \(G\) \textbf{cosolenoidal} if \(G\times\mathbb{R}\) admits a lattice while \(G\) itself does not. More generally, we call \(G\) \textbf{$d$-cosolenoidal}, for \(d\in\mathbb{N}\), if \(G \times \mathbb{R}^{d}\) admits a lattice but \(G \times \mathbb{R}^{d-1}\) does not.
\end{defn}

The role of \(\mathbb{R}^{d}\) in the definition of cosolenoidality is universal among abelian groups:

\begin{prop}\label{prop:cosoH}
Let \(G\) be an lcsc abelian group. If there exists an lcsc abelian group \(H\) such that \(G\times H\) admits a lattice, then \(G\) is \(d\)-cosolenoidal for some \(d\in\mathbb{N}\).
\end{prop}

Next, we introduce a technical tool that will be useful in computing the degree of cosolenoidality. The reader should bear in mind the case \(G=\mathbb{Q}_{p}^{d}\) which will be discussed later in Example~\ref{exm:qp2}. Let us use the following convenient terminology: a finitely generated group is {\bf \(d\)-generated}, for some \(d\in\mathbb{N}\), if it admits a generating set of cardinality at most \(d\).

\begin{prop}\label{prop:cyclicity}
Let \(G\) be a \hyperlink{classQ}{class \(\mathcal{Q}\)} group and \(d\in\mathbb{N}\) such that \(G\times\mathbb{R}^{d}\) admits a lattice. Then for all compact open subgroups \(V<U<G\), either \(U\) admits a nontrivial finite subgroup, or \(V\backslash U\) is \(d\)-generated.
\end{prop}

The following proposition is useful in classifying the lattices in \(G\times\mathbb{R}\) for certain class \(\mathcal{Q}\) groups \(G\). The reader should bear in mind the case of \(G=\mathbb{A}_{\mathrm{fin}}\) which will be discussed later in Example~\ref{exm:adeles2}.

\begin{prop}\label{prop:adlat}
Let \(G\) be a \hyperlink{classQ}{class \(\mathcal{Q}\)} group satisfying the following properties:
\begin{enumerate}
    \item \(G\) is torsion-free and divisible.
    \item There is a compact open subgroup \(U<G\) with no nontrivial finite subgroups, and \(U\backslash G\) is torsion.
\end{enumerate}
Let \(Q\) be the set of orders of elements in \(U\backslash G\). Then every lattice \(\Gamma< G\times\mathbb{R}\) is of the form
\[\Gamma=\mathbb{Z}\big[1/Q\big]\left(u,t\right)=\big\{ \big({\textstyle \frac{m}{q}}u,{\textstyle\frac{m}{q}}t\big):m\in\mathbb{Z},q\in Q\big\}<G\times\mathbb{R},\]
for some \(\left(u,t\right)\in U\times\mathbb{R}\).
\end{prop}

We stress that Proposition~\ref{prop:adlat} does not guarantee that \(\mathbb{Z}\big[1/Q\big]\) is always a lattice in \(G\times\mathbb{R}\).

For the rest of this section we mention some important instances of cosolenoidal groups.

\begin{exm}\label{exm:qp}
For a prime \(p\), the class \(\mathcal{Q}\) group \(\mathbb{Q}_{p}\) is cosolenoidal, for there is the lattice
\[\mathbb{Z}\left[1/p\right]<\mathbb{Q}_{p}\times \mathbb{R},\]
where \(\mathbb{Z}[1/p]\) is embedded diagonally in \(\mathbb{Q}_{p}\times \mathbb{R}\). Indeed, one verifies that \(\mathbb{Z}\left[1/p\right]<\mathbb{Q}_{p}\times\mathbb{R}\) is discrete, with fundamental domain \(\mathbb{Z}_{p}\times\left[0,1\right)\). In fact, by Proposition~\ref{prop:adlat} every lattice in \(\mathbb{Q}_{p}\times\mathbb{R}\) is a dilation of \(\mathbb{Z}\left[1/p\right]\). The quotient space associated with \(\mathbb{Z}\left[1/p\right]<\mathbb{Q}_{p}\times\mathbb{R}\) is the \(p\)-solenoid (see~\cite[pp. 114-115]{HewittRoss}):
\[\mathbb{Z}\left[1/p\right]\backslash\left(\mathbb{Q}_{p}\times\mathbb{R}\right)\cong\mathbb{S}_{p}\coloneqq\varprojlim\big(\mathbb{T}\xleftarrow{z^{p}\mapsfrom z}\mathbb{T}\xleftarrow{z^{p}\mapsfrom z}\mathbb{T}\dashleftarrow\big)=\left\{ \left(z_{n}\right)_{n\in\mathbb{N}}\in\mathbb{T}^{\mathbb{N}}:\forall n\in\mathbb{N},\,\,z_{n}=z_{n+1}^{p}\right\}.\]
\end{exm}

Using Proposition~\ref{prop:cyclicity}, we can find \(d\)-cosolenoidal groups for arbitrary \(d\in\mathbb{N}\):

\begin{exm}\label{exm:qp2}
For \(d\in\mathbb{N}\) and \(p\) prime, the class \(\mathcal{Q}\) group \(\mathbb{Q}_{p}^{d}\) is \(d\)-cosolenoidal, for there is the lattice
\[\mathbb{Z}\left[1/p\right]^{d}<\mathbb{Q}_{p}^{d}\times\mathbb{R}^{d},\]
but there is no lattice in \(\mathbb{Q}_{p}^{d}\times\mathbb{R}^{d-1}\). Indeed, consider the compact open subgroups \(U\coloneqq\mathbb{Z}_{p}^{d}\) (which admits no nontrivial finite subgroups) and \(V\coloneqq\left(p\mathbb{Z}_{p}\right)^{d}\), whose quotient \(V\backslash U\) is isomorphic to \(\left(p\mathbb{Z}\backslash\mathbb{Z}\right)^{d}\) which is not \(d-1\)-generated. From Proposition~\ref{prop:cyclicity} it follows that \(\mathbb{Q}_{p}^{d}\times\mathbb{R}^{d-1}\) admits no lattice. Note however that \(\mathbb{Q}_{p}\times\mathbb{Q}_{q}\) for distinct primes \(p\) and \(q\) is cosolenoidal, for there is the lattice \(\mathbb{Z}\left[1/pq\right]<\mathbb{Q}_{p}\times\mathbb{Q}_{q}\times\mathbb{R}\).
\end{exm}

The following generalizes a construction of Meyer~\cite[Ch. II, \S11]{meyer1972algebraic}, who considered the case of \(G_{o}=\mathbb{Z}_{p}\) and \(U_{o}=p\mathbb{Z}_{p}\) in the following example.\footnote{Although \(\mathbb{Z}_{p}\) is compact and thus not class \(\mathcal{Q}\), our reasoning remains applicable in this case, since Proposition~\ref{prop:cyclicity} actually holds for all lcsc totally disconnected abelian groups with torsion-free dual, such as \(\prod\nolimits_{\mathbb{N}}\left(\mathbb{Z}_{p},p\mathbb{Z}_{p}\right)\).}

\begin{exm}\label{exm:Meyer}
Let \(G_{o}\) be any class \(\mathcal{Q}\) group with a compact open subgroup \(U_{o}<G_{o}\) admitting no nontrivial finite subgroups (such as \(G_{o}=\mathbb{Q}_{p}\) and \(U_{o}=\mathbb{Z}_{p}\)), and form the restricted direct product
\[G\coloneqq\prod\nolimits_{\mathbb{N}}^{\prime}\left(G_{o},U_{o}\right).\]
Then \(G\) is a class \(\mathcal{Q}\) group with the property that \(G\times H\) admits no lattice for any lcsc abelian group \(H\). In order to show this, by Proposition~\ref{prop:cosoH} it suffices to show that \(G\) is not \(d\)-cosolenoidal for any \(d\in\mathbb{N}\) (since \(G\) is class \(\mathcal{Q}\), \(G\) itself does not have a lattice). Indeed, pick some compact open proper subgroup \(V_{o}<U_{o}\) and, for an arbitrary \(d\in\mathbb{N}\), consider the compact open subgroups
\[U\coloneqq\prod\nolimits_{\mathbb{N}}U_{o}<G\quad\text{and}\quad V_{d}\coloneqq\big(\prod\nolimits_{1\leq n\leq d+1}V_{o}\big)\times\big(\prod\nolimits_{n>d+1}U_{o}\big)<G.\]
Since \(U_{o}\) admits no finite subgroup, neither does \(U\). Also, \(V_{d}\backslash U\cong\left(V_{o}\backslash U_{o}\right)^{d+1}\) is not \(d\)-generated. Since \(d\) is arbitrary, from Proposition~\ref{prop:cyclicity} it follows that \(G\) is not \(d\)-cosolenoidal.
\end{exm}

The following example is of special interest to us (see Section~\ref{sct:gff}).

\begin{exm}\label{exm:adeles2}
The class \(\mathcal{Q}\) group \(\mathbb{A}_{\mathrm{fin}}\) (recall Example~\ref{exm:adeles}) is cosolenoidal, for there is the lattice
\[\mathbb{Q}<\mathbb{A}_{\mathrm{fin}}\times\mathbb{R}=\mathbb{A},\]
where \(\mathbb{Q}\) is embedded diagonally into \(\mathbb{A}\), the adeles group (see~\cite[Ch.~VII, §2]{lang1994algebraic}). In fact, every lattice in \(\mathbb{A}\) is a dilation of \(\mathbb{Q}\). Indeed, \(\mathbb{A}_{\mathrm{fin}}\) is torsion-free and divisible, and its compact open subgroup
\[\mathbb{Z}_{\mathrm{fin}}\coloneqq\prod\nolimits_{p\text{ prime}}\mathbb{Z}_{p}<\mathbb{A}_{\mathrm{fin}},\]
admits no nontrivial finite subgroups. Additionally, \(\mathbb{Z}_{\mathrm{fin}}\backslash\mathbb{A}_{\mathrm{fin}}\cong\bigoplus_{p\text{ prime}}\mathbb{Z}_{p}\backslash\mathbb{Q}_{p}\) is torsion, and every prime arises as an order of its elements, thus the set \(Q\) as in Proposition~\ref{prop:adlat} contains all the primes. By the conclusion of Proposition~\ref{prop:adlat}, every lattice in \(\mathbb{A}_{\mathrm{fin}}\times\mathbb{R}\) is of the form \(\mathbb{Q}\left(u,t\right)\) for some \(\left(u,t\right)\in \mathbb{Z}_{\mathrm{fin}}\times\mathbb{R}\).
\end{exm}

\section{Cross sections, intensity and intersection covolume}\label{sct:intcovol}

Let us recall the well-known notion of cross sections, as well as the notion of intersection covolume of cross sections, which was defined and discussed extensively in~\cite{AvBjCuI}. In the following, we introduce one way of defining the intersection covolume, which is not necessarily the most intuitive, but is the most useful for the purposes of this work.

Let \(X\) be a Borel \(G\)-space. For a Borel set \(Y\subseteq X\) and \(x\in X\), the {\bf the return times set} is the set
\[Y_{x}\coloneqq\{g\in G:g.x\in Y\}.\]
Note that
\[Y_{g.x}=Y_{x}g^{-1}\text{ for all }g\in G\text{ and }x\in X.\]
A Borel set \(Y\subseteq X\) is called a {\bf cross section} if the return times set \(Y_{x}\) of every \(x\in X\) is nonempty (so that \(G.Y=X\)) and locally finite. 
The {\bf return times set} of \(Y\) itself is the set
\[\Lambda_{Y}\coloneqq \bigcup\nolimits_{y\in Y}Y_{y}=\{g\in G:g.Y\cap Y\neq\emptyset\}.\]
Note that \(e_{G}\in\Lambda_{Y}\) and that \(\Lambda_{Y}\) is symmetric: \(\Lambda_{Y}^{-1}=\Lambda_{Y}\). A cross section \(Y\) is said to be {\bf \(s\)-separated}, for some \(s\in\mathbb{N}\), if there exists an identity neighborhood \(U\subset G\) such that \(\Lambda_{Y}^{s}\cap U=\{e_{G}\}\). In this case we will refer to \(U\) as an \(s\){\bf-separating neighborhood} for \(Y\). For \(s=1\), we will call \(Y\) {\bf separated} and \(U\) a {\bf separating neighborhood}. For \(s=2\), we will call \(Y\) {\bf doubly separated}. Note that \(Y\) is doubly separated if and only if \(\Lambda_{Y}\) is a uniformly discrete set.

We now recall the construction of \emph{intersection spaces} following~\cite{bjorklund2025int,AvBjCuI}. Here we only consider intersections of order \(2\). For a Borel \(G\)-space \(X\), consider two actions of \(G\) on \(X\times X\); one is on the first coordinate,
\[g\first\left(x,x^{\prime}\right)=\left(g.x,x^{\prime}\right),\]
and the other is the diagonal action,
\[g\diag\left(x,x^{\prime}\right)=\left(g.x,g.x^{\prime}\right).\]
Define the Borel set
\[Y^{\left[2\right]}\coloneqq G\diag\left(Y\times Y\right)=\left\{\left(h.y,h.y^{\prime}\right):\left(y,y^{\prime}\right)\in Y\times Y\right\}.\]
Considering \(X\times X\) as a Borel \(G\)-space with respect to the action on the first coordinate, we obtain that \(Y^{\left[2\right]}\subset X\times X\) is a cross section to this action: to see that \(G\first Y^{\left[2\right]}=X\times X\), for every \(\left(x,x^{\prime}\right)\in X\times X\) pick \(g\in Y_{x}, g^{\prime}\in Y_{x^{\prime}}\) and let \(y\coloneqq g.x\in Y, y^{\prime}\coloneqq g^{\prime}.x^{\prime}\in Y\), so that
\[\left(x,x^{\prime}\right)=\left(g^{-1}.y,g^{\prime -1}.y^{\prime}\right)=g^{-1}g^{\prime}\first g^{\prime -1}\diag\left(y,y^{\prime}\right)\in G\first Y^{\left[2\right]}.\]
The return times set of points to the cross section \(Y^{\left[2\right]}\) are countable and take the form
\[Y_{\left(x,x^{\prime}\right)}^{\left[2\right]}=Y_{x^{\prime}}^{-1}Y_{x},\quad\left(x,x^{\prime}\right)\in X\times X,\]
and the return times set of \(Y^{\left[2\right]}\) itself is
\[\Lambda_{Y^{\left[2\right]}}=\Lambda_{Y}^{2}.\]
Therefore, \(Y\subset X\) is doubly separated precisely when \(Y^{\left[2\right]}\subset X\times X\) is separated.

Following the terminology in~\cite{AvBjCuI}, a {\bf transverse \(G\)-space} is a triplet \(\left(X,\mu,Y\right)\), where \(\left(X,\mu\right)\) is a probability preserving \(G\)-space and \(Y\subset X\) is a locally integrable cross section, namely
\[\int_{X}\left|Y_{x}\cap K\right|d\mu\left(x\right)<\infty\text{ for every compact set }K\subset G.\]
It is a basic fact that separated cross sections are always locally integrable (see~\cite[Lemma 5.2]{AvBjCuI}).

To define the intersection covolume we use the transverse measure construction (see~\cite[\S4]{AvBjCuI} and the references therein). For a Borel function \(f:G\times Y\to\left[0,+\infty\right]\), define the {\bf periodization} of \(f\) to be the Borel function
\[f_{X}:X\longrightarrow\left[0,+\infty\right],\quad f_{X}\left(x\right)\coloneqq\sum\nolimits_{g\in Y_{x}}f\left(g^{-1},g.x\right).\]

\begin{defn}
Let \(\left(X,\mu,Y\right)\) be a transverse \(G\)-space. The associated {\bf transverse measure}, denoted \(\mu_{Y}\), is the unique Borel measure on \(Y\) satisfying the {\bf Campbell theorem}:
\[m_{G}\otimes\mu_{Y}\left(f\right)=\mu\left(f_{X}\right),\]
for every Borel function \(f:G\times Y\to\left[0,+\infty\right]\). The {\bf intensity} of \(Y\) is defined by
\[\iota_{\mu}\left(Y\right)\coloneqq\mu_{Y}\left(Y\right)\in\left(0,+\infty\right].\]
When \(Y\) is separated, then for every sufficiently small identity neighborhood \(U\) in \(G\) and every Borel set \(A\subseteq X\), it holds that
\[m_{G}\left(U\right)\cdot\mu_{Y}\left(A\right)=\mu\left(U.A\right).\]
(see~\cite[Proposition 4.8]{AvBjCuI}). In particular, in this case \(\mu_{Y}\) is a finite measure and thus \(\iota_{\mu}\left(Y\right)<+\infty\).
\end{defn}

The general definition of the intersection covolume applies to all transverse \(G\)-spaces (see~\cite[\S7]{AvBjCuI}). However, here we focus only on doubly separated cross sections, where~\cite[Proposition 6.2]{AvBjCuI} provides us with a different formula for the intersection covolume which will be extremely useful for us.

\begin{defn}\label{dfn:intvol}
Let \(\left(X,\mu,Y\right)\) be a transverse \(G\)-space. The associated {\bf intersection covolume} is
\[I_{\mu}\left(Y\right)\coloneqq\iota_{\mu\otimes\mu}\big(Y^{\left[2\right]}\big)=\left(\mu\otimes\mu\right)_{Y^{\left[2\right]}}\big(Y^{\left[2\right]}\big)\in\left(0,+\infty\right).\]
Here, \(\left(\mu\otimes\mu\right)_{Y^{\left[2\right]}}\) is the transverse measure associated with the transverse \(G\)-space \(\left(X\times X,\mu\otimes\mu,Y^{\left[2\right]}\right)\), with the action of \(G\) on the first coordinate.
\end{defn}

\section{Cut--and--project spaces: intensity and intersection covolume}

The following construction of transverse spaces goes back to Meyer~\cite{meyer1972algebraic} (see also~\cite{BaakeGrimm2013}). Here, we present this construction and compute its intensity and intersection covolume, proving Theorem~\ref{mthm:cap}.

A (abelian) {\bf cut--and--project scheme} is a triplet \(\left(G,H;\Gamma\right)\) consisting of lcsc abelian groups \(G\) and \(H\), and a lattice \(\Gamma<G\times H\), such that, for the projections \(\operatorname{proj}_{G}\) and \(\operatorname{proj}_{H}\) defined on \(G\times H\),
\[\operatorname{proj}_{G}\mid_{\Gamma}:\Gamma\longrightarrow G\text{ is injective and }\operatorname{proj}_{H}\left(\Gamma\right)<H\text{ is dense.}\]
With every cut--and--project scheme \(\left(G,H;\Gamma\right)\), there is associated a probability preserving \(G\)-space, called {\bf cut--and--project \(G\)-space}, defined as the compact abelian group
\[\left(\Xi,\xi\right)\coloneqq\left(\Gamma\backslash\left(G\times H\right),m_{\Gamma\backslash\left(G\times H\right)}\right),\]
where \(\xi\) is the Haar probability measure, with the Borel action of \(G\) on the first coordinate,
\[g_{o}.x=\Gamma\left(g,h\right)\left(g_{o}^{-1},e_{H}\right)=\Gamma\left(gg_{o}^{-1},h\right)\text{ for }x\coloneqq \Gamma\left(g,h\right)\in\Xi.\]
The measure \(\xi\) is the pushforward of \(m_{G}\otimes m_{H}\) along the quotient map \(G\times H\to\Xi\), normalized by \(\operatorname{covol}_{G\times H}\left(\Gamma\right)\), so that \(\xi\) is given by the Weil integration formula,
\begin{equation}\label{eq:xiformula}
\iint\nolimits_{G\times H}\varphi\left(g,h\right)dm_{G}\otimes m_{H}\left(g,h\right)=\operatorname{covol}_{G\times H}\left(\Gamma\right)\cdot\int_{\Xi}\sum\nolimits_{\left(\gamma,\eta\right)\in\Gamma}\varphi\left(\gamma g,\eta h\right)d\xi\left(\Gamma\left(g,h\right)\right),
\end{equation}
for every nonnegative Borel function \(\varphi:G\times H\to\mathbb{R}_{\geq 0}\). It was shown in~\cite[\S3.2]{BjHaPoI} that \(\Xi\) is uniquely ergodic already as a Borel \(G\)-space (and even among the stationary measures).

Recall that \(W\) is {\bf Jordan measurable} if \(m_{H}\left(\overline{W}\backslash W^{o}\right)=0\), or equivalently, if \(\mathbf{1}_{W}\) is Riemann integrable. Natural cross sections for cut--and--project \(G\)-spaces arise from the following sets in \(H\).

\begin{defn}\hypertarget{window}{}
A {\bf window} in an lcsc group \(H\) is a relatively compact Jordan measurable set \(W\subset H\),\footnote{By definition, \(W\subset H\) is Jordan measurable when \(m_{H}\left(\partial{W}\right)=0\), where \(\partial{W}\) denotes the topological boundary of \(W\).} such that \(0<m_{H}\left(W\right)+\infty\).
\end{defn}

Let \(\left(G,H;\Gamma\right)\) be a cut--and--project scheme and \(W\subset H\) a window. In the associated cut--and--project \(G\)-space \(\left(\Xi,\xi\right)\), define the Borel set
\[Y_{W}\coloneqq\left\{\Gamma\left(e_{G},w\right):w\in W\right\} \subset\Xi.\]
It can be verified that the return times sets of \(Y_{W}\) are of the form
\[\left(Y_{W}\right)_{x}=\mathrm{proj}_{G}\left(\Gamma\cap\left(G\times h_{o}W^{-1}\right)\right)^{-1}g_{o},\quad x=\Gamma\left(g_{o},h_{o}\right).\]
It is a abasic fact that \(\Lambda_{Y_{W}}\) is uniformly discrete (see~\cite[\S7.1]{BaakeGrimm2013},~\cite[\S2.3]{BjHaAL}) so that \(Y_{W}\) is a doubly separated cross section. Thus, \(\left(\Xi,\xi,Y_{W}\right)\) is a transverse \(G\)-space, called {\bf transverse cut--and--project \(G\)-space}.

For the rest of this section we will prove Theorem~\ref{mthm:cap} and Corollary~\ref{cor:cap}, starting with some lemmas.

\begin{lem}\label{lem:capform}
For every cut--and--project \(G\)-space \(\left(\Xi,\xi,Y_{W}\right)\) as above, the transverse measure \(\xi_{Y_{W}}\) satisfies
\[\mathrm{covol}_{G\times H}\left(\Gamma\right)\cdot\xi_{Y_{W}}\left(\psi\right)=\int_{W}\psi\left(\Gamma\left(e_{G},h\right)\right)dm_{H}\left(h\right),\]
for every Borel function \(\psi:\Xi\to\left[0,+\infty\right]\).
\end{lem}

\begin{rem}
The formula in Lemma~\ref{lem:capform} holds for every relatively compact measurable set \(W\subset H\).
\end{rem}

\begin{proof}[Proof of Lemma~\ref{lem:capform}]
Let \(f:G\times Y_{W}\to\left[0,+\infty\right]\) be an arbitrary Borel function. Recall for \(x=\Gamma\left(g_{o},h_{o}\right)\in\Xi\) the return times set of \(x\) to \(Y_{W}\) is
\[\left(Y_{W}\right)_{x}=\mathrm{proj}_{G}\left(\Gamma\cap\left(G\times h_{o}W^{-1}\right)\right)^{-1}g_{o},\]
so we can write the periodization of \(f\) in \(x\) as
\begin{align*}
f_{\Xi}\left(x\right)
&=\sum\nolimits_{g\in\left(Y_{W}\right)_{x}}f\left(g^{-1},g.x\right)\\
&=\sum\nolimits_{g\in\mathrm{proj}_{G}\left(\Gamma\cap\left(G\times h_{o}W^{-1}\right)\right)}f\left(g_{o}^{-1}g,\Gamma\left(g,h_{o}\right)\right)\\
&=\sum\nolimits_{\left(\gamma,\eta\right)\in\Gamma}f\left(g_{o}^{-1}\gamma,\Gamma\left(\gamma,h_{o}\right)\right)\mathbf{1}_{W}\left(h_{o}\eta^{-1}\right)\\
&=\sum\nolimits_{\left(\gamma,\eta\right)\in\Gamma}f\left(g_{o}^{-1}\gamma,\Gamma\left(e_{G},h_{o}\eta^{-1}\right)\right)\mathbf{1}_{W}\left(h_{o}\eta^{-1}\right)\\
&=\sum\nolimits_{\left(\gamma,\eta\right)\in\Gamma}f\left(g_{o}^{-1}\gamma^{-1},\Gamma\left(e_{G},h_{o}\eta\right)\right)\mathbf{1}_{W}\left(h_{o}\eta\right)\\
&\coloneqq\sum\nolimits_{\left(\gamma,\eta\right)\in\Gamma}\varphi_{f}\left(\gamma g_{o},\eta h_{o}\right),
\end{align*}
where in the third equality we used that \(\mathrm{proj}_{G}\mid_{\Gamma}\) is injective, and in the last equality we used the function
\[\varphi_{f}:G\times H\to\left[0,+\infty\right],\quad\varphi_{f}\left(g_{o},h_{o}\right)\coloneqq f\left(g_{o}^{-1},\Gamma\left(e_{G},h_{o}\right)\right)\mathbf{1}_{W}\left(h_{o}\right).\]
Using the general definition of transverse measures and the formula \eqref{eq:xiformula}, we obtain the formula
\[\mathrm{covol}_{G\times H}\left(\Gamma\right)\cdot m_{G}\otimes\xi_{Y_{W}}\left(f\right)	=\mathrm{covol}_{G\times H}\left(\Gamma\right)\cdot\xi\left(f_{\Xi}\right)=\iint\nolimits_{G\times W}f\left(g,\Gamma\left(e_{G},h\right)\right)dm_{G}\otimes m_{H}\left(g,h\right).\]
Taking \(f=w\otimes\psi\) for some \(w:G\to\left[0,+\infty\right]\) with \(m_{G}\left(w\right)=1\), the desired formula follows.
\end{proof}

Since \(\Xi=\Gamma\backslash\left(G\times H\right)\) is a compact abelian group, it acts on itself by \(\Xi\)-translations. Recall also that \(G\) acts on \(\Xi\) by \(\Xi\)-translations: each \(g\in G\) acts on \(\Xi\) as the translation by \(\Gamma\left(g,e_{H}\right)\). Therefore, both \(G\) and \(\Xi\) act diagonally on \(\Xi\times\Xi\). In view of this, we have the following identity:

\begin{lem}\label{lem:capident}
For every cut--and--project \(G\)-space \(\left(\Xi,\xi,Y_{W}\right)\) as above with open \hyperlink{window}{window} \(W\),
\[G\diag\left(Y_{W}\times Y_{W}\right)=\Xi\diag\left(Y_{W}\times Y_{W}\right).\]
\end{lem}

\begin{proof}[Proof of Lemma~\ref{lem:capident}]
The inclusion \(G\diag\left(Y_{W}\times Y_{W}\right)\subseteq\Xi\diag\left(Y_{W}\times Y_{W}\right)\) is clear, and we prove the opposite inclusion. Consider arbitrary elements \(x=\Gamma\left(g_{o},h_{o}\right)\) and \(\left(y_{1},y_{2}\right)=\left(\Gamma\left(e_{G},w_{1}\right),\Gamma\left(e_{G},w_{2}\right)\right)\in Y_{W}\times Y_{W}\). Denoting \(h_{1}\coloneqq w_{1}h_{o}\) and \(h_{2}\coloneqq w_{2}h_{o}\), we have
\[x\diag\left(y_{1},y_{2}\right)=\left(\Gamma\left(g_{o},h_{1}\right),\Gamma\left(g_{o},h_{2}\right)\right)\in\Xi\diag\left(Y_{W}\times Y_{W}\right).\]
Define the open set
\[U\coloneqq h_{1}W^{-1}\cap h_{2}W^{-1}\subset H,\]
and note that \(U\neq\emptyset\) since \( h_{1}w_{1}^{-1}=h_{2}w_{2}^{-1}\in h_{1}W^{-1}\cap h_{2}W^{-1}\). Using that \(\mathrm{proj}_{H}\left(\Gamma\right)\) is dense in \(H\), there is \(\gamma=\left(\gamma_{G},\gamma_{H}\right)\in\Gamma\) with \(\gamma_{H}\in U\), hence
\[h_{1}\gamma_{H}^{-1}\in W\text{ and }h_{2}\gamma_{H}^{-1}\in W.\]
We finally obtain that
\begin{align*}
x\diag\left(y_{1},y_{2}\right)
&=\left(\Gamma\left(g_{o}\gamma_{G}^{-1},h_{1}\gamma_{H}^{-1}\right),\Gamma\left(g_{o}\gamma_{G}^{-1},h_{2}\gamma_{H}^{-1}\right)\right)\\
&=\gamma_{G}g_{o}^{-1}\diag\left(\Gamma\left(e_{G},h_{1}\gamma_{H}^{-1}\right),\Gamma\left(e_{G},h_{2}\gamma_{H}^{-1}\right)\right)\in G\diag\left(Y_{W}\times Y_{W}\right).\qedhere
\end{align*}
\end{proof}

\begin{lem}\label{lem:capapp}
Let \(W\subset H\) be a \hyperlink{window}{window} such that \(W-W\) is Jordan measurable. Then for every \(\epsilon>0\), there is an open set \(W^{\epsilon}\) such that \(W\subseteq W^{\epsilon}\) and \(m_{H}\left(\left(W^{\epsilon}-W^{\epsilon}\right)\backslash\left(W-W\right)\right)<\epsilon\).
\end{lem}

\begin{proof}[Proof of Lemma~\ref{lem:capapp}]
Put \(D=W-W\) and note that \(D\) is relatively compact since \(W\) is. Let \(\epsilon>0\) be arbitrary. By the outer regularity of \(m_{H}\), there is an open set \(V^{\epsilon}\) such that \(\overline{D}\subseteq V^{\epsilon}\) and \(m_{H}\left(V^{\epsilon}\backslash\overline{D}\right)<\epsilon\). For every \(d\in\overline{D}\), pick a symmetric identity neighborhood \(U_{d}\) such that \(d+U_{d}+U_{d}+U_{d}\subseteq V^{\epsilon}\). Since \(\overline{D}\) is compact, there are finitely many \(d_{1},\dotsc,d_{n}\in\overline{D}\) such that \(d_{i}+U_{d_{i}}+U_{d_{i}}+U_{d_{i}}\subseteq V^{\epsilon}\) and \(\overline{D}\subseteq\bigcup\nolimits_{i=1}^{n}\left(d_{i}+U_{d_{i}}\right)\). Let then \(U^{\epsilon}\coloneqq\bigcap\nolimits_{i=1}^{n}U_{d_{i}}\), which is a symmetric identity neighborhood, and define \(W^{\epsilon}\coloneqq W+U^{\epsilon}\). Then \(W^{\epsilon}\) is an open set containing \(W\), and it satisfies
\[W^{\epsilon}-W^{\epsilon}=D+U^{\epsilon}+U^{\epsilon}\subseteq\bigcup\nolimits_{i=1}^{n}\left(d_{i}+U_{d_{i}}+U_{d_{i}}+U_{d_{i}}\right)\subseteq V^{\epsilon},\]
concluding that
\[m_{H}\left(\left(W^{\epsilon}-W^{\epsilon}\right)\backslash\left(W-W\right)\right)\leq m_{H}\left(V^{\epsilon}\backslash D\right)=m_{H}\big(V^{\epsilon}\backslash\overline{D}\big)<\epsilon,\]
where we used Jordan measurability of \(D\), by which \(m_{H}\left(\left(V^{\epsilon}\backslash D\right)\backslash\left(V^{\epsilon}\backslash\overline{D}\right)\right)\leq m_{H}\left(\overline{D}\backslash D\right)=0\).
\end{proof}

\begin{proof}[Proof of Theorem~\ref{mthm:cap}]
For part (1), we apply Lemma~\ref{lem:capform} with \(\psi\equiv 1\), so that
\[\iota_{\xi}\left(Y_{W}\right)=\xi_{Y_{W}}\left(Y_{W}\right)=\mathrm{covol}_{G\times H}\left(\Gamma\right)^{-1}\cdot m_{H}\left(W\right).\]
We then prove part (2). Let us assume first that \(W\) is open. Consider the transverse \(G\)-space
\[\big(\Xi\times\Xi,\xi\otimes\xi,Y_{W}^{\left[2\right]}\big),\]
where \(G\) acts on the first coordinate, and define the map
\[\delta:\big(\Xi\times\Xi,\xi\otimes\xi,Y_{W}^{\left[2\right]}\big)\longrightarrow\left(\Xi,\xi,Y_{W-W}\right),\quad \delta\left(x,x^{\prime}\right)=x-x^{\prime}.\]
Note that \(\delta\) surjective and \(G\)-equivariant, and therefore, by the uniqueness of the Haar measure, it is measure preserving. We then have
\begin{equation}\label{eq:delta}
\delta^{-1}\left(Y_{W-W}\right)=\delta^{-1}\left(Y_{W}-Y_{W}\right)=\Xi\diag\left(Y_{W}\times Y_{W}\right)=G\diag\left(Y_{W}\times Y_{W}\right)=Y_{W}^{\left[2\right]},
\end{equation}
where the first two equalities are by commutativity, and the third equality is by Lemma~\ref{lem:capident} and since \(W\) is open. Therefore, \(\delta\) forms a transverse \(G\)-factor, and we conclude that
\[I_{\xi}\left(Y_{W}\right)=\iota_{\xi\otimes\xi}\big(Y_{W}^{\left[2\right]}\big)=\iota_{\xi}\left(Y_{W-W}\right)=\mathrm{covol}_{G\times H}\left(\Gamma\right)^{-1}\cdot m_{H}\left(W-W\right),\]
where the first equality is the definition of intersection covolume; the second equality is by~\cite[Proposition 7.5]{AvBjCuI} in light of \eqref{eq:delta}; and, the last equality is by part (1) for the window \(W-W\) (recall \(W\) is open). Of course, \(W^{o}-W^{o}=W-W\) and this completes the proof of part (2) for the case that \(W\) is open.

Assume now that \(W\) is a window, not necessarily open, but \(W-W\) is Jordan measurable. For every \(\epsilon>0\) pick \(W^{\epsilon}\) as in Lemma~\ref{lem:capapp}, and then \(Y_{W^{o}}\subseteq Y_{W}\subseteq Y_{W^{\epsilon}}\). By the monotonicity of intersection covolume~\cite[Proposition 7.1]{AvBjCuI} we obtain
\[\mathrm{covol}_{G\times H}\left(\Gamma\right)^{-1}\cdot m_{H}\left(W^{o}-W^{o}\right)=I_{\xi}\left(Y_{W^{o}}\right)\leq I_{\xi}\left(Y_{W}\right)\leq I_{\mu}\left(Y_{W^{\epsilon}}\right)=\mathrm{covol}_{G\times H}\left(\Gamma\right)^{-1}\cdot m_{H}\left(W^{\epsilon}-W^{\epsilon}\right),\]
where the two equalities are by the first part of the proof, as \(W^{o}\) are \(W^{\epsilon}\) are open. Finally, since \(m_{H}\left(W^{\epsilon}-W^{\epsilon}\right)<m_{H}\left(W-W\right)+\epsilon\) and \(\epsilon\) is arbitrary, the inequalities of part (2) readily follow.
\end{proof}

\begin{proof}[Proof of Corollary~\ref{cor:cap}]
Using the Brunn--Minkowski inequality we obtain
\begin{align*}
I_{\xi}\left(Y_{W}\right)\geq I_{\xi}\left(Y_{W^{o}}\right)
&=\mathrm{covol}_{G\times \mathbb{R}^{d}}\left(\Gamma\right)^{-1}\cdot m_{\mathbb{R}^{d}}\left(W^{o}-W^{o}\right)\\
&\geq 2^{d}\cdot\mathrm{covol}_{G\times \mathbb{R}^{d}}\left(\Gamma\right)^{-1}\cdot m_{\mathbb{R}^{d}}\left(W^{o}\right)=2^{d}\cdot\iota_{\xi}\left(Y_{W^{o}}\right)=2^{d}\cdot\iota_{\xi}\left(Y_{W}\right), 
\end{align*}
where the first inequality is by \(Y_{W}\supseteq Y_{W^{o}}\) and the monotonicity of intersection covolume~\cite[Proposition 7.1]{AvBjCuI}; the next equality is by part (1) of Theorem~\ref{mthm:cap}; the inequality in the second line is the Brunn--Minkowski inequality~\cite[\S4]{gardner2002brunn}; and, the last equality is by part (2) of Theorem~\ref{mthm:cap} and since \(W^{o}\) is open.

Now for (1), if \(W=\mathrm{conv}\left(W\right)\) modulo \(m_{\mathbb{R}^{d}}\) then \(m_{\mathbb{R}^{d}}\left(W-W\right)=2^{d}\cdot m_{\mathbb{R}^{d}}\left(W\right)\) by the extremes of the Brunn--Minkowski inequality.\footnote{See~\cite[\S4]{gardner2002brunn}, and the discussion of the equality cases following the proof of~\cite[Theorem 4.1]{gardner2002brunn}.} Therefore,
\begin{align*}
2^{d}\cdot\iota_{\xi}\left(Y_{W}\right)\leq I_{\xi}\left(Y_{W}\right)
&\leq\mathrm{covol}_{G\times\mathbb{R}^{d}}\left(\Gamma\right)^{-1}\cdot m_{\mathbb{R}^{d}}\left(W-W\right)\\
&=\mathrm{covol}_{G\times\mathbb{R}^{d}}\left(\Gamma\right)^{-1}\cdot2^{d}\cdot m_{\mathbb{R}^{d}}\left(W\right)=2^{d}\cdot\iota_{\xi}\left(Y_{W}\right),
\end{align*}
implying that \(I_{\xi}\left(Y_{W}\right)=2^{d}\cdot\iota_{\xi}\left(Y_{W}\right)\).

As for (2), note that from the computation at the beginning of this proof, if \(I_{\xi}\left(Y_{W}\right)=2^{d}\cdot\iota_{\xi}\left(Y_{W}\right)\) then \(m_{\mathbb{R}^{d}}\left(W^{o}-W^{o}\right)=2^{d}\cdot m_{\mathbb{R}^{d}}\left(W^{o}\right)\). Then by the extremes of the Brunn--Minkowski inequality we obtain \(W^{o}=\mathrm{conv}\left(W^{o}\right)\) modulo \(m_{\mathbb{R}^{d}}\), and since \(W\) is Jordan measurable, \(W=\mathrm{conv}\left(W^{o}\right)\) modulo \(m_{\mathbb{R}^{d}}\).
\end{proof}

\begin{rem}\label{rem:sumset}
In~\cite[\S7]{AvBjCuI} we defined higher order intersection covolume \(I_{\mu}^{r}\left(Y\right)\) for \(r\in\mathbb{N}\), and in the setting of Theorem~\ref{mthm:cap} one may ask for a formula for \(I_{\xi}^{r}\left(Y_{W}\right)\) with \(r\in\mathbb{N}\). In~\cite{shkredov2015}, Shkredov introduced \emph{higher order sumsets} of order \(r\in\mathbb{N}\) of a window \(W\subset H\), which are defined by
\[W^{\left[r\right]}\coloneqq\big\{ \left(h_{1},\dotsc,h_{r-1}\right)\in H^{\otimes \left(r-1\right)}:W\cap\left(W-h_{1}\right)\cap\dotsm\cap\left(W-h_{r-1}\right)\neq\emptyset\big\}\subset H^{\otimes\left(r-1\right)}.\]
Note that \(W^{\left[2\right]}=W-W\subset H\). It seems reasonable to ask whether the following bounds hold for \(r\in\mathbb{N}\):
\[\mathrm{covol}_{G\times H}\left(\Gamma\right)^{-1}\cdot m_{H}^{\otimes\left(r-1\right)}\big(\left(W^{o}\right)^{\left[r\right]}\big)\leq I_{\xi}^{r}\left(Y_{W}\right)\leq\mathrm{covol}_{G\times H}\left(\Gamma\right)^{-1}\cdot m_{H}^{\otimes\left(r-1\right)}\big(W^{\left[r\right]}\big).\]
\end{rem}

\section{The main inequality}

Here we aim to prove Theorem~\ref{mthm:mineq}. We carry out the analysis using the Kronecker factor of the system and the corresponding Kronecker compactification. While this is a classical construction (see~\cite[\S6.4]{EinsiedlerWard}), we include a fairly detailed treatment of the theory in the generality of lcsc abelian groups.

\subsection{The Kronecker factor}

Throughout this section, \(G\) stands for an lcsc abelian group and \(m_{G}\) stands for a Haar measure of \(G\), and \(\left(X,\mu\right)\) an ergodic probability preserving \(G\)-space. Recall that a Borel function \(f:X\to\mathbb{C}\) is a {\bf \(G\)-eigenfunction}, if there is a character \(\chi\in\widehat{G}\) such that
\[f\left(g.x\right)=\chi\left(g\right)f\left(x\right)\text{ for every }g\in G\text{ and }\mu\text{-a.e. }x\in X.\]
We refer to \(\chi\) as the {\bf eigencharacter} of \(f\). For \(\chi\in\widehat{G}\), denote by \(E_{\chi}\) the closed linear space of of \(G\)-eigenfunctions whose eigencharacter is \(\chi\). By ergodicity, the modulus of every \(G\)-eigenfunction is constant \(\mu\)-a.e. Also, every \(E_{\chi}\) is at most \(1\)-dimensional for all \(\chi\in\widehat{G}\), and \(E_{\chi},E_{\chi^{\prime}}\) are orthogonal subspaces of \(L^{2}\left(\mu\right)\) for all two different  \(\chi,\chi^{\prime}\in\widehat{G}\). In particular, since \(L^{2}\left(\mu\right)\) is separable, there are at most countably many \(\chi\in\widehat{G}\) with \(E_{\chi}\neq\{0\}\). We denote this countable set by
\[\mathcal{E}\coloneqq \{\chi\in\widehat{G}:E_{\chi}\neq\{0\}\}.\]
By their definition, \(G\)-eigenfunctions satisfy a property that holds for every \(g\in G\) on a \(\mu\)-conull set. It is useful to have a pointwise-defined version of this property for all \(G\)-eigenfunctions at once:

\begin{prop}\label{prop:pointeig}
There exist a \(\mu\)-conull \(G\)-invariant Borel set \(X_{o}\subseteq X\), and a collection \(\mathbf{E}\) of pointwise-defined square-integrable functions on \(\left(X,\mu\right)\), such that:
\begin{enumerate}
    \item \(\mathbf{E}\) contains a constant-multiple of each \(G\)-eigenfunction.
    \item For every \(\chi\in\widehat{G}\) and \(f\in\mathbf{E}\cap E_{\chi}\),
    \[f\left(g.x\right)=\chi\left(g\right)f\left(x\right)\text{ for every }\left(g,x\right)\in G\times X_{o}.\]
\end{enumerate}
\end{prop}

\begin{proof}
For each \(\chi\in\mathcal{E}\) choose some \(f_{\chi}\in E_{\chi}\) normalized such that \(f_{\chi}:X\to\mathbb{T}\). Since \(E_{\chi}\) is \(1\)-dimensional, this determines \(f_{\chi}\) as an element of \(E_{\chi}\) up to a \(\mu\)-null set. Note that \(f_{\chi}\) is a \(G\)-factor from \(\left(X,\mu\right)\) to \(\left(\mathbb{T},m_{\mathbb{T}}\right)\), where the latter becomes a probability preserving \(G\)-space with the action \(g.t=\chi\left(g\right)t\). Therefore, it is well-known (see~\cite[Proposition 5.7]{AvBjCuI}) that there are a \(\mu\)-conull \(G\)-invariant Borel set \(X_{\chi}\subseteq X\) as well as a Borel function \(f_{\chi}^{\prime}:X\to\mathbb{T}\) such that
\[f_{\chi}^{\prime}=f_{\chi}\mid_{X_{\chi}}\,\operatorname{mod}\,\mu\,\text{ and }\,f_{\chi}^{\prime}\left(g.x\right)=\chi\left(g\right)f_{\chi}^{\prime}\left(x\right)\text{ for all }\left(g,x\right)\in G\times X_{\chi}.\]
Then the assertion of the proposition holds with
\[X_{o}\coloneqq\bigcap\nolimits_{\chi\in\mathcal{E}}X_{\chi}\quad\text{and}\quad\mathbf{E}\coloneqq \bigcup\nolimits_{\chi\in\mathcal{E}}\mathbb{C}\cdot f_{\chi}^{\prime}.\qedhere\]
\end{proof}

The {\bf Kronecker factor} of \(\left(X,\mu\right)\), denoted \(\mathcal{K}_{G}\), is defined as the smallest Borel sub-\(\sigma\)-algebra with respect to which all \(G\)-eigenfunctions of \(\left(X,\mu\right)\) are measurable. In light of Proposition~\ref{prop:pointeig}, the Kronecker factor \(\mathcal{K}_{G}\) can be concretely realized as the smallest \(\sigma\)-algebra with respect to which all the pointwise-defined functions in \(\mathbf{E}\) are measurable. Recall that a {\bf compactification} of an lcsc group \(G\) is a pair \(\left(K,\tau\right)\), consisting of a compact abelian group \(K\) and a continuous homomorphism \(\tau:G\to K\), such that \(\tau\left(G\right)\) is dense in \(K\).

\begin{thm}\label{thm:krongrp}
There exists a compactification \(\left(K,\tau\right)\) of \(G\), called the {\bf Kronecker compactification}, a \(\mu\)-conull \(G\)-invariant set \(X_{o}\subseteq X\), and a Borel map \(\kappa:X_{o}\to K\), such that:
\begin{enumerate}
    \item \(\kappa_{\ast}\mu=m_{K}\) (the Haar measure of \(K\)).
    \item \(\kappa\left(g.x\right)=\tau\left(g\right)+\kappa\left(x\right)\) for all \(\left(g,x\right)\in G \times X_{o}\). 
    \item \(\kappa^{-1}\left(\mathcal{B}_{K}\right)=\mathcal{K}_{G}\mid_{X_{o}}\) modulo \(\mu\), where \(\mathcal{B}_{K}\) is the Borel \(\sigma\)-algebra of \(K\).
\end{enumerate}
Thus, \(\kappa:\left(X,\mu\right)\to\left(K,m_{K}\right)\) is a \(G\)-factor map, where the latter is the the probability preserving \(G\)-space with the action \(g.k=\tau\left(g\right)+k\).
\end{thm}

\begin{proof}
Pick \(\mathbf{E}\) and \(X_{o}\) as in Proposition~\ref{prop:pointeig}, so that for each \(\chi\in\mathcal{E}\) we may fix \(f_{\chi}\in\mathbf{E}\cap E_{\chi}\) normalized such that \(f_{\chi}:X_{o}\to\mathbb{T}\). Consider the compact abelian group \(A\coloneqq \mathbb{T}^{\mathcal{E}}\), and for every \(a\in A\), let \(R_{a}:A\to A\) be the \(a\)-rotation on \(A\). Look at the continuous homomorphism
\[\tau:G\longrightarrow A,\quad\tau\left(g\right)\coloneqq \left(\chi\left(g\right):\chi\in\mathcal{E}\right),\]
and define the Kronecker compactification \(\left(K,\tau\right)\) to be
\[K\coloneqq \overline{\tau\left(G\right)}\leq A.\]
Let us now construct the \(G\)-factor \(\kappa_{o}\). We start by defining the map
\[\kappa_{o}:X_{o}\longrightarrow A,\quad\kappa_{o}\left(x\right)=\left(f_{\chi}\left(x\right):\chi\in\mathcal{E}\right),\]
and note that
\[\kappa_{o}\left(g.x\right)=R_{\tau\left(g\right)}\left(\kappa_{o}\left(x\right)\right)\text{ for all }\left(g,x\right)\in G\times X_{o}.\]
Put the pushforward measure \(\nu\coloneqq \kappa_{o\ast}\mu\) on \(A\), and let \(S\coloneqq \operatorname{supp}\left(\nu\right)\) be the support of \(\nu\) in \(A\), obtained by removing from \(A\) the largest \(\nu\)-null open set. Since rotations are continuous and \(\nu\) is invariant to the rotations \(R_{\tau\left(G\right)}\), then \(S\) becomes an \(R_{\tau\left(G\right)}\)-invariant set. Then \(\left(S,\nu\right)\) with the action \(g.s=R_{\tau\left(g\right)}\left(s\right)\) becomes a probability preserving \(G\)-space, and \(\kappa_{o}:\left(X_{o},\mu_{o}\right)\to\left(S,\nu\right)\) is a \(G\)-factor map. Since \(\left(X_{o},\mu_{o}\right)\) is ergodic, \(\left(S,\nu\right)\) is ergodic as well.

By the continuity of rotations, the action of \(\tau\left(G\right)\) extends to an action of \(K\), and thus \(\left(S,\nu\right)\) becomes an ergodic probability preserving \(K\)-space. Since \(K\) is compact and acts continuously, the \(K\)-orbit \(R_{K}\left(s\right)\) is compact in \(S\) for every \(s\in S\). We claim that \(R_{K}\left(s\right)=S\) for \(\nu\)-a.e. \(s\in S\). To see this, let \(\mathrm{d}\) be a compatible invariant metric on \(A\), and for every \(s\in S\) we have the bounded \(K\)-invariant Borel function \(\mathrm{d}_{s}:S\to\mathbb{R}\), \(\mathrm{d}_{s}\left(s^{\prime}\right)\coloneqq\mathrm{d}\left(s^{\prime},R_{K}\left(s\right)\right)\), and by the ergodicity it is constant for \(\nu\)-a.e. \(s\in S\). At the same time, by the invariance of \(\mathrm{d}\) we have \(\mathrm{d}_{s}\circ R_{k}=\mathrm{d}_{k.s}\) for all \(s\in S\) and \(k\in K\). Then again by the ergodicity, on a \(\nu\)-conull set of \(s\in S\), the constants \(\mathrm{d}_{s}\) are all the same, and hence \(\mathrm{d}_{s}\equiv 0\), showing that \(R_{K}\left(s\right)=S\).

We thus pick some \(s_{o}\in S\) such that \(R_{K}\left(s_{o}\right)=S\), and rotate \(\kappa_{o}\) so to take values in \(K\), via
\[\kappa:X_{o}\to K,\quad \kappa\left(x\right)\coloneqq R_{-s_{o}}\left(\kappa_{o}\left(x\right)\right).\]
At the same time, \(\left(R_{-s_{o}}\right)_{\ast}\nu\) is a Borel \(K\)-invariant measure on \(K\), and therefore, using the uniqueness of the Haar measure, \(\left(R_{-s_{o}}\right)_{\ast}\nu=m_{K}\), showing that \(\nu=m_{K}\) and \(S=K\). One can routinely verify that Properties (1) and (2) hold, making \(\kappa:\left(X_{o},\mu\right)\to\left(K,m_{K}\right)\) the desired \(G\)-factor map. Finally, Property (3) follows since all \(G\)-eigenfunctions are, modulo \(\mu\), constant multiples of the elements of \(\mathbf{E}\).
\end{proof}

The following simple lemma allows in some situations to make a reduction to the Kronecker factor of a countable dense subgroup.

\begin{lem}\label{lem:gamkro}
Let \(G\) be an lcsc abelian group and \(\Gamma<G\) a countable dense subgroup. Denote by \(\Gamma_{\mathrm{d}}\) the group \(\Gamma\) equipped with the discrete topology. Then for every ergodic probability preserving \(G\)-space \(\left(X,\mu\right)\), all \(\Gamma_{\mathrm{d}}\)-eigenfunctions\footnote{This means with respect to the subaction of \(\Gamma\) on \(\left(X,\mu\right)\), and with eigencharacters from \(\widehat{\Gamma_{\mathrm{d}}}\).} are also \(G\)-eigenfunctions. Consequently, \[\mathcal{K}_{G}=\mathcal{K}_{\Gamma_{\mathrm{d}}}.\]
\end{lem}

\begin{proof}
Let \(f:X\to\mathbb{C}\) be a \(\Gamma_{\mathrm{d}}\)-eigenfunction with an eigencharacter \(\chi\in\widehat{\Gamma_{\mathrm{d}}}\). Then for all \(\gamma,\gamma^{\prime}\in\Gamma\),
\begin{equation}\label{eq:Gamma}
\left\Vert f\left(\gamma.\left(\cdot\right)\right)-f\left(\gamma^{\prime}.\left(\cdot\right)\right)\right\Vert_{L^{2}\left(\mu\right)}=\left|\chi\left(\gamma\right)-\chi\left(\gamma^{\prime}\right)\right|\left\Vert f\right\Vert_{L^{2}\left(\mu\right)} .
\end{equation}
Since the Koopman representation of \(G\) on \(L^{2}\left(\mu\right)\) is continuous, its restriction to \(\Gamma\) (now with the subgroup topology) is sequentially continuous. It then follows from \eqref{eq:Gamma} that \(\chi:\Gamma\to\mathbb{T}\) is sequentially continuous in the subgroup topology. Since \(\Gamma\) is dense in \(G\), there is a unique continuous extension of \(\chi\) to \(G\), denoted \(\widetilde{\chi}:G\to\mathbb{T}\), and one can verify that \(\widetilde{\chi}\in\widehat{G}\). Finally, since the continuous map
\[G\longrightarrow\mathbb{R},\quad g\mapsto\left\Vert \widetilde{\chi}\left(g\right)f-f\left(g.\left(\cdot\right)\right)\right\Vert_{L^{2}\left(\mu\right)},\]
vanishes on \(\Gamma\), it is the zero map. Thus, \(f\) is a \(G\)-eigenfunction with eigencharacter \(\widetilde{\chi}\in\widehat{G}\).
\end{proof}

We need the following two technical tools. The first can be found in~\cite[Proposition 4.1]{BjFi2019}, whose proof holds in the following general setting.

\begin{lem}\label{lem:ijprime}
Let \(K\) be a compact abelian group with Haar measure \(m_{K}\), and \(\Gamma<K\) be a countable dense subgroup. For all measurable sets \(I,J\subseteq K\) there are \(m_{K}\)-conull Borel subsets \(I_{o}\subseteq I,J_{o}\subseteq J\) with
\[\Gamma.\left(I_{o}\times J_{o}\right)=K.\left(I_{o}\times J_{o}\right)\text{ modulo }m_{K}\otimes m_{K}.\]
\end{lem}

The following notions can be put for a general \(G\)-factor, but we focus on the Kronecker factor.

\begin{defn}
Let \(\left(X,\mu\right)\) be an ergodic probability preserving \(G\)-space. The {\bf Kronecker shadow} of a Borel set \(A\) is defined up to a \(\mu\)-null set by
\[A_{\mathcal{K}_{G}}\coloneqq \{x\in X:\mu\left(\mathbf{1}_{A}\mid\mathcal{K}_{G}\right)\left(x\right)>0\},\]
where \(\mu\left(\mathbf{1}_{A}\mid\mathcal{K}_{G}\right)\) is a pointwise-defined version of the conditional expectation associated with \(\mathcal{K}_{G}\).
\end{defn}

The following lemma follows from~\cite[Lemma 2.4]{BjFi2019}, using that every Borel set in \(X\times X\) invariant to the diagonal action of \(G\) belongs to \(\mathcal{K}_{G}\otimes\mathcal{K}_{G}\) modulo \(\mu\otimes\mu\) (see~\cite[Lemma 2.3]{BjFi2019}).

\begin{lem}\label{lem:shad}
Let \(\Gamma\) be a countable discrete abelian group and \(\left(X,\mu\right)\) a probability preserving \(\Gamma\)-space. Then for all Borel sets \(A\) and \(B\) in \(X\) we have:
\begin{enumerate}
    \item \(\left(A\times B\right)_{\mathcal{K}_{\Gamma}}=A_{\mathcal{K}_{\Gamma}}\times B_{\mathcal{K}_{\Gamma}}\) modulo \(\mu\).
    \item \(\Gamma\diag\left(A\times B\right)=\Gamma\diag\left(A_{\mathcal{K}_{\Gamma}}\times B_{\mathcal{K}_{\Gamma}}\right)\) modulo \(\mu\).
\end{enumerate}
\end{lem}

\subsection{The main inequality, revisited}

Let us start with a basic observation. Recall that for an lcsc abelian group \(G\) and an ergodic probability preserving \(G\)-space, triviality of the Kronecker factor is equivalent to weak mixing (see~\cite[Theorem 2.36]{EinsiedlerWard}).

\begin{prop}
Let \(G\) be a non-discrete lcsc abelian group and \(\left(X,\mu,Y\right)\) a probability preserving \(G\)-space. If \(\left(X,\mu\right)\) admits a doubly separated cross section then its Kronecker factor \(\mathcal{K}_{G}\) is nontrivial.
\end{prop}

\begin{proof}
On one hand, the double separation of \(Y\) implies that \(Y^{\left[2\right]}\) is a separated cross section for \(\left(X^{\otimes2},\mu^{\otimes2}\right)\) in the action of \(G\) on the first coordinate, and therefore
\[I_{\mu}\left(Y\right)=\big(\mu^{\otimes2}\big)_{Y^{\left[2\right]}}\big(Y^{\left[2\right]}\big)=\iota_{\mu^{\otimes2}}\big(Y^{\left[2\right]}\big)<+\infty.\]
(see~\cite[Theorem 7.7]{AvBjCuI}). On the other hand, pick any identity neighborhood \(e_{G}\in V\subset G\) which is sufficiently small to be a separating neighborhood for \(Y^{\left[2\right]}\), and then, for every further identity neighborhood \(e_{G}\in U\subseteq V\) satisfying \(U-U\subseteq V\), we have
\begin{equation}\label{eq:vid}
\begin{aligned}
m_{G}\left(V\right)\cdot I_{\mu}\left(Y\right)	
&=m_{G}\left(V\right)\cdot\left(\mu^{\otimes 2}\right)_{Y^{\left[2\right]}}\big(Y^{\left[2\right]}\big)=\mu^{\otimes 2}\big(V\first Y^{\left[2\right]}\big)\\
&\geq\mu^{\otimes 2}\big(\left(U-U\right)\first Y^{\left[2\right]}\big)=\mu^{\otimes 2}\left(\left(U-U\right)\first G\diag\left(Y\times Y\right)\right)\\
&=\mu^{\otimes 2}\left(G\diag\left(U.Y\times U.Y\right)\right),
\end{aligned}
\end{equation}
where the first and second equalities are by Definition~\ref{dfn:intvol} of the intersection covolume and the transverse measure (using that \(V\) is a separating neighborhood for \(Y^{\left[2\right]}\)), and the last equality is by commutativity. If \(\left(X,\mu\right)\) was weakly mixing, then since \(G\diag\left(U.Y\times U.Y\right)\) is a \(G\)-invariant set of \(\mu^{\otimes 2}\)-positive measure (already \(U.Y\times U.Y\) has \(\mu^{\otimes 2}\)-positive measure) in \(\left(X^{\otimes 2},\mu^{\otimes 2}\right)\), it is necessarily \(\mu^{\otimes 2}\)-conull. Then from \eqref{eq:vid} it follows that \(m_{G}\left(V\right)\cdot I_{\mu}\left(Y\right)\geq1\). Since \(G\) is non-discrete, \(V\) can be chosen with an arbitrarily small Haar measure, hence \(I_{\mu}\left(Y\right)=+\infty\), which is a contradiction.
\end{proof}

We will now refine Theorem~\ref{mthm:mineq} using the Kronecker factor of \(\left(X,\mu\right)\) as in Theorem~\ref{thm:krongrp}.

\begin{thm}\label{thm:mineq}
Let \(G\) be a \hyperlink{classQ}{class \(\mathcal{Q}\)} group, \(\left(X,\mu,Y\right)\) an ergodic transverse \(G\)-space where \(Y\) is doubly separated. Let \(\mathcal{K}_{G}\) be the associated Kronecker factor and \(\left(K,\tau\right)\) the corresponding Kronecker compactification. Then for every sufficiently small compact open subgroup \(V<G\) with Kronecker shadow \(S\coloneqq\left(V.Y\right)_{\mathcal{K}_{G}}\), there is a \(\mu\)-conull set \(S_{o}\subseteq S\) such that the following inequalities hold:
\begin{equation}\label{eq:refineq}
m_{G}\left(V\right)\cdot I_{\mu}\left(Y\right)\geq m_{K}\left(S_{o}-S_{o}\right)\geq2\cdot m_{K}\left(S_{o}\right)\geq 2\cdot m_{G}\left(V\right)\cdot\iota_{\mu}\left(Y\right).
\end{equation}
\end{thm}

As will be shown in the proof, Theorem~\ref{thm:mineq} holds for every \(V\) that is a separating neighborhood of \(Y^{\left[2\right]}\), that is, \(\Lambda_{Y}^{2}\cap V=\{e_{G}\}\).

Recall that since \(G\) is of class \(\mathcal{Q}\), by Proposition~\ref{prop:classqgrp} the Kronecker compactification \(\left(K,\tau\right)\) is connected. Therefore, the central inequality appearing in \eqref{eq:refineq} is the well-known Kneser's inequality~\cite[Satz (2)]{Kneser}, whose extreme cases will be crucial in the later proof of Theorem~\ref{mthm:mineqext}, so we mention it here:

\begin{thm}[Kneser's Inequality]\label{thm:kneser}
Let \(K\) be a compact connected abelian group and \(m_{K}\) its Haar probability measure. Then for every measurable set \(C\subseteq K\) with \(0<m_{K}\left(C\right)\leq1/2\), the following holds:
\[m_{K}\left(C-C\right)\geq 2\cdot m_{K}\left(C\right),\]
with equality if and only if there is a character \(\chi\in\widehat{K}\) and a closed interval \(I\subset\mathbb{T}\) with
\[C=\chi^{-1}\left(I\right)\text{ modulo }m_{K}.\]
\end{thm}

\begin{proof}[Proof of Theorem~\ref{thm:mineq}]
As \(G\) is totally disconnected, we may fix an arbitrarily small compact open subgroup \(V<G\). Applying the computation \eqref{eq:vid} with \(U=V\) and using that \(V-V=V\), we obtain
\begin{equation}\label{eq:vid1}
m_{G}\left(V\right)\cdot I_{\mu}\left(Y\right)=\mu^{\otimes2}\left(G\diag\left(V.Y\times V.Y\right)\right).
\end{equation}
Pick some countable dense subgroup \(\Gamma<G\), and recall that \(\mathcal{K}_{G}=\mathcal{K}_{\Gamma_{\mathrm{d}}}\) by Lemma~\ref{lem:gamkro}. Then for the Kronecker shadow \(S\coloneqq \left(V.Y\right)_{\mathcal{K}_{G}}\subseteq K\) we get from \eqref{eq:vid1} and using Lemma~\ref{lem:shad} that
\begin{align*}
m_{G}\left(V\right)\cdot I_{\mu}\left(Y\right)
&=\mu^{\otimes2}\left(G\diag\left(V.Y\times V.Y\right)\right)\\
&\geq\mu\otimes\mu\left(\Gamma\diag\left(V.Y\times V.Y\right)\right)=m_{K}\otimes m_{K}\left(\tau\left(\Gamma\right)\diag\left(S\times S\right)\right).
\end{align*}
Using Lemmas~\ref{lem:ijprime} and~\ref{lem:shad}, there is a \(m_{K}\)-conull measurable set \(S_{o}\subseteq S\) such that \(\tau\left(\Gamma\right).\left(S\times S\right)=K.\left(S_{o}\times S_{o}\right)\) modulo \(m_{K}\otimes m_{K}\). All together, it follows that
\[m_{G}\left(V\right)\cdot I_{\mu}\left(Y\right)\geq m_{K}\otimes m_{K}\left(K.\left(S_{o}\times S_{o}\right)\right)=m_{K}\left(S_{o}-S_{o}\right),\]
where the last equality holds because \(K.\left(S_{o}\times S_{o}\right)\) is the inverse image of \(S_{o}-S_{o}\) under the probability preserving map \(K\times K\to K\), \(\left(k,k^{\prime}\right)\mapsto k-k^{\prime}\). We thus have established the first inequality in \eqref{eq:refineq}. For the second inequality of \eqref{eq:refineq}, assume that \(V\) is sufficiently small so that \(m_{G}\left(V\right)\cdot I_{\mu}\left(Y\right)<1/2\), and therefore \(m_{K}\left(S_{o}\right)\leq m_{K}\left(S_{o}\times S_{o}\right)<1/2\), so the desired inequality is Kneser's inequality~\ref{thm:kneser}. For the last inequality of \eqref{eq:refineq}, we note that in general \(A\subseteq A_{\mathcal{K}_{G}}\) modulo \(\mu\), and therefore
\[m_{K}\left(S_{o}\right)=m_{K}\left(S\right)\leq\mu\left(V.Y\right)=m_{G}\left(V\right)\cdot\mu_{Y}\left(Y\right)=m_{G}\left(V\right)\cdot\iota_{\mu}\left(Y\right).\qedhere\]
\end{proof}

\section{The extremes of the main inequality}

The following section will deal with proving Theorem~\ref{mthm:mineqext}. We start with the outline of the main idea.

\subsection{Outline of the proof of Theorem~\ref{mthm:mineqext}}

Let \(G\) be a class \(\mathcal{Q}\) group. Fix a sequence of compact open subgroups of \(G\),
\[V_{1}\gneq V_{2}\gneq V_{3}\dotsm\text{ such that }\bigcap\nolimits_{n\in\mathbb{N}}V_{n}=\left\{e_{G}\right\}.\]
Suppose \(\left(X,\mu,Y\right)\) is an ergodic transverse \(G\)-space, \(Y\) is doubly separated, such that \(I_{\mu}\left(Y\right)=2\cdot\iota_{\mu}\left(Y\right)\). Then it forces a sequence of equalities in \eqref{eq:refineq} for \(V_{1},V_{2},\dotsc\). By the extremes of Kneser's inequality~\ref{thm:kneser}, one obtains \(G\)-eigenfunctions \(f_{1},f_{2},\dotsc:X\to\mathbb{T}\) ans closed intervals \(\mathbb{T}\supset I_{1}\supseteq I_{2}\supseteq\dotsm\), such that
\[f_{n}^{-1}\left(I_{n}\right)=V_{n}.Y\text{ modulo }\mu\text{ for every }n\in\mathbb{N},\]
and it can be shown that the corresponding eigencharacters \(\chi_{1},\chi_{2},\dotsc\in\widehat{G}\) satisfy
\[V_{n}\leq\ker\chi_{n}\text{ and }\chi_{n}=\chi_{n+1}^{q_{n}}\text{ for some }q_{n}\in\mathbb{Z},\quad n\in\mathbb{N}.\]
With an appropriate choice of the sequence \(V_{n}\), \(n\in\mathbb{N}\), it can be shown that the integers \(q_{n}\), \(n\in\mathbb{N}\), are all primes. Letting \(V_{n}\backslash X\) be the space of \(V_{n}\)-orbits in \(X\),\footnote{\(V_{n}\backslash X\) is a standard Borel space since \(V_{n}\) is compact.} we obtain an inverse limit system
\[\vcenter{\xymatrix{
\left(X,\mu\right) \ar[d]^{\phi}\\ \left(\mathbb{S},m_{\mathbb{S}}\right)
}}
\,\,\,=\varprojlim\,\,\,\,
\vcenter{\xymatrix@C=3.8em{X/V_{1}\ar[d]^{f_{1}} & \dotsm\ar@{_{(}->}[l] & X/V_{n}\ar@{_{(}->}[l]\ar[d]^{f_{n}} & X/V_{n+1}\ar@{_{(}->}[l]\ar[d]^{f_{n+1}} & \ar@{_{(}-->}[l]\\
\mathbb{T} & \dotsm\ar[l]_{z^{q_{1}}\mapsfrom z} & \mathbb{T}\ar[l]_{z^{q_{n-1}}\mapsfrom z} & \mathbb{T}\ar[l]_{z^{q_{n}}\mapsfrom z} & \ar@{-->}[l]
}
}\]
We will show that the solenoid \(\mathbb{S}\coloneqq \varprojlim\big(\mathbb{T}\xleftarrow{z^{q_{1}}\mapsfrom z}\mathbb{T}\xleftarrow{z^{q_{2}}\mapsfrom z}\mathbb{T}\dashleftarrow\big)\) is a transitive \(G\times\mathbb{R}\)-space with discrete stabilizers. Moreover, \(\mathbb{S}\) is a Borel \(G\)-space admitting a cross section arising from the intervals \(I_{1}\supseteq I_{2}\supseteq\dotsm\), in such a way that it becomes the required transverse cut--and--project \(G\)-space.

\subsection{Some lemmas on characters}

For \(\delta>0\) let us denote the open interval
\[J_{\delta}\coloneqq\left\{ z\in\mathbb{T}:\operatorname{dist}\left(1,z\right)<\delta\right\},\]
where \(\operatorname{dist}\) stands for the usual arc length on \(\mathbb{T}\). For the rest of this discussion, we fix an arbitrary lcsc abelian group \(G\).

\begin{lem}\label{lem:genlem1}
For every character \(\chi\in\widehat{G}\) and identity neighborhood \(V\subseteq G\), there exists \(\delta>0\) such that
\[\chi^{-1}\left(J_{\delta}\right)\subseteq V\ker\chi.\]
\end{lem}

\begin{proof}
Let \(\pi:G\to\ker\chi\backslash G\) be the quotient map. By the first isomorphism theorem for topological groups, we have the topological group isomorphism
\[\widetilde{\chi}:\ker\chi\backslash G\longrightarrow\chi\left(G\right),\quad\widetilde{\chi}\left(\pi\left(g\right)\right)=\chi\left(g\right).\]
Let \(V\subseteq G\) be an arbitrary identity neighborhood. Since \(\pi\) is an open map,
\[\pi\left(V\right)\subseteq\ker\chi\backslash G\]
is an identity neighborhood, and since \(\widetilde{\chi}\) is an open map, \(\widetilde{\chi}\left(\pi\left(V\right)\right)\subseteq\chi\left(G\right)\) is an identity neighborhood. Because \(\pi\) is surjective we have
\[\chi\left(V\right)=\widetilde{\chi}\left(\pi\left(V\right)\right),\]
and therefore \(\chi\left(V\right)\subseteq\chi\left(G\right)\) is an identity neighborhood. It follows that if \(\delta>0\) is sufficiently small then \(J_{\delta}\cap\chi\left(G\right)\subseteq\chi\left(V\right)\). Thus, for every \(g\in \chi^{-1}\left(J_{\delta}\right)\) there is \(v\in V\) such that \(\chi\left(g\right)=\chi\left(v\right)\), hence \(g=v\left(gv^{-1}\right)\in V\ker\chi\), concluding that \(\chi^{-1}\left(J_{\delta}\right)\subseteq V\ker\chi\).
\end{proof}

\begin{lem}\label{lem:genlem2}
Suppose \(\widehat{G}\) is torsion-free. For every pair of characters \(\chi_{1},\chi_{2}\in\widehat{G}\) such that \(\ker\chi_{1}\leq\ker\chi_{2}\), there is an integer \(q\) such that \(\chi_{2}=\chi_{1}^{q}\), and in fact
\[q=\pm\left[\ker\chi_{2}:\ker\chi_{1}\right]<\infty.\]
\end{lem}

\begin{proof}
Consider the group homomorphism
\[\Phi:\chi_{1}\left(G\right)\longrightarrow\chi_{2}\left(G\right),\quad\Phi\left(\chi_{1}\left(g\right)\right)=\chi_{2}\left(g\right),\]
which is well-defined since \(\ker\chi_{1}\leq\ker\chi_{2}\). We claim that \(\Phi\) is uniformly continuous in the corresponding subgroup topologies. Thus, we must show that for every \(\epsilon>0\) there exists \(\delta>0\) such that
\[\operatorname{dist}\left(\chi_{1}\left(g\right),\chi_{1}\left(h\right)\right)<\delta\implies\operatorname{dist}\left(\chi_{2}\left(g\right),\chi_{2}\left(h\right)\right)<\epsilon\text{ for all }g,h\in G,\]
or equivalently,
\[\operatorname{dist}\left(1,\chi_{1}\left(g\right)\right)<\delta\implies\operatorname{dist}\left(1,\chi_{2}\left(g\right)\right)<\epsilon\text{ for all }g\in G.\]
Let \(\epsilon>0\) be arbitrary. By the continuity of \(\chi_{2}\) there is an identity neighborhood \(V\subset G\) such that \(\chi_{2}\left(V\right)\subseteq J_{\epsilon}\). By Lemma~\ref{lem:genlem1} for \(\chi_{1}\) and \(V\), there exists \(\delta>0\) such that \(\chi_{1}^{-1}\left(J_{\delta}\right)\subseteq V\ker\chi_{1}\) and therefore, since \(\ker\chi_{1}\leq\ker\chi_{2}\), in fact \(\chi_{1}^{-1}\left(J_{\delta}\right)\subseteq V\ker\chi_{2}\). Thus, for every \(g\in G\), if \(\operatorname{dist}\left(1,\chi_{1}\left(g\right)\right)<\delta\) then \(g\in V\ker\chi_{2}\), so we can write \(g=vh\) with \(v\in V\) and \(h\in\ker\chi_{2}\), and therefore
\[\operatorname{dist}\left(1,\chi_{2}\left(g\right)\right)=\operatorname{dist}\left(1,\chi_{2}\left(vh\right)\right)=\operatorname{dist}\left(1,\chi_{2}\left(v\right)\right)<\epsilon.\]
This establishes the uniform continuity of \(\Phi\). Since \(\widehat{G}\) is torsion-free, both subgroups \(\chi_{1}\left(G\right)\) and \(\chi_{2}\left(G\right)\) of \(\mathbb{T}\) are infinite and hence dense. Therefore, \(\Phi\) extends uniquely to a continuous epimorphism of \(\Phi:\mathbb{T}\to\mathbb{T}\), so there exists an integer \(q\) such that \(\Phi\left(z\right)=z^{q}\), and thus \(\chi_{2}=\chi_{1}^{q}\). As a consequence,
\[\ker\chi_{1}\backslash\ker\chi_{2}\to\mathbb{T},\quad g+\ker\chi_{1}\mapsto\chi_{1}\left(g\right),\]
is an injective homomorphism into the group of \(q^{\mathrm{th}}\) roots of unity, implying that
\[r\coloneqq\left[\ker\chi_{2}:\ker\chi_{1}\right]<\infty\text{ satisfies }r\mid q.\]
Let us show that also \(q\mid r\). For every \(g\in\ker\chi_{2}\) we have \(g^{r}\in\ker\chi_{1}\), and thus \(\ker\chi_{2}\leq\ker\chi_{1}^{r}\). By the first part of the proof it follows that there exists a nonzero integer \(s\) such that \(\chi_{1}^{r}=\chi_{2}^{s}\), hence \(\chi_{1}^{r}=\chi_{1}^{qs}\), and since \(\widehat{G}\) is torsion-free it follows that \(r=qs\), thus \(q\mid r\). All together, we obtained \(q=\pm r\).
\end{proof}

\begin{lem}\label{lem:genlem3}
Let \(J\subset\mathbb{T}\) be a closed interval and \(F<\mathbb{T}\) a finite subgroup such that \(JF\neq\mathbb{T}\). Then for every \(t\in\mathbb{T}\), if \(JF=JFt\) modulo \(m_{\mathbb{T}}\) then \(t\in F\).
\end{lem}

\begin{proof}
The case of \(F\) being the trivial subgroup is clear. For a general finite subgroup \(F\), note that \(F\backslash\mathbb{T}\) is again a torus, and let \(\pi:\mathbb{T}\to F\backslash\mathbb{T}\) be the quotient map. If for \(t\in\mathbb{T}\) we have \(JF=JFt\), then \(\pi\left(J\right)\pi\left(F\right)=\pi\left(J\right)\pi\left(F\right)\pi\left(t\right)\). But \(\pi\left(J\right)\) is a compact connected set, hence a closed interval in the torus \(F\backslash\mathbb{T}\). From the case of trivial subgroup we have \(\pi\left(t\right)=1\in F\backslash\mathbb{T}\), thus \(t\in F\).
\end{proof}

\subsection{Proof of Theorem~\ref{mthm:mineqext}}

Let \(G\) be a class \(\mathcal{Q}\) group, \(\left(X,\mu,Y\right)\) an ergodic transverse \(G\)-space where \(Y\) is doubly separated and assume \(I_{\mu}\left(Y\right)=2\cdot\iota_{\mu}\left(Y\right)\). Fix a sequence of compact open subgroups of \(G\),
\[V_{1}\gneq V_{2}\gneq V_{3}\dotsm\text{ such that }\bigcap\nolimits_{n\in\mathbb{N}}V_{n}=\left\{e_{G}\right\}.\]

\subsubsection*{\textbf{Step~1}}

We exploit the equality \(I_{\mu}\left(Y\right)=2\cdot\iota_{\mu}\left(Y\right)\) and the extremes of Kneser's inequality~\ref{thm:kneser}, in order to construct \(G\)-eigenfunctions \(f_{1},f_{2},\dotsc:X\to\mathbb{T}\) with corresponding eigencharacters \(\chi_{1},\chi_{2},\dotsc\in\widehat{G}\), as well as closed centered intervals \(\mathbb{T}\supset I_{1}\supseteq I_{2}\supseteq\dotsm\), such that for every \(n\in\mathbb{N}\),
\[f_{n}^{-1}\left(I_{n}\right)=V_{n}.Y\text{ modulo }\mu.\]

\smallskip

Let us first modify the sequence \(V_{1},V_{2},\dotsc\) as follows. We first assume that \(V_{1}\) is sufficiently small to satisfy Theorem~\ref{thm:mineq}, and thus so is every \(V_{n}\). We assume further that \(V_{1}\) is sufficiently small so that \(V_{1}V_{1}V_{1}\) is a separating neighborhood for \(Y\), and thus so is every \(V_{n}\). Since each \(V_{n+1}\backslash V_{n}\) is finite, we can refine the sequence into a maximal one, in that for every \(n\in\mathbb{N}\), there is no compact open subgroup lies strictly between \(V_{n}\) and \(V_{n+1}\). For every \(n\in\mathbb{N}\) fix a finite set
\[Q_{n}\subset V_{n}\text{ such that }Q_{n}V_{n+1}=V_{n}.\]
Note that by the maximality, \(\left|Q_{n}\right|=\left[V_{n}:V_{n+1}\right]\) is a prime (by Cauchy's theorem).

Let \(\mathcal{K}_{G}\) be the Kronecker factor of \(\left(X,\mu\right)\), with the corresponding Kronecker compactification \(\left(K,\tau\right)\) and the \(G\)-factor map \(\kappa:\left(X,\mu\right)\to\left(K,m_{K}\right)\) as in Theorem~\ref{thm:krongrp}. For every \(n\in\mathbb{N}\) put the Kronecker shadow \(S_{n}\coloneqq \left(V_{n}.Y\right)_{\mathcal{K}_{G}}\subseteq K\). Since \(I_{\mu}\left(Y\right)=2\cdot\iota_{\mu}\left(Y\right)\), we obtain an equality for \(V_{n}\) in the inequality \eqref{eq:refineq} of Theorem~\ref{thm:mineq}, and thus \(m_{K}\left(S_{n}-S_{n}\right)=2\cdot m_{K}\left(S_{n}\right)\). Then by the extremes of Kneser's inequality~\ref{thm:kneser} there exists a character \(\xi_{n}\in\widehat{K}\) and a closed interval \(I_{n}\subset\mathbb{T}\) such that
\[S_{n}=\xi_{n}^{-1}\left(I_{n}\right)\text{ modulo }m_{K}.\]
Since \(m_{K}\left(S_{n}\right)\xrightarrow[n\to\infty]{}0\) we may assume that every \(\xi_{n}\) is nontrivial. Define then a \(G\)-eigenfunction \(f_{n}\) with eigencharacter \(\chi_{n}\) by
\[f_{n}:X\to\mathbb{T},\quad f_{n}\coloneqq \xi_{n}\circ\kappa,\quad\text{and}\quad\chi_{n}\coloneqq \xi_{n}\circ\tau\in\widehat{G},\]
and note that
\[f_{n}^{-1}\left(I_{n}\right)=\kappa^{-1}\left(\xi_{n}^{-1}\left(I_{n}\right)\right)=\kappa^{-1}\left(S_{n}\right)=V_{n}.Y\text{ modulo }\mu,\]
where the last equality is because there is an equality in \eqref{eq:refineq}, by which \(m_{K}\left(S_{n}\right)=\mu\left(V_{n}.Y\right)\).

Finally, in order to simplify later arguments, we will modify this construction so the intervals \(I_{n}\) are centered (i.e. symmetric about \(1\in\mathbb{T}\)) and nested as follows. Since \(V_{n}\gneq V_{n+1}\) and
\[m_{\mathbb{T}}\left(I_{n+1}\right)=\mu\left(f_{n}^{-1}\left(I_{n}\right)\right)=m_{G}\left(V_{n}.Y\right),\]
we have \(m_{\mathbb{T}}\left(I_{n}\right)\geq m_{\mathbb{T}}\left(I_{n+1}\right)\), so there is \(t_{n+1}\in\mathbb{T}\) such that \(I_{n}\supseteq I_{n+1}t_{n+1}\) and moreover \(I_{n+1}t_{n+1}\) lies at the center of \(I_{n}\). Replace \(I_{n+1}\) and \(f_{n+1}\) with \(I_{n+1}^{\prime}\coloneqq I_{n+1}t_{n+1}\) and \(f_{n+1}^{\prime}\coloneqq f_{n+1}t_{n+1}\). Evidently, \(f_{n+1}^{\prime}\) is a \(G\)-eigenfunction with eigencharacter \(\chi_{n+1}\), and \(f_{n+1}^{\prime -1}\left(I_{n+1}^{\prime}\right)=V_{n+1}.Y\). Doing this inductively, we obtain \(I_{1}^{\prime}\supseteq I_{2}^{\prime}\supseteq\dotsm\) and that \(I_{n+1}^{\prime}\) lies at the center of \(I_{n}^{\prime}\) for every \(n\in\mathbb{N}\).

\subsubsection*{\textbf{Step~2}}

We now show that for every \(n\in\mathbb{N}\), there is an integer \(q_{n}\notin\{0,-1,+1\}\) such that
\[\chi_{n}=\chi_{n+1}^{q_{n}}\text{ and }f_{n}=f_{n+1}^{q_{n}}.\]

\smallskip

Let \(n\in\mathbb{N}\). Aiming to apply Lemma~\ref{lem:genlem2}, let us show that
\begin{equation}\label{eq:ker}
\ker\chi_{n+1}\leq\ker\chi_{n}.
\end{equation}
To establish this property, recall that \(Q_{n}V_{n+1}=V_{n}\), and therefore
\[f_{n}^{-1}\left(I_{n}\right)=V_{n}.Y=Q_{n}V_{n+1}.Y=Q_{n}.f_{n+1}^{-1}\left(I_{n+1}\right)\text{ modulo }\mu,\]
so that for every \(g\in G\) we have
\[f_{n}^{-1}\left(I_{n}\chi_{n}\left(g\right)\right)=g.f_{n}^{-1}\left(I_{n}\right)=gQ_{n}.f_{n+1}^{-1}\left(I_{n+1}\right)=Q_{n}g.f_{n+1}^{-1}\left(I_{n+1}\right)\text{ modulo }\mu.\]
Then whenever \(g\in\ker\chi_{n+1}\) we obtain
\[f_{n}^{-1}\left(I_{n}\chi_{n}\left(g\right)\right)=Q_{n}.f_{n+1}^{-1}\left(I_{n+1}\right)=f_{n}^{-1}\left(I_{n}\right)\text{ modulo }\mu,\]
and consequently, since \(f_{n}\) is a \(G\)-factor map and thus measure preserving,
\[I_{n}\chi_{n}\left(g\right)=I_{n}\text{ modulo }m_{\mathbb{T}}.\]
Since \(I_{n}\subsetneq\mathbb{T}\) is an interval, necessarily \(\chi_{n}\left(g\right)=1\) and \(g\in\ker\chi_{n}\), establishing \eqref{eq:ker}. Now by Lemma~\ref{lem:genlem2} applied to \(\chi_{n}\) and \(\chi_{n+1}\), we have for \(q_{n}\coloneqq\left[\ker\chi_{n}:\ker\chi_{n+1}\right]\) that \(\chi_{n}=\chi_{n+1}^{q_{n}}\). Since \(\chi_{n}=\xi_{n}\circ\tau\) and \(\tau\left(G\right)\leq K\) is dense, from \(\chi_{n}=\chi_{n+1}^{q_{n}}\) it follows that \(\xi_{n}=\xi_{n+1}^{q_{n}}\), which implies that \(f_{n}=f_{n+1}^{q_{n}}\).

We must show that \(q_{n}\neq\pm 1\). To this end, let us first make the useful observation that
\begin{equation}\label{eq:kersV}
V_{1}\cap\ker\chi_{n}\leq V_{n}\leq\ker\chi_{n}\text{ for every }n\in\mathbb{N}.
\end{equation}
Indeed, for the first containment, if \(g\in V_{1}\cap\ker\chi_{n}\) then we have
\[gV_{n}.Y=g.f_{n}^{-1}\left(I_{n}\right)=f_{n}^{-1}\left(I_{n}\chi_{n}\left(g\right)\right)=f_{n}^{-1}\left(I_{n}\right)=V_{n}.Y\text{ modulo }\mu,\]
and therefore by Fubini's theorem there are \(v,v^{\prime}\in V_{n}\) and \(y,y^{\prime}\in Y\) such that \(gv.y=v^{\prime}.y^{\prime}\), so that \(gvv^{\prime -1}\in Y_{y}\cap V_{1}V_{n}V_{n}\subseteq\Lambda_{Y}\cap V_{1}V_{1}V_{1}=\{e_{G}\}\), hence \(g=v^{\prime}v\in V_{n}\) as \(V_{n}\) is a subgroup, establishing the first containment in \eqref{eq:kersV}. For the second containment, if \(g\in V_{n}\) then we have
\[f_{n}^{-1}\left(I_{n}\chi_{n}\left(g\right)\right)=g.f_{n}^{-1}\left(I_{n}\right)=gV_{n}.Y=V_{n}.Y=f_{n}^{-1}\left(I_{n}\right)\text{ modulo }\mu,\]
hence \(I_{n}=I_{n}\chi_{n}\left(g\right)\) modulo \(m_{\mathbb{T}}\) and \(\chi_{n}\left(g\right)=0\), establishing the second containment in \eqref{eq:kersV}.

We can now show that \(q_{n}\neq\pm 1\). Note that from \(V_{n}\leq\ker\chi_{n}\) it follows that \(V_{n}.f_{n}^{-1}\left(\cdot\right)=f_{n}^{-1}\left(\cdot\right)\). Now if \(q_{n}=1\), then \(\chi_{n}=\chi_{n+1}\) and hence \(f_{n}=f_{n+1}\), so we have
\[V_{n+1}.Y=f_{n+1}^{-1}\left(I_{n+1}\right)=f_{n}^{-1}\left(I_{n+1}\right)=V_{n}.f_{n}^{-1}\left(I_{n+1}\right)=V_{n}.f_{n+1}^{-1}\left(I_{n+1}\right)=V_{n}V_{n+1}.Y=V_{n}.Y\text{ modulo }\mu,\]
where in the third equality we used that \(V_{n}\leq\ker f_{n}\) by \eqref{eq:kersV}. However, since \(V_n\gneq V_{n+1}\) are open,
\[\mu\left(V_{n}.Y\right)=m_{G}\left(V_{n}\right)\cdot\mu_{Y}\left(Y\right)\gneq m_{G}\left(V_{n+1}\right)\cdot\mu_{Y}\left(Y\right)=\mu\left(V_{n+1}.Y\right),\]
so it is impossible that \(q_{n}=1\). A similar argument shows that \(q_{n}\neq -1\).

\subsubsection*{\textbf{Step~3}}

We now show that \(q_{n}=\left|Q_{n}\right|\), and in particular \(q_{n}\) is a prime, for every \(n\in\mathbb{N}\).

\smallskip

Let \(n\in\mathbb{N}\), and consider the finite subgroup
\[F_{n}\coloneqq\chi_{n+1}\left(V_{n}\right)=\chi_{n+1}\left(Q_{n}V_{n+1}\right)=\chi_{n+1}\left(Q_{n}\right)<\mathbb{T},\]
where the third equality is since \(V_{n+1}\leq\ker\chi_{n+1}\) by \eqref{eq:kersV}. Since \(m_{G}\left(V_{n}\right)=\left|Q_{n}\right|\cdot m_{G}\left(V_{n+1}\right)\) we have
\[m_{\mathbb{T}}\left(I_{n}\right)=\mu\left(V_{n}.Y\right)=m_{G}\left(V_{n}\right)\cdot\mu_{Y}\left(Y\right)=\left|Q_{n}\right|\cdot m_{G}\left(V_{n+1}\right)\cdot\mu_{Y}\left(Y\right)=\left|Q_{n}\right|\cdot\mu\left(V_{n+1}.Y\right)=\left|Q_{n}\right|\cdot m_{\mathbb{T}}\left(I_{n+1}\right).\]
At the same time, \(V_{n}.Y=Q_{n}V_{n+1}.Y=Q_{n}.f_{n+1}^{-1}\left(I_{n+1}\right)=f_{n+1}^{-1}\left(I_{n+1}F_{n}\right)\) modulo \(\mu\), hence
\[m_{\mathbb{T}}\left(I_{n}\right)=\mu\left(f_{n}^{-1}\left(I_{n}\right)\right)=\mu\left(V_{n}.Y\right)=\mu\left(f_{n+1}^{-1}\left(I_{n+1}F_{n}\right)\right)=m_{\mathbb{T}}\left(I_{n+1}F_{n}\right).\]
Both identities together imply
\[\left|Q_{n}\right|\cdot m_{\mathbb{T}}\left(I_{n+1}\right)=m_{\mathbb{T}}\left(I_{n+1}F_{n}\right)\leq\left|F_{n}\right|\cdot m_{\mathbb{T}}\left(I_{n+1}\right),\text{ hence }\left|Q_{n}\right|\leq\left|F_{n}\right|,\]
and at the same time, we simply have
\[\left|F_{n}\right|=\left|\chi_{n+1}\left(Q_{n}\right)\right|\leq\left|Q_{n}\right|.\]
All together, \(\left|Q_{n}\right|=\left|F_{n}\right|\) so \(F_{n}\) is a finite group of prime order. Using that \(q_{n}=\left[\ker\chi_{n}:\ker\chi_{n+1}\right]\) by construction (as in Lemma~\ref{lem:genlem2}), pick \(P_{n}\subset\ker\chi_{n}\) with \(q_{n}=\left|P_{n}\right|\) and \(\ker\chi_{n}\coloneqq P_{n}\ker\chi_{n+1}\). Recall
\[f_{n}^{-1}\left(I_{n}\right)=V_{n}.Y=Q_{n}V_{n+1}.Y=f_{n+1}^{-1}\left(I_{n+1}F_{n}\right),\]
and note that for every \(g\in\ker\chi_{n}\),
\[f_{n}^{-1}\left(I_{n}\right)=g.f_{n}^{-1}\left(I_{n}\right)=g.f_{n+1}^{-1}\left(I_{n+1}F_{n}\right)=f_{n+1}^{-1}\left(I_{n+1}F_{n}\chi_{n+1}\left(g\right)\right).\]
It follows from the last two identities that
\[I_{n+1}F_{n}=I_{n+1}F_{n}\chi_{n+1}\left(g\right)\text{ modulo }m_{\mathbb{T}}\text{ for all }g\in\ker\chi_{n},\]
and therefore, from Lemma~\ref{lem:genlem3} we deduce that
\[\chi_{n+1}\left(g\right)\in F_{n}\text{ for all }g\in\ker\chi_{n}.\]
In other words, \(\chi_{n+1}\left(P_{n}\right)=\chi_{n+1}\left(P_{n}\ker\chi_{n+1}\right)=\chi_{n+1}\left(\ker\chi_{n}\right)\leq F_{n}\) is a (nontrivial) subgroup of the prime order group \(F_{n}\), and therefore \(\chi_{n+1}\left(P_{n}\right)=F_{n}\), hence \(\left|\chi_{n+1}\left(P_{n}\right)\right|=\left|\chi_{n+1}\left(Q_{n}\right)\right|=\left|Q_{n}\right|\). In particular, since \(P_{n}\cap\ker\chi_{n+1}=\{e_{G}\}\), we conclude that \(q_{n}=\left|P_{n}\right|=\left|Q_{n}\right|\).

\subsubsection*{\textbf{Step~4}}

Here we deduce from previous steps that \(I_{n}=I_{n+1}^{q_{n}}\) for every \(n\in\mathbb{N}\); that is, the image of \(I_{n+1}\) under the map \(\mathbb{T}\to\mathbb{T}\), \(z\mapsto z^{q_{n}}\), is \(I_{n}\).

\smallskip

Start by recalling that
\[f_{n}^{-1}\left(I_{n}\right)=V_{n}.Y=Q_{n}V_{n+1}.Y=Q_{n}.f_{n+1}^{-1}\left(I_{n+1}\right)\text{ where }\left|Q_{n}\right|=q_{n}.\]
For all \(u,u^{\prime}\in Q_{n}\), if it holds that \(uV_{n+1}.Y\cap u^{\prime}V_{n+1}.Y\neq\emptyset\) then there are \(v,v^{\prime}\in V_{n+1}\) such that \(uu^{\prime -1}vv^{\prime -1}\in\Lambda_{Y}\cap V_{n}=\{e_{G}\}\),\footnote{\(V_{n}\) is a separating neighborhood for \(Y^{\left[2\right]}\), and a fortiori a separating neighborhood for \(Y\).} and therefore \(uu^{\prime -1}=v^{\prime}v^{-1}\in V_{n+1}\), which implies that \(uV_{n+1}.Y=u^{\prime}V_{n+1}.Y\). However, from \(m_{G}\left(V_{n}\right)=q_{n}\cdot m_{G}\left(V_{n+1}\right)\) it follows that \(\mu\left(V_{n}.Y\right)=q_{n}\cdot\mu\left(V_{n+1}.Y\right)\), so that \(q_{n}\) is minimal, and thus \(uV_{n+1}.Y=u^{\prime}V_{n+1}.Y\) forces \(u=u^{\prime}\). This discussion shows that \(\left\{uV_{n+1}.Y:u\in Q_{n}\right\}\) are pairwise disjoint and covering \(V_{n}.Y\), and therefore
\[m_{\mathbb{T}}\left(I_{n}\right)=\mu\left(f_{n}^{-1}\left(I_{n}\right)\right)=\mu\left(V_{n}.Y\right)=q_{n}\cdot\mu\left(V_{n+1}.Y\right)=q_{n}\cdot\mu\left(f_{n+1}^{-1}\left(I_{n+1}\right)\right)=q_{n}\cdot m_{\mathbb{T}}\left(I_{n+1}\right).\]
Since by construction both \(I_{n+1}\) and \(I_{n}\) are centered, necessarily \(I_{n}=I_{n+1}^{q_{n}}\).

\subsubsection*{\textbf{Step~5}}

We now construct a solenoid \(\mathbb{S}\) that is Borel isomorphic to a quotient of \(G\times\mathbb{R}\) by a lattice as Borel \(G\)-spaces.

\smallskip

Define the solenoid associated with the sequence of primes \(q_{1},q_{2},\dotsc\),
\[\mathbb{S}\coloneqq \varprojlim\big(\mathbb{T}\xleftarrow{z^{q_{1}}\mapsfrom z}\mathbb{T}\xleftarrow{z^{q_{2}}\mapsfrom z}\mathbb{T}\dashleftarrow\big)=\left\{ \left(z_{n}\right)_{n\in\mathbb{N}}\in\mathbb{T}^{\mathbb{N}}:\forall n\in\mathbb{N},\,\,z_{n}=z_{n+1}^{q_{n}}\right\}.\]
Thus, \(\mathbb{S}\) is a compact abelian group with Haar probability measure \(m_{\mathbb{S}}\). Letting \(q_{0}\coloneqq 1\), define the map
\[\rho:\mathbb{R}\longrightarrow\mathbb{S},\quad\rho\left(r\right)\coloneqq\big(e^{2\pi ir/\left(q_{0}\dotsm q_{n-1}\right)}\big)_{n\in\mathbb{N}}.\]
It is a basic fact that \(\rho\) is a continuous group-embedding with dense image, thus \(\left(\mathbb{S},\rho\right)\) is a compactification of \(\mathbb{R}\) (see~\cite[(25.3)]{HewittRoss}). Recalling the identities \(\chi_{n}=\chi_{n+1}^{q_{n}}\) for \(n\in\mathbb{N}\), we can define the continuous map
\[\chi:G\longrightarrow\mathbb{S},\quad\chi\left(g\right)\coloneqq\left(\chi_{n}\left(g\right)\right)_{n\in\mathbb{N}}.\]
Define the Borel map
\[\Phi:G\times\mathbb{R}\longrightarrow\mathbb{S},\quad\Phi\left(g,r\right)\coloneqq\chi\left(g\right)\oplus\rho\left(r\right),\]
where \(\oplus\) denotes addition in the group \(\mathbb{S}\). We claim that \(\Phi\) is surjective and that its kernel is discrete:
\begin{itemize}
    \item \(\Phi\) is surjective, namely \begin{equation}\label{eq:So}
    \mathbb{S}=\Phi\left(G\times\mathbb{R}\right)=\chi\left(G\right)\oplus\rho\left(\mathbb{R}\right).
    \end{equation}
    Fix an arbitrary \(z=\left(z_{n}\right)_{n\in\mathbb{N}}\in\mathbb{S}\), so \(z_{n}=z_{n+1}^{q_{n}}\) for all \(n\in\mathbb{N}\). Choose \(r_{o}\in\mathbb{R}\) with \(e^{2\pi i r_{o}}=z_{1}\). We now construct inductively a convergent sequence \(\left(g_{n}\right)_{n\in\mathbb{N}}\) in \(G\). Assume that for some \(n\in\mathbb{N}\) we have constructed \(g_{n}\in G\) satisfying
    \begin{equation}\label{eq:induct}
    \chi_{k}\left(g_{n}\right)\,e^{2\pi i r_{o}/\left(q_0\cdots q_{k-1}\right)}=z_{k},\qquad k=1,\dots,n.
    \end{equation}
    Let
    \[a_{n+1}\coloneqq \chi_{n+1}(g_{n})\,e^{2\pi i r_{o}/\left(q_{0}\cdots q_{n}\right)}\in\mathbb{T}.\]
    By the construction and since \(\chi_{n}=\chi_{n+1}^{q_{n}}\), we have that \(a_{n+1}^{q_{n}}=z_{n}=z_{n+1}^{q_{n}}\), hence \(\omega_{n}\coloneqq z_{n+1}a_{n+1}^{-1}\) is a \(q_{n}^{\mathrm{th}}\) root of unity. By Step~3 above, the finite group \(F_{n}\coloneqq \chi_{n+1}\left(V_{n}\right)\) has order \(\left|Q_{n}\right|=q_{n}\), hence it is the group of \(q^{\mathrm{th}}\) roots of unity so that \(\omega_{n}\in F_n\). Pick \(v_{n}\in V_{n}\) with \(\chi_{n+1}\left(v_{n}\right)=\omega_{n}\) and set \(g_{n+1}\coloneqq v_{n}g_{n}\). To see that \eqref{eq:induct} holds for \(g_{n+1}\), recall that according to \eqref{eq:ker} and \eqref{eq:kersV} we have \(V_{n}\leq\ker\chi_{n}\leq\ker\chi_{n-1}\leq\cdots\leq\ker\chi_1\), then multiplying by \(v_{n}\) does not change the first \(n\) coordinates in \eqref{eq:induct}, while
    \[\chi_{n+1}(g_{n+1})\,e^{2\pi i r_{o}/(q_{0}\cdots q_{n})}=\chi_{n+1}\left(v_{n}\right)\,\chi_{n+1}\left(g_{n}\right)\,e^{2\pi i r_{o}/(q_{0}\cdots q_{n})}=\omega_{n}a_{n+1}=z_{n+1}.\]
    Thus, \eqref{eq:induct} holds with \(n+1\). To verify that \(\left(g_{n}\right)_{n\in\mathbb{N}}\) is convergent in \(G\) note that for all \(m>n\),
    \[g_{m}g_{n}^{-1}=v_{m-1}\cdots v_{n}\in V_{n},\]
    so \(g_{m}\in V_{n}g_{n}\), and hence \(V_{1}g_{1}\supset V_{2}g_{2}\supset\dotsm\) are compact and nested. Since \(\bigcap_{n\in\mathbb{N}}V_{n}=\{e_G\}\), we have \(\bigcap_{n\in\mathbb{N}}V_{n}g_{n}=\{g_{o}\}\) for some \(g_{o}\in G\), namely \(g_{n}\to g_{o}\) as \(n\to\infty\).\\
    We can now conclude the surjectivity of \(\Phi\) by showing that \(\Phi\left(g_{o},r_{o}\right)=z\). For an arbitrary \(m\in\mathbb{N}\), the \(m^{\mathrm{th}}\) coordinate of \(\Phi\left(g_{n},r_{o}\right)\) equals \(z_{m}\) for all \(n\geq m\), and by continuity the \(m^{\mathrm{th}}\) coordinate of \(\Phi\left(g_{o},r_{o}\right)\) equals \(z_{m}\). Thus, \(\Phi\left(g_{o},r_{o}\right)=\left(z_{m}\right)_{m\in\mathbb{N}}=z\), and this proves the surjectivity of \(\Phi\).
    \item Denote the kernel
    \[\Gamma\coloneqq\ker\Phi=\left\{ \left(g,r\right)\in G\times\mathbb{R}:\chi\left(g\right)\oplus\rho\left(r\right)=e_{\mathbb{S}}\right\}.\]
    In order to show that \(\Gamma\) is discrete, we will show that
    \begin{equation}\label{eq:gamdisc}
    \Gamma\cap\left(V_{1}\times W\right)=\left\{ \left(e_{G},0\right)\right\}.
    \end{equation}
    Indeed, if \(\left(g,w\right)\in\Gamma\cap\left(V_{1}\times W\right)\), then \(\chi\left(g\right)\oplus\rho\left(w\right)=e_{\mathbb{S}}\) so that \(\chi\left(g\right)=\rho\left(-w\right)\in\rho\left(W\right)\), and in particular, \(\chi_{1}\left(g\right)=e^{-2\pi iw}\in I_{1}\). Recall \(V_{1}\leq\ker\chi_{1}\) by \eqref{eq:kersV}, so that \(e^{-2\pi iw}=\chi_{1}\left(g\right)=1\), hence \(w\in\mathbb{Z}\), and since \(w\in W\) it follows that \(w=0\) (recall \(\left|I_{1}\right|<1\) so that \(W\subseteq\left(-1/2,1/2\right)\)). Therefore, \(\chi\left(g\right)=e_{\mathbb{S}}\) and \(g\in\ker\chi\), concluding that \(\Gamma\cap\left(V_{1}\times W\right)\subseteq\left(V_{1}\cap\ker\chi\right)\times\left\{ 0\right\}\). However, from \eqref{eq:kersV} it follows that
    \[V_{1}\cap\ker\chi=\bigcap\nolimits_{n\in\mathbb{N}}V_{1}\cap\ker\chi_{n}\leq\bigcap\nolimits_{n\in\mathbb{N}}V_{n}=\left\{ e_{G}\right\},\]
    so that \eqref{eq:gamdisc} follows, and \(\Gamma<G\times\mathbb{R}\) is discrete.
\end{itemize}

Define the Borel space \(\Xi\coloneqq\Gamma\backslash\left(G\times\mathbb{R}\right)\). We obtain the Borel isomorphism (in fact, a homeomorphism),
\[\widetilde{\Phi}:\Xi\longrightarrow \mathbb{S},\qquad \widetilde{\Phi}\big(\Gamma(g,r)\big)=\Phi(g,r).\]
Since \(\mathbb{S}\) is compact, \(\Gamma<G\times\mathbb{R}\) is cocompact, and since \(\Gamma\) is also discrete we found that \(\Gamma<G\times\mathbb{R}\) forms a lattice, and thus \(G\) is cosolenoidal. Consider \(\Xi\) as a Borel \(G\times\mathbb{R}\)-space in the usual way, and turn \(\mathbb{S}\) into a Borel \(G\times\mathbb{R}\)-space via
\[\left(g,r\right).s=\Phi\left(g,r\right)\oplus s=\chi\left(g\right)\oplus\rho\left(r\right)\oplus s.\]
It is routine to verify that the map \(\widetilde{\Phi}\) is \(G\times\mathbb{R}\)-equivariant, and hence it forms an isomorphism of Borel \(G\times\mathbb{R}\)-spaces. Consequently, the pullback measure \(\xi\coloneqq m_{\mathbb{S}}\circ\widetilde{\Phi}\) on \(\Xi\) is the unique \(G\times\mathbb{R}\)-invariant probability measure, namely \(\xi=m_{\Gamma\backslash\left(G\times\mathbb{R}\right)}\). Altogether, \(\widetilde{\Phi}\) forms an isomorphism of probability preserving \(G\times\mathbb{R}\)-spaces of the solenoid \(\left(\mathbb{S},m_{\mathbb{S}}\right)\) and the homogeneous space \(\left(\Xi,\xi\right)\).

\subsubsection*{\textbf{Step~6}}

Consider now \(\mathbb{S}\) as a Borel \(G\)-space by restricting the action of \(G\times\mathbb{R}\) to \(G\), namely via
\[g.z=\chi\left(g\right)\oplus z.\]
Observe that since \(\left(\mathbb{S},m_{\mathbb{S}}\right)
\) is isomorphic to \(\left(\Xi,\xi\right)\), it is the cut--and--project \(G\)-space associated with the cut--and--project scheme \(\left(G,\mathbb{R};\Gamma\right)\). We now find a cross section \(S\subset\mathbb{S}\) with respect to this action of \(G\), in such a way that  \(\widetilde{\Phi}:\left(\Xi,\xi,Y_{W}\right)\to\left(\mathbb{S},m_{\mathbb{S}},S\right)\) forms a transverse \(G\)-isomorphism.

\smallskip

Define the Borel set
\[S\coloneqq\left(I_{1}\times I_{2}\times\dotsm\right)\cap\mathbb{S}\subset\mathbb{S}.\]
Pick a compact (symmetric) interval \(W\subset\mathbb{R}\) whose image under the map \(\mathbb{R}\to\mathbb{T}\), \(r\mapsto e^{2\pi ir}\), is the closed interval \(I_{1}\subset\mathbb{T}\), and we start by proving the identity
\begin{equation}\label{eq:SW}
S=\rho\left(W\right).
\end{equation}
For one inclusion, let \(\left(z_{n}\right)_{n\in\mathbb{N}}\in S\). Since \(z_{1}\in I_{1}\), there is \(w\in W\) with \(z_{1}=e^{2\pi iw}\). Now for every \(n\in\mathbb{N}\) we have \(z_{n}^{q_{0}\dotsm q_{n-1}}=z_{1}=e^{2\pi iw}\), and since \(z_{n}\in I_{n}\) while \(\left|I_{n}\right|=\frac{1}{q_{0}\dotsm q_{n-1}}\left|I_{1}\right|\) (by Step~4), necessarily \(z_{n}=e^{2\pi iw/\left(q_{0}\dotsm q_{n-1}\right)}\), hence \(\left(z_{n}\right)_{n\in\mathbb{N}}=\rho\left(w\right)\in\rho\left(W\right)\). For the opposite inclusion, let \(\left(z_{n}\right)_{n\in\mathbb{N}}\in\rho\left(W\right)\), so that \(\left(z_{n}\right)_{n\in\mathbb{N}}=\rho\left(w\right)\) for some \(w\in W\), and in particular \(z_{1}=e^{2\pi iw}\in I_{1}\). For every \(n\in\mathbb{N}\) we have \(z_{n}^{q_{0}\dotsm q_{n-1}}=z_{1}\in I_{1}\) so that \(z_{n}\in I_{1}^{1/\left(q_{0}\dotsm q_{n-1}\right)}=I_{n}\) by Step~4, hence \(\left(z_{n}\right)_{n\in\mathbb{N}}\in S\). This establishes \eqref{eq:SW}. As a consequence, and using the identity \eqref{eq:So}, we see that \(S\) is indeed a cross section:
\[G.S=\chi\left(G\right)\oplus S=\chi\left(G\right)\oplus\rho\left(\mathbb{R}\right)=\mathbb{S}.\]

With the compact interval \(W\subset\mathbb{R}\) as a window, form the cross section
\[Y_{W}=\left\{\Gamma\left(e_{G},w\right):w\in W\right\}.\]
Then using \eqref{eq:SW} we have the (pointwise) identity
\[\widetilde{\Phi}\left(Y_{W}\right)=\widetilde{\Phi}\left(\left\{\Gamma\left(e_{G},w\right):w\in W\right\} \right)=\left\{ \chi\left(e_{G}\right)\oplus\rho\left(w\right):w\in W\right\} =\rho\left(W\right)=S,\]
showing that \(\widetilde{\Phi}\) forms the desired transverse \(G\)-isomorphism.

\subsubsection*{\textbf{Step~7}}

In this final step, we will show that the desired transverse \(G\)-factor in the theorem is given by
\[\phi:\left(X,\mu,Y\right)\longrightarrow\left(\mathbb{S},m_{\mathbb{S}},S\right),\quad\phi\left(x\right)\coloneqq\left(f_{n}\left(x\right)\right)_{n\in\mathbb{N}}.\]

\smallskip

The \(G\)-equivariance of this \(\phi\) follows because \(f_{n}\) has eigencharacter \(\chi_{n}\) for every \(n\in\mathbb{N}\). Let us show that \(\phi\) is measure preserving, namely that \(\phi_{\ast}\mu=m_{\mathbb{S}}\).
For \(N\in\mathbb{N}\), let \(\Pi_{N}:\mathbb{S}\to\Pi_{N}(\mathbb{S})\) be the projection to the first \(N\) coordinates of \(\mathbb{S}\), and define the continuous surjective homomorphism
\[\psi_{N}:\mathbb{T}\longrightarrow \Pi_{N}(\mathbb{S}),\qquad\psi_{N}(z)\coloneqq\big(z^{\,q_{1}\cdots q_{n-1}}\big)_{n=1}^{N}.\]
Iterating the relations \(f_{n}=f_{n+1}^{q_{n}}\) gives \(f_{n}=f_{N}^{q_{n}\cdots q_{N-1}}\) for \(n\le N\), hence
\[\Pi_{N}\circ\phi=\psi_{N}\circ f_{N}.\]
Pushing forward \(\mu\) and using that \(f_{N}:(X,\mu)\to(\mathbb{T},m_{\mathbb{T}})\) is measure preserving, we get that
\[\left(\Pi_{N}\right)_{\ast}\left(\phi_{\ast}\mu\right)=\left(\psi_{N}\right)_{\ast}\big(\left(f_{N}\right)_{\ast}\mu\big)
=\left(\psi_{N}\right)_{\ast}m_{\mathbb{T}}
= m_{\Pi_{N}\left(\mathbb{S}\right)},
\]
where the last equality is the standard fact that the pushforward of Haar measure under a continuous surjective homomorphism of compact groups is Haar measure~\cite[Ex.~(15.17)(j)]{HewittRoss}. Since \(m_{\mathbb{S}}\) is characterized by \(\left(\Pi_{N}\right)_{\ast}m_{\mathbb{S}}=m_{\Pi_{N}\left(\mathbb{S}\right)}\) for all \(N\) (this is the inverse limit property;~\cite[Ex.~(15.17)(k)]{HewittRoss}), we conclude that \(\phi_{\ast}\mu=m_{\mathbb{S}}\).

In order to complete the proof that \(\phi\) is a transverse \(G\)-factor, we must show that
\[S_{\phi\left(x\right)}=\left(\phi^{-1}\left(S\right)\right)_{x}=Y_{x}\text{ for }\mu\text{-a.e. }x\in X.\]
(as in~\cite[Definition~5.8]{AvBjCuI}). The first equality holds because \(\phi\) is \(G\)-equivariant. For the second equality, denote \(Z\coloneqq\phi^{-1}\left(S\right)\), and by~\cite[Proposition 5.4]{AvBjCuI} it suffices to show that \(Z\) is \(\mu_{Y}\)-conull and that \(Y\) is \(\mu_{Z}\)-conull. In showing these two properties we recall that \(V_{n}\leq\ker\chi_{n}\) by \eqref{eq:kersV}, hence \(f_{n}^{-1}\left(I_{n}\right)\) is \(V_{n}\)-invariant for every \(n\in\mathbb{N}\).
\begin{description}[leftmargin=20pt, labelsep=1em]
    \item[\(\bullet\) \(Z\) is \(\mu_{Y}\)-conull] For every \(n\in\mathbb{N}\), since \(V_{n}.Y=f_{n}^{-1}\left(I_{n}\right)\) modulo \(\mu\), the set \(C_{n}\coloneqq V_{n}.Y\cap f_{n}^{-1}\left(I_{n}\right)\) is \(\mu\)-conull in \(V_{n}.Y\), hence
    \[\mu\left(C_{n}\right)=\mu\left(V_{n}.Y\right)=m_{G}\left(V_{n}\right)\cdot\mu_{Y}\left(Y\right).\]
    By Fubini's theorem there exist some \(v\in V_{n}\) and \(\mu_{Y}\)-conull subset \(Y_{n}\subseteq Y\) such that \(v.Y_{n}\subseteq C_{n}\subseteq f_{n}^{-1}\left(I_{n}\right)\). Since \(f_{n}^{-1}\left(I_{n}\right)\) is \(V_{n}\)-invariant, it follows that
    \[Y_{n}\subseteq Y\cap f_{n}^{-1}\left(Y\right)\subseteq V_{n}.Y\cap f_{n}^{-1}\left(Y\right).\]
    Intersecting over \(n\in\mathbb{N}\), we obtain the \(\mu_{Y}\)-conull set
    \[\bigcap\nolimits_{n\in\mathbb{N}}Y_{n}\subseteq\bigcap\nolimits_{n\in\mathbb{N}}V_{n}.Y\cap\bigcap\nolimits_{n\in\mathbb{N}}f_{n}^{-1}\left(I_{n}\right)\subseteq Y\cap\phi^{-1}\left(S\right)\subseteq Z,\]
    concluding that \(Z\) is \(\mu_{Y}\)-conull.
    \item[\(\bullet\) \(Y\) is \(\mu_{Z}\)-conull] For every \(n\in\mathbb{N}\), since \(Z\subseteq f_{n}^{-1}\left(I_{n}\right)\) and \(f_{n}^{-1}\left(I_{n}\right)\) is \(V_{n}\)-invariant, we have \(V_{n}.Z\subseteq f_{n}^{-1}\left(I_{n}\right)\). Since \(V_{n}.Y=f_{n}^{-1}\left(I_{n}\right)\) modulo \(\mu\), the set \(D_{n}\coloneqq V_{n}.Z\cap V_{n}.Y\) is \(\mu\)-conull in \(V_{n}.Z\). By Fubini's theorem there exist some \(v\in V_{n}\) and \(\mu_{Z}\)-conull subset \(Z_{n}\subseteq Z\) such that \(v.Z_{n}\subseteq D_{n}\subseteq V_{n}.Y\), and hence \(Z_{n}\subseteq V_{n}.Y\). Intersecting over \(n\in\mathbb{N}\), we obtain the \(\mu_{Z}\)-conull set
    \[\bigcap\nolimits_{n\in\mathbb{N}}Z_{n}\subseteq\bigcap\nolimits_{n\in\mathbb{N}}V_{n}.Y=Y,\]
    concluding that \(Y\) is \(\mu_{Z}\)-conull.
\end{description}

\section{A characterization of generalized Farey fractions}

\subsection{Strong F\o lner sequences}\label{sct:Banach}

Let \(G\) be an lcsc abelian group. An increasing exhaustion of \(G\) by compact sets \(\mathcal{F}=\left(F_{n}\right)_{n\in\mathbb{N}}\) is called a {\bf F\o lner sequence} if for every \(g\in G\),
\[\frac{m_{G}\left(F_{n}\triangle F_{n}g\right)}{m_{G}\left(F_{n}\right)}\xrightarrow[n\to\infty]{}0,\]
and it is called a {\bf strong F\o lner sequence}, if the following conditions hold:
\begin{enumerate}
    \item For every compact set \(K\subset G\),
    \[\frac{m_{G}\left(F_{n}K\right)}{m_{G}\left(F_{n}\right)}\xrightarrow[n\to\infty]{}1.\]
    \item \(\mathcal{F}\) is {\bf \(U\)-adapted}: there exists an identity neighborhood \(e_{G}\in U\subset G\) such that
    \[\frac{m_{G}\left(\bigcap\nolimits_{u\in U}F_{n}u^{-1}\right)}{m_{G}\left(F_{n}\right)}\xrightarrow[n\to\infty]{}1.\]
\end{enumerate}
Every abelian group admits a strong F\o lner sequence (see~\cite[\S2.3]{BjFi2024} and the references therein). For a strong F\o lner sequence \(\mathcal{F}=\left(F_{n}\right)_{n\in\mathbb{N}}\) in \(G\), the {\bf asymptotic density} of a locally finite set \(P\subset G\) is
\[d_{\mathcal{F}}\left(P\right)\coloneqq\varlimsup_{n\to\infty}\frac{\left|P\cap F_{n}\right|}{m_{G}\left(F_{n}\right)},\]
and the {\bf upper Banach density} of a locally finite set \(P\subset G\) is
\[d^{\ast}\left(P\right)=\sup\left\{ d_{\mathcal{F}}\left(P\right):\mathcal{F}\text{ is a strong F\o}\text{lner sequence in }G\right\}.\]

We record here a basic lemma on strong F\o lner sequences.

\begin{lem}\label{lem:foln}
Let \(\left(F_{n}\right)_{n\in\mathbb{N}}\) be a \(U\)-adapted strong F\o lner sequence and \(K\subseteq U\) a compact identity neighborhood. Then:
\begin{enumerate}
     \item The sequence \(\big(F_{n}K\big)_{n\in\mathbb{N}}\) is a strong F\o lner sequence in \(G\).
     \item Define \(\widetilde{F}_{n}\coloneqq\bigcap\nolimits_{u\in U}F_{n}u^{-1}\) for \(n\in\mathbb{N}\). The sequence \(\big(\widetilde{F}_{n}K\big)_{n\in\mathbb{N}}\) is a F\o lner sequence in \(G\) that satisfies \(m_{G}\big(\widetilde{F}_{n}K\big)/m_{G}\left(F_{n}\right)\xrightarrow[n\to\infty]{}1\).
\end{enumerate}
\end{lem}

\begin{proof}
For (1) we refer to~\cite[Lemma 2.8]{BjFi2024}. In order to prove (2), for \(n\in\mathbb{N}\) put \(K_{n}\coloneqq\widetilde{F}_{n}K\), and note that the sequence \(\left(K_{n}\right)_{n\in\mathbb{N}}\) satisfies two properties. (i): \(K_{n}=\widetilde{F}_{n}K\subseteq\widetilde{F}_{n}U\subseteq F_{n}\), and (ii):
\[1\xleftarrow[\infty\leftarrow n]{}\frac{m_{G}\big(\widetilde{F}_{n}\big)}{m_{G}\left(F_{n}\right)}\leq\frac{m_{G}\big(K_{n}\big)}{m_{G}\left(F_{n}\right)}\leq\frac{m_{G}\left(F_{n}K\right)}{m_{G}\left(F_{n}\right)}\xrightarrow[n\to\infty]{}1,\]
(since \(\left(F_{n}\right)_{n\in\mathbb{N}}\) is a strong F\o lner sequence). We now claim that any sequence \(\left(K_{n}\right)_{n\in\mathbb{N}}\) satisfying (i) and (ii) forms a F\o lner sequence. To see this let \(g\in G\) be arbitrary, and for each \(n\in\mathbb{N}\) we argue that
\begin{equation}\label{eq:foln}
K_{n}\triangle K_{n}g\subseteq\left(F_{n}\triangle F_{n}g\right)\cup\left(F_{n}\backslash K_{n}\right)\cup\left(F_{n}\backslash K_{n}\right)g.
\end{equation}
Indeed, for an arbitrary \(x\in K_{n}\backslash K_{n}g\), if \(x\notin F_{n}g\) then \(x\in F_{n}\backslash F_{n}g\subseteq F_{n}\triangle F_{n}g\), and if \(x\in F_{n}g\) then \(x\in\left(F_{n}g\right)\setminus\left(K_{n}g\right)=\left(F_{n}\setminus K_{n}\right)g\). This shows that \(K_{n}\backslash K_{n}g\subseteq\left(F_{n}\triangle F_{n}g\right)\cup\left(F_{n}\backslash K_{n}\right)g\). Similarly one shows that \(K_{n}g\backslash K_{n}\subseteq\left(F_{n}\triangle F_{n}g\right)\cup\left(F_{n}\backslash K_{n}\right)\), concluding \eqref{eq:foln}. Then it follows from \eqref{eq:foln} and the properties (i) and (ii) of \(\left(K_{n}\right)_{n\in\mathbb{N}}\) that
\begin{align*}
\frac{m_{G}\left(K_{n}\triangle K_{n}g\right)}{m_{G}\left(K_{n}\right)}
&\leq\frac{m_{G}\left(F_{n}\triangle F_{n}g\right)}{m_{G}\left(K_{n}\right)}+2\cdot\frac{m_{G}\left(F_{n}\backslash K_{n}\right)}{m_{G}\left(K_{n}\right)}\\
&=\frac{m_{G}\left(F_{n}\right)}{m_{G}\left(K_{n}\right)}\cdot\frac{m_{G}\left(F_{n}\triangle F_{n}g\right)}{m_{G}\left(F_{n}\right)}+2\cdot\Big(\frac{m_{G}\left(F_{n}\right)}{m_{G}\left(K_{n}\right)}-1\Big)\xrightarrow[n\to\infty]{}1\cdot0+2\cdot\left(1-1\right)=0.\qedhere
\end{align*}
\end{proof}

\subsection{Upper Banach density of cut--and--project sets}

Let \(\left(G,H;\Gamma\right)\) be an abelian cut--and--project scheme and \(W\subset H\) a \hyperlink{window}{window} such that \(W-W\) is Jordan measurable. Let the cut--and--project set
\[P\coloneqq\operatorname{proj}_{G}\left(\Gamma\cap\left(G\times W\right)\right)\subset G.\]

\begin{prop}\label{prop:denscap}
For every strong F\o lner sequence \(\mathcal{F}\) in \(G\), the asymptotic density of \(P\) is
\[d_{\mathcal{F}}\left(P\right)=\operatorname{covol}_{G\times H}\left(\Gamma\right)^{-1}\cdot m_{H}\left(W\right),\]
and the asymptotic density of \(P-P\) satisfies
\[\operatorname{covol}_{G\times\mathbb{R}}\left(\Gamma\right)^{-1}\cdot m_{H}\left(W^{o}-W^{o}\right)\leq d_{\mathcal{F}}\left(P-P\right)\leq\operatorname{covol}_{G\times\mathbb{R}}\left(\Gamma\right)^{-1}\cdot m_{H}\left(W-W\right).\]
In particular, when \(H=\mathbb{R}\) and \(W\subset\mathbb{R}\) is a compact interval,
\[d_{\mathcal{F}}\left(P-P\right)=2\cdot\operatorname{covol}_{G\times\mathbb{R}}\left(\Gamma\right)^{-1}\cdot m_{\mathbb{R}}\left(W\right)=2\cdot d_{\mathcal{F}}\left(P\right).\]
\end{prop}

\begin{proof}
Let \(\mathcal{F}=\left(F_{n}\right)_{n\in\mathbb{N}}\) be a strong F\o lner sequence. Since \(\Gamma\) projects injectively to \(G\), we may write
\begin{equation}\label{eq:pdens}
\left|P\cap F_{n}\right|=\left|\operatorname{proj}_{G}\left(\Gamma\cap\left(F_{n}\times W\right)\right)\right|=\sum\nolimits_{\left(\gamma_{1},\gamma_{2}\right)\in\Gamma}\mathbf{1}_{F_{n}}\left(\gamma_{1}\right)\cdot \mathbf{1}_{W}\left(\gamma_{2}\right),\quad n\in\mathbb{N}.
\end{equation}
Pick a compact symmetric identity neighborhood \(e_{G}\in K\subset G\) and a compactly supported continuous function \(w:G\to\mathbb{R}_{\geq 0}\) with \(m_{G}\left(w\right)=1\) and \(\operatorname{supp}\left(w\right)\subseteq K\). Using~\cite[Lemma 5.11]{BjHa2022} we have
\[\mathbf{1}_{F_{n}}\leq\mathbf{1}_{F_{n}K}\ast w,\]
where \(\ast\) is the convolution on \(G\). It follows from \eqref{eq:pdens} and Fubini's theorem that
\begin{align*}
\left|P\cap F_{n}\right|
&\leq\sum\nolimits_{\left(\gamma_{1},\gamma_{2}\right)\in\Gamma}\left(\mathbf{1}_{F_{n}K}\ast w\right)\left(\gamma_{1}\right)\cdot\mathbf{1}_{W}\left(\gamma_{2}\right)\\
&=\int_{F_{n}K}\Big(\sum\nolimits_{\left(\gamma_{1},\gamma_{2}\right)\in\Gamma}w\left(g^{-1}\gamma_{1}\right)\cdot\mathbf{1}_{W}\left(\gamma_{2}\right)\Big)dm_{G}\left(g\right)\coloneqq\int_{F_{n}K}\varphi\left(\Gamma\left(g,e_{H}\right)\right)dm_{G}\left(g\right),
\end{align*}
where we put
\[\varphi:\Gamma\backslash\left(G\times H\right)\longrightarrow\mathbb{R}_{\geq 0},\quad \varphi\left(\Gamma\left(g,h\right)\right)\coloneqq\sum\nolimits_{\left(\gamma_{1},\gamma_{2}\right)\in\Gamma}w\left(g^{-1}\gamma_{1}\right)\cdot\mathbf{1}_{W}\left(h^{-1}\gamma_{2}\right).\]
We then obtain, using that \(\mathcal{F}\) is a strong F\o lner sequence,
\[d_{\mathcal{F}}\left(P\right)\leq\varlimsup_{n\to\infty}\frac{1}{m_{G}\left(F_{n}K\right)}\int_{F_{n}K}\varphi\left(\Gamma\left(g,e_{H}\right)\right)dm_{G}\left(g\right)=\varlimsup_{n\to\infty}\frac{1}{m_{G}\left(F_{n}K\right)}\int_{F_{n}K}g.\varphi\left(\Gamma\left(e_{G},e_{H}\right)\right)dm_{G}\left(g\right),\]
where in the last equality we used that \(\varphi\left(\Gamma\left(g,e_{H}\right)\right)=g.\varphi\left(\Gamma\left(e_{G},e_{H}\right)\right)\) for every \(g\in G\). Recall that \(\left(F_{n}K\right)_{n\in\mathbb{N}}\) is a F\o lner sequence by Lemma~\ref{lem:foln}(1), and that the \(G\)-space \(\Gamma\backslash\left(G\times H\right)\) is uniquely ergodic. Observe also that \(\varphi\) is Riemann integrable, since \(\mathbf{1}_{W}\) is Riemann integrable and \(w\) is continuous and compactly supported while \(\Gamma\) is uniformly discrete. Then by the Portmanteau theorem, the ergodic averages of \(\varphi\) along \(\left(F_{n}K\right)_{n\in\mathbb{N}}\) converge at the point \(\Gamma\left(e_{G},e_{H}\right)\in\Gamma\backslash\left(G\times H\right)\), concluding that
\[d_{\mathcal{F}}\left(P\right)\leq m_{\Gamma\backslash\left(G\times H\right)}\left(\varphi\right)=\operatorname{covol}_{G\times H}\left(\Gamma\right)^{-1}\cdot\int_{G}w\left(g\right)dm_{G}\left(g\right)\cdot m_{H}\left(W\right)=\operatorname{covol}_{G\times H}\left(\Gamma\right)^{-1}\cdot m_{H}\left(W\right).\]
To deduce the converse inequality, with \(\mathcal{F}\) being \(U\)-adapted, pick a compact symmetric identity neighborhood \(e_{G}\in K\subseteq U\) with \(KK\subseteq U\), and a compactly supported continuous function \(w:G\to\mathbb{R}_{\geq 0}\) with \(m_{G}\left(w\right)=1\) and \(\operatorname{supp}\left(w\right)\subseteq K\). For \(n\in\mathbb{N}\) put \(\widetilde{F}_{n}\coloneqq\bigcap\nolimits_{u\in U}F_{n}u^{-1}\), and using~\cite[Lemma 5.11]{BjHa2022} together with that \(\widetilde{F}_{n}KK\subseteq\widetilde{F}_{n}U\subseteq F_{n}\), we have
\[\mathbf{1}_{\widetilde{F}_{n}}\leq\mathbf{1}_{\widetilde{F}_{n}K}\ast w\leq\mathbf{1}_{\widetilde{F}_{n}KK}\leq\mathbf{1}_{\widetilde{F}_{n}U}\leq\mathbf{1}_{F_{n}}.\]
It then follows from \eqref{eq:pdens}, with the same function \(\varphi\) as above, that
\[\left|P\cap F_{n}\right|\geq\sum\nolimits_{\left(\gamma_{1},\gamma_{2}\right)\in\Gamma}\big(\mathbf{1}_{\widetilde{F}_{n}K}\ast w\big)\left(\gamma_{1}\right)\cdot\mathbf{1}_{W}\left(\gamma_{2}\right)=\int_{\widetilde{F}_{n}K}\varphi\left(\Gamma\left(g,e_{H}\right)\right)dm_{G}\left(g\right).\]
By Lemma~\ref{lem:foln}(2), \(\big(\widetilde{F}_{n}K\big)_{n\in\mathbb{N}}\) is a F\o lner sequence with \(m_{G}\big(\widetilde{F}_{n}K\big)/m_{G}\left(F_{n}\right)\xrightarrow[n\to\infty]{}1\), so by the same reasoning as above it follows that
\[d_{\mathcal{F}}\left(P\right)\geq\varlimsup_{n\to\infty}\frac{m_{G}\big(\widetilde{F}_{n}K\big)}{m_{G}\left(F_{n}\right)}\cdot m_{\Gamma\backslash\left(G\times H\right)}\left(\phi\right)=\operatorname{covol}_{G\times H}\left(\Gamma\right)^{-1}\cdot m_{H}\left(W\right).\]
We thus deduce that \(d_{\mathcal{F}}\left(P\right)=\operatorname{covol}_{G\times H}\left(\Gamma\right)\cdot m_{H}\left(W\right)\), proving (1).

In order to establish the bounds for \(d_{\mathcal{F}}\left(P-P\right)\), it suffices to show generally that
\begin{equation}\label{eq:incW}
\operatorname{proj}_{G}\left(\Gamma\cap\left(G\times\left(W^{o}-W^{o}\right)\right)\right)\subseteq P-P\subseteq\operatorname{proj}_{G}\left(\Gamma\cap\left(G\times\left(W-W\right)\right)\right),
\end{equation}
from which it follows, using the formula we have already proved, that
\[\operatorname{covol}_{G\times H}\left(\Gamma\right)^{-1}\cdot m_{H}\left(W^{o}-W^{o}\right)\leq d_{\mathcal{F}}\left(P-P\right)\leq\operatorname{covol}_{G\times H}\left(\Gamma\right)^{-1}\cdot m_{H}\left(W-W\right).\]
We then verify \eqref{eq:incW}. The second inclusion is straightforward. For the first inclusion we use that \(\Gamma\) projects densely to \(\mathbb{R}\): suppose \(g\in\operatorname{proj}_{G}\left(\Gamma\cap\left(G\times\left(W^{o}-W^{o}\right)\right)\right)\) so that \(\left(g,t\right)\in\Gamma\) for some \(t\in W^{o}-W^{o}\). Pick \(s\in W^{o}\) with \(t+s\in W^{o}\) and let \(\delta>0\). Since \(\Gamma\) projects densely to \(\mathbb{R}\) we can pick \(\left(g^{\prime},t^{\prime}\right)\in\Gamma\) with \(\left|t^{\prime}-s\right|<\delta\), and therefore \(t^{\prime}\in W\) so that \(g^{\prime}\in P\). Note that \(\left|t+t^{\prime}-\left(t+s\right)\right|=\left|t^{\prime}-s\right|<\delta\), and therefore if \(\delta\) is sufficiently small we deduce that \(t+t^{\prime}\in W\). It now readily follows that
\[g=\left(g+g^{\prime}\right)-g\in P-P.\qedhere\]
\end{proof}

\subsection{The Furstenberg correspondence principle for intensity and intersection covolume}

Let \(P\subset G\) be such that \(P-P\) is uniformly discrete. Denote by \(\Omega_{P}\) the \(G\)-hull of \(P\), which is the closure of the orbit of \(P\) in the space of closed subsets of \(G\), and by \(\Omega_{P}^{\times}=\Omega_{P}\backslash\{\emptyset\}\) the punctured \(G\)-hull of \(P\). Then \(\Omega_{P}^{\times}\) is a Borel \(G\)-space, and it admits the natural cross section
\[Y_{P}=\left\{P^{\prime}\in\Omega_{P}^{\times}:e_{G}\in P^{\prime}\right\}.\]
For details, we refer to~\cite[\S2.2]{BjFi2024}. With the goal of deducing Theorem~\ref{mthm:farey} from Theorems~\ref{mthm:mineq} and~\ref{mthm:mineqext}, we use the transverse version of the Furstenberg correspondence principle due to the second author and Fish~\cite{BjFi2024}:

\begin{prop}\label{prop:furst}
Let \(G\) be a \hyperlink{classQ}{class \(\mathcal{Q}\)} group. Then for every \(P\subset G\) with \(d^{\ast}\left(P\right)>0\) and such that \(P-P\) is uniformly discrete, there exists an ergodic \(G\)-invariant probability measure \(\mu\) on \(\Omega_{P}^{\times}\) satisfying
\begin{equation}\label{eq:furst}
d^{\ast}\left(P-P\right)\geq I_{\mu}\left(Y_{P}\right)\geq 2\cdot\iota_{\mu}\left(Y_{P}\right)=2\cdot d^{\ast}\left(P\right).
\end{equation}
\end{prop}

\begin{proof}
By~\cite[Theorem 3.1]{BjFi2024} there exists an ergodic \(G\)-invariant probability measure \(\mu\) on \(\Omega_{P}^{\times}\) satisfying
\[\iota_{\mu}\left(Y_{P}\right)=d^{\ast}\left(P\right).\]
We show that this \(\mu\) automatically satisfies \eqref{eq:furst}. In fact, we will show that both inequalities in \eqref{eq:furst} holds for any ergodic \(G\)-invariant probability measure \(\mu\) on \(\Omega_{P}^{\times}\). Indeed, the middle inequality in \eqref{eq:furst} holds generally due to Theorem~\ref{mthm:mineq}. For the first inequality in \eqref{eq:furst}, we will show that for every strong F\o lner sequence \(\mathcal{F}=\left(F_{n}\right)_{n\in\mathbb{N}}\) in \(G\),
\begin{equation}\label{eq:firstineq}
d_{\mathcal{F}}\left(P-P\right)\geq\varliminf_{n\to\infty}\frac{\left|\left(P-P\right)\cap F_{n}\right|}{m_{G}\left(F_{n}\right)}\geq I_{\mu}\left(Y_{P}\right).
\end{equation}

Pick an identity neighborhood \(e_{G}\in U\subset G\) such that \(\left(P-P\right)^{2}\cap U=\left\{ e_{G}\right\}\) (as \(P-P\) is uniformly discrete), and a symmetric compact identity neighborhood \(e_{g}\in V\) with \(V^{4}\subseteq U\). We then have
\[\left(P+V\right)-\left(P+V\right)=\bigsqcup\nolimits_{p\in P-P}\left(p+V-V\right)\]
as a disjoint union, and consequently, for every \(U\)-adapted strong F\o lner sequence \(\mathcal{F}=\left(F_{n}\right)_{n\in\mathbb{N}}\), we have
\[m_{G}\big(\left(\left(P+V\right)-\left(P+V\right)\right)\cap\widetilde{F}_{n}\big)\leq m_{G}\left(V-V\right)\cdot\left|\left(P-P\right)\cap F_{n}\right|.\]
Since \(P-P\) is discrete, by~\cite[Lemma 2.4]{BjFi2024} for every \(P^{\prime}\in\Omega_{P}^{\times}\) it holds that \(P^{\prime}-P^{\prime}\subseteq P-P\), and thus
\[m_{G}\big(\left(\left(P^{\prime}+V\right)-\left(P^{\prime}+V\right)\right)\cap\widetilde{F}_{n}\big)\leq m_{G}\left(V-V\right)\cdot\left|\left(P-P\right)\cap F_{n}\right|.\]
We then obtain the bound
\begin{equation}\label{eq:genbnd}
\begin{aligned}
d_{\mathcal{F}}\left(P-P\right)	&=\varlimsup_{n\to\infty}\frac{\left|\left(P-P\right)\cap F_{n}\right|}{m_{G}\left(F_{n}\right)}\\
&\geq\frac{1}{m_{G}\left(V-V\right)}\cdot\varlimsup_{n\to\infty}\frac{m_{G}\big(\left(\left(P^{\prime}+V\right)-\left(P^{\prime}+V\right)\right)\cap\widetilde{F}_{n}\big)}{m_{G}\left(F_{n}\right)},\quad P^{\prime}\in\Omega_{P}^{\times}.
\end{aligned}
\end{equation}

Consider \(\Omega_{P}^{\times}\times\Omega_{P}^{\times}\) as a Borel \(G\)-space (\(G\) acts on the first coordinate), and the cross section
\[Y_{P}^{\left[2\right]}\coloneqq G\diag\left(Y_{P}\times Y_{P}\right)=\{\left(P^{\prime},P^{\prime\prime}\right)\in\Omega_{P}^{\times}\times\Omega_{P}^{\times}:P^{\prime}\cap P^{\prime\prime}\neq\emptyset\}.\]
The return time sets are of the form
\[\left(P^{\prime}+V\right)-\left(P^{\prime}+V\right)=\big(\left(V-V\right)\first Y_{P}^{\left[2\right]}\big)_{\left(P^{\prime},P^{\prime}\right)},\quad P^{\prime}\in\Omega_{P}^{\times}.\]
We next claim that \(\left(V-V\right)\first Y_{P}^{[2]}\) is an open set (in the Chabauty--Fell topology). To see this, define
\[R\coloneqq\{\left(g,h\right)\in G\times G:g-h\in V-V\},\]
and note that \(R\) is open in \(G\times G\). It is now routine to verify that
\[\left(V-V\right)\first Y_{P}^{[2]}=\bigl\{\left(P^{\prime},P^{\prime\prime}\right)\in\Omega_{P}^{\times}\times\Omega_{P}^{\times}:(P^{\prime}\times P^{\prime\prime})\cap R\neq\emptyset\bigr\},\]
hence by the definition of the Chabauty--Fell topology this set is indeed open.

Let now \(\mu\) be an arbitrary ergodic \(G\)-invariant probability measure on \(\Omega_{P}^{\times}\). Fix a compact identity neighborhood \(K\subseteq U\). By Lemma~\ref{lem:foln}(2), the sequence \(\big(\widetilde{F}_{n}K\big)_{n\in\mathbb{N}}\) is a F\o lner sequence in \(G\) and satisfies \(m_{G}\big(\widetilde{F}_{n}K\big)/m_{G}(F_{n})\xrightarrow[n\to\infty]{}1\).  
After passing to a tempered subsequence if necessary, the pointwise ergodic theorem yields a \(\mu\)-conull set \(\Omega_{o}\subseteq\Omega_{P}^{\times}\) such that, for every \(P^{\prime}\in\Omega_{o}\),
\[\frac{1}{m_{G}\big(F_{n}\big)}\int_{\widetilde{F}_{n}K}\delta_{g.P^{\prime}}\,dm_{G}(g)
\xrightarrow[n\to\infty]{w^{\ast}}\mu,\]
and hence also
\[\frac{1}{m_{G}\big(F_{n}\big)}\int_{\widetilde{F}_{n}K}
\delta_{g.P^{\prime}}\otimes\delta_{P^{\prime}}\,dm_{G}(g)
\xrightarrow[n\to\infty]{w^{\ast}}
\mu\otimes\delta_{P^{\prime}}.\]
Since \(\left(V-V\right)\first Y_{P}^{[2]}\) is open, the Portmanteau theorem yields
\[\varliminf_{n\to\infty}\frac{1}{m_{G}\big(F_{n}\big)}
\int_{\widetilde{F}_{n}K}\mathbf{1}_{\left(V-V\right)\first Y_{P}^{\left[2\right]}}(g.P^{\prime},P^{\prime})\,dm_{G}\left(g\right)\geq\mu\otimes\delta_{P^{\prime}}\big(\left(V-V\right)\first Y_{P}^{\left[2\right]}\big),\]
and since \(\mathbf{1}_{\left(V-V\right)\first Y_{P}^{\left[2\right]}}(g.P^{\prime},P^{\prime})
=\mathbf{1}_{\left(P^{\prime}+V\right)-\left(P^{\prime}+V\right)}(g)\), we obtain
\[\varliminf_{n\to\infty}\frac{m_{G}\big(\big(\left(P^{\prime}+V\right)-\left(P^{\prime}+V\right)\big)\cap\widetilde{F}_{n}K\big)}{m_{G}\big(F_{n}\big)}\geq\mu\otimes\delta_{P^{\prime}}\big(\left(V-V\right)\first Y_{P}^{\left[2\right]}\big).\]

By the definition of strong F\o lner sequence and by Lemma~\ref{lem:foln}(2), we have
\[\Big|\frac{m_{G}\big(B\cap\widetilde F_{n}K\big)}{m_{G}(F_{n})}-\frac{m_{G}\big(B\cap \widetilde F_{n}\big)}{m_{G}(F_{n})}\Big|\leq\frac{m_{G}(\widetilde F_{n}K)}{m_{G}(F_{n})}-\frac{m_{G}(\widetilde F_{n})}{m_{G}(F_{n})}\xrightarrow[n\to\infty]{}0\]
for any measurable set \(B\subset G\). Therefore we conclude that
\[\varliminf_{n\to\infty}\frac{m_{G}\big(\big(\left(P^{\prime}+V\right)-\left(P^{\prime}+V\right)\big)\cap \widetilde F_{n}\big)}{m_{G}\left(F_{n}\right)}\geq\mu\otimes\delta_{P^{\prime}}\big(\left(V-V\right)\first Y_P^{[2]}\big).\]
Integrating this inequality \(d\mu\left(P^{\prime}\right)\) over \(\Omega_{o}\) and using \eqref{eq:genbnd}, we deduce the desired inequality \eqref{eq:firstineq}:
\begin{align*}
d_{\mathcal{F}}\left(P-P\right)
&\geq\frac{1}{m_{G}\left(V-V\right)}\cdot\mu^{\otimes2}\big(\left(V-V\right)\first Y_{P}^{\left[2\right]}\big)\\
&=\frac{1}{m_{G}\left(V-V\right)}\cdot m_{G}\left(V-V\right)\cdot\mu^{\left[2\right]}\big(Y_{P}^{\left[2\right]}\big)=\mu^{\left[2\right]}\big(Y_{P}^{\left[2\right]}\big)=I_{\mu}\left(Y_{P}\right).\qedhere
\end{align*}
\end{proof}

\subsection{Generalized Farey fractions}\label{sct:gff}

Let us start by recalling some basics and conventions about the finite adeles. We thus have the group \(\mathbb{A}_{\mathrm{fin}}\coloneqq\prod_{p\text{ prime}}^{\prime}\left(\mathbb{Q}_{p},\mathbb{Z}_{p}\right)\) of finite adeles, and its compact open subgroup \(\mathbb{Z}_{\mathrm{fin}}\coloneqq\prod_{p\text{ prime}}\mathbb{Z}_{p}<\mathbb{A}_{\mathrm{fin}}\) of finite adelic integers. Consider also the adeles group, \(\mathbb{A}\coloneqq\mathbb{A}_{\mathrm{fin}}\times\mathbb{R}\). Let the Haar measure \(m_{\mathbb{A}_{\mathrm{fin}}}=\bigotimes_{p\text{ prime}}^{\prime}m_{\mathbb{Q}_{p}}\), where each \(m_{\mathbb{Q}_{p}}\) is normalized so that \(m_{\mathbb{Q}_{p}}\left(\mathbb{Z}_{p}\right)=1\), and let the Haar measure \(m_{\mathbb{A}}=m_{\mathbb{A}_{\mathrm{fin}}}\otimes m_{\mathbb{R}}\), for which \(m_{\mathbb{A}}\left(\mathbb{Z}_{\mathrm{fin}}\times\left[0,1\right)\right)=1\). Let \(\left\Vert\cdot\right\Vert_{\mathrm{fin}}\) be the adelic norm on \(\mathbb{A}_{\mathrm{fin}}\), given by
\[\left\Vert u\right\Vert_{\mathrm{fin}}=\prod\nolimits_{p\text{ prime}}\left|u_{p}\right|_{p},\quad u=\left(u_{p}\right)_{p\text{ prime}}\in\mathbb{A}_{\mathrm{fin}},\]
where each \(\left|\cdot\right|_{p}\) is the \(p\)-adic norm on \(\mathbb{Q}_{p}\).

The canonical lattice \(\mathbb{Q}<\mathbb{A}\) admits the fundamental domain \(\mathbb{Z}_{\mathrm{fin}}\times\left[0,1\right)\), and therefore \(\operatorname{covol}_{\mathbb{A}}\left(\mathbb{Q}\right)=1\). Consider the lattice \(\mathbb{Q}\cdot\left(u,t\right)<\mathbb{A}\) for some \(u=\left(u_{p}\right)_{p\text{ prime}}\in\mathbb{Z}_{\mathrm{fin}}\) and \(t\in\mathbb{R}\). It admits the fundamental domain \(\mathbb{Z}_{\mathrm{fin}}\cdot u\times\left[0,\left|t\right|\right)\), and therefore
\[\operatorname{covol}_{\mathbb{A}}\left(\mathbb{Q}\cdot\left(u,t\right)\right)=m_{\mathbb{A}_{\mathrm{fin}}}\big(\mathbb{Z}_{\mathrm{fin}}\cdot u\big)\cdot m_{\mathbb{R}}\left(\left[0,\left|t\right|\right)\right)=\left\Vert u\right\Vert_{\mathrm{fin}}\cdot\left|t\right|.\]

Let us now redefine and update the notations for the generalized Farey fractions.

\begin{defn}
The {\bf generalized Farey fractions} associated with a compact interval \(W\subset\mathbb{R}\), is the cut--and--project set, defined in the cut--and--project scheme \(\left(\mathbb{A}_{\mathrm{fin}},\mathbb{R};\mathbb{Q}\right)\), given by
\[\mathfrak{F}\left(W\right)\coloneqq\operatorname{proj}_{\mathbb{A}_{\mathrm{fin}}}\left(\mathbb{Q}\cap\left(\mathbb{A}_{\mathrm{fin}}\times W\right)\right)=\left\{ \iota\left(q\right):q\in\mathbb{Q}\cap W\right\}.\]
A {\bf dilated generalized Farey fractions} by a finite adelic integer \(u\in\mathbb{Z}_{\mathrm{fin}}\) takes the form
\[u\cdot\mathfrak{F}\left(W\right)=u\cdot\operatorname{proj}_{\mathbb{A}_{\mathrm{fin}}}\left(\mathbb{Q}\cap\left(\mathbb{A}_{\mathrm{fin}}\times W\right)\right)=\operatorname{proj}_{\mathbb{A}_{\mathrm{fin}}}\left(\mathbb{Q}\left(u,1\right)\cap\left(\mathbb{A}_{\mathrm{fin}}\times W\right)\right).\]
\end{defn}

Let us make the observation that for every \(u\in\mathbb{Z}_{\mathrm{fin}}\) and every compact interval \(W\subset\mathbb{R}\), it holds that
\begin{equation}\label{eq:twoendpts}
u\cdot\mathfrak{F}\left(W\right)-u\cdot\mathfrak{F}\left(W\right)=u\cdot\mathfrak{F}\left(W-W\right)\text{ with at most two exceptional points}.
\end{equation}
More precisely, \(u\cdot\mathfrak{F}\left(W\right)-u\cdot\mathfrak{F}\left(W\right)\subseteq u\cdot\mathfrak{F}\left(W-W\right)\), and \(u\cdot\mathfrak{F}\left(W-W\right)\backslash\left(u\cdot\mathfrak{F}\left(W\right)-u\cdot\mathfrak{F}\left(W\right)\right)\) has cardinality at most two, where those two exceptional points arise from the potentially-rational two endpoints of the interval \(W\).

We can now compute the formula \eqref{eq:densff} for the upper Banach density of a dilated generalized Farey fractions. Recalling that \(\operatorname{covol}_{\mathbb{A}}\left(\mathbb{Q}\left(u,1\right)\right)=\left\Vert u\right\Vert_{\mathrm{fin}}\), from Proposition~\ref{prop:denscap} (in the case \(H=\mathbb{R}\)) we obtain the first part of formula \eqref{eq:densff}:
\[d^{\ast}\left(u\cdot\mathfrak{F}\left(W\right)\right)=\operatorname{covol}_{\mathbb{A}_{\mathrm{fin}}}\left(\mathbb{Q}\left(u,1\right)\right)^{-1}\cdot m_{\mathbb{R}}\left(W\right)=\left\Vert u\right\Vert _{\mathrm{fin}}^{-1}\cdot m_{\mathbb{R}}\left(W\right).\]
As for the second part of the formula \eqref{eq:densff}, using the observation \eqref{eq:twoendpts} we have
\begin{align*}
d^{\ast}\left(u\cdot\mathfrak{F}\left(W\right)-u\cdot\mathfrak{F}\left(W\right)\right)
&=d^{\ast}\left(u\cdot\mathfrak{F}\left(W-W\right)\right)\\
&=\left\Vert u\right\Vert_{\mathrm{fin}}^{-1}\cdot m_{\mathbb{R}}\left(W-W\right)=2\cdot\left\Vert u\right\Vert_{\mathrm{fin}}^{-1}\cdot m_{\mathbb{R}}\left(W\right).
\end{align*}

\subsection{Final proof of Theorem~\ref{mthm:farey}}

The general inequality in Theorem~\ref{mthm:farey} follows from Proposition~\ref{prop:furst}, so we prove the second part. Suppose \(d^{\ast}\left(P-P\right)=2\cdot d^{\ast}\left(P\right)>0\). Pick \(\mu\) as in Proposition~\ref{prop:furst}, and then from \eqref{eq:furst} it follows that the transverse \(\mathbb{A}_{\mathrm{fin}}\)-space \(\left(\Omega_{P}^{\times},\mu,Y_{P}\right)\) satisfies \(I_{\mu}\left(Y_{P}\right)=2\cdot\iota_{\mu}\left(Y_{P}\right)\). Since \(\mathbb{A}_{\mathrm{fin}}\) is class \(\mathcal{Q}\), by Theorem~\ref{mthm:mineqext} there is a cut--and--project scheme \(\left(\mathbb{A}_{\mathrm{fin}},\mathbb{R};\Gamma\right)\) and a compact interval \(W_{o}\subset\mathbb{R}\), together with a transverse \(\mathbb{A}_{\mathrm{fin}}\)-factor map
\[\phi:\left(\Omega_{P}^{\times},\mu,Y_{P}\right)\to\left(\Xi,\xi,Y_{W_{o}}\right),\]
where \(\left(\Xi,\xi,Y_{W_{o}}\right)\) is the transverse cut--and--project \(\mathbb{A}_{\mathrm{fin}}\)-space associated with \(\left(\mathbb{A}_{\mathrm{fin}},\mathbb{R};\Gamma\right)\) and \(W_{o}\). By the general structure of lattices in \(\mathbb{A}\) (recall Example~\ref{exm:adeles2}), there exists \(\left(u,t\right)\in\mathbb{Z}_{\mathrm{fin}}\times\mathbb{R}\) such that \(\Gamma=\mathbb{Q}\left(u,t\right)\). Therefore, the return times set of a point \(\Gamma+\left(u^{\prime},t^{\prime}\right)\in\Xi\) takes the form
\begin{align*}
\left(Y_{W_{o}}\right)_{\Gamma+\left(u^{\prime},t^{\prime}\right)}
&\coloneqq\left\{u^{\prime\prime}\in\mathbb{A}_{\mathrm{fin}}:\exists w\in W_{o},\,\,\Gamma+\left(u^{\prime\prime},w\right)=\Gamma+\left(u^{\prime},t^{\prime}\right)\right\}\\
&=\operatorname{proj}_{\mathbb{A}_{\mathrm{fin}}}\left(\left(\left(u^{\prime},t^{\prime}\right)+\Gamma\right)\cap\left(\mathbb{A}_{\mathrm{fin}}\times W_{o}\right)\right)\\
&=\operatorname{proj}_{\mathbb{A}_{\mathrm{fin}}}\left(\left(\left(u^{\prime},t^{\prime}\right)+\mathbb{Q}\left(u,t\right)\right)\cap\left(\mathbb{A}_{\mathrm{fin}}\times W_{o}\right)\right)\\
&=u^{\prime}+\operatorname{proj}_{\mathbb{A}_{\mathrm{fin}}}\left(\mathbb{Q}\left(u,t\right)\cap\left(\mathbb{A}_{\mathrm{fin}}\times\left(W_{o}-t\right)\right)\right)\\
&=u^{\prime}+u\cdot\operatorname{proj}_{\mathbb{A}_{\mathrm{fin}}}\left(\mathbb{Q}\cap\left(\mathbb{A}_{\mathrm{fin}}\times t^{-1}\cdot\left(W_{o}-t\right)\right)\right)\\
&=u^{\prime}+u\cdot\mathfrak{F}\left(t^{-1}\cdot\left(W_{o}-t^{\prime}\right)\right),
\end{align*}
which is a translation of a dilated generalized Farey fractions. Pick any \(P^{\prime}\in Y_{P}\cap\phi^{-1}\left(Y_{W_{o}}\right)\) (recall this set is \(\mu_{Y_{P}}\)-conull), hence \(\phi\left(P^{\prime}\right)=\Gamma+\left(0_{\mathbb{A}_{\mathrm{fin}}},t_{P^{\prime}}\right)\) for some \(t_{P^{\prime}}\in W_{o}\). Since \(\phi\) is a transverse \(G\)-factor,
\[P^{\prime}=\left\{ u\in\mathbb{A}_{\mathrm{fin}}:e_{\mathbb{A}_{\mathrm{fin}}}\in P^{\prime}-u\right\}=\left(Y_{P}\right)_{P^{\prime}}=\left(Y_{W_{o}}\right)_{\Gamma+\left(0_{\mathbb{A}_{\mathrm{fin}}},t_{P^{\prime}}\right)}=u\cdot\mathfrak{F}\left(t^{-1}\cdot\left(W_{o}-t_{P^{\prime}}\right)\right).\]
Letting the interval \(W\coloneqq t^{-1}\cdot\left(W_{o}-t_{P^{\prime}}\right)\), by~\cite[Lemma 2.4]{BjFi2024} we obtain that
\[P-P\supseteq P^{\prime}-P^{\prime}=u\cdot\mathfrak{F}\left(W\right)-u\cdot\mathfrak{F}\left(W\right).\]
To show that this is a full density subset, on one hand we have
\begin{align*}
d^{\ast}\left(u\cdot\mathfrak{F}\left(W\right)-u\cdot\mathfrak{F}\left(W\right)\right)
&=d^{\ast}\left(u\cdot\mathfrak{F}\left(W-W\right)\right)=d^{\ast}\left(u\cdot\mathfrak{F}\left(t^{-1}\cdot\left(W_{o}-W_{o}\right)\right)\right)\\
&=\left\Vert u\right\Vert _{\mathrm{fin}}^{-1}\cdot\left|t\right|^{-1}\cdot m_{\mathbb{R}}\left(W_{o}-W_{o}\right)=2\cdot\left\Vert u\right\Vert _{\mathrm{fin}}^{-1}\cdot\left|t\right|^{-1}\cdot m_{\mathbb{R}}\left(W_{o}\right),
\end{align*}
where the first equality is by the observation \eqref{eq:twoendpts}; the third equality is by the general formula \eqref{eq:densff}; and, the last equality is because \(W_{o}\) is an interval. On the other hand,
\begin{align*}
d^{\ast}\left(P\right)	
&=\iota_{\mu}\left(Y_{P}\right)=\iota_{\xi}\left(Y_{W_{o}}\right)\\
&=\operatorname{covol}_{\mathbb{A}}\left(\mathbb{Q}\left(u,t\right)\right)^{-1}\cdot m_{\mathbb{R}}\left(W_{o}\right)=\left\Vert u\right\Vert _{\mathrm{fin}}^{-1}\cdot\left|t\right|^{-1}\cdot m_{\mathbb{R}}\left(W_{o}\right),
\end{align*}
where the first equality is by the construction of \(\mu\) as in Proposition~\ref{prop:furst}; the second equality is because intensity is preserved under transverse factors~\cite[Proposition 5.10]{AvBjCuI}, and here there is a transverse \(\mathbb{A}_{\mathrm{fin}}\)-factor \(\left(\Omega_{P}^{\times},\mu,Y_{P}\right)\to\left(\Xi,\xi,Y_{W_{o}}\right)\); and, the third equality is by Theorem~\ref{mthm:cap}. All together, we obtain
\[d^{\ast}\left(P-P\right)=2\cdot d^{\ast}\left(P\right)=2\cdot\left\Vert u\right\Vert _{\mathrm{fin}}^{-1}\cdot\left|t\right|^{-1}\cdot m_{\mathbb{R}}\left(W_{o}\right)=d^{\ast}\left(u\cdot\mathfrak{F}\left(W\right)-u\cdot\mathfrak{F}\left(W\right)\right).\]

\appendix

\section{Fundamentals of class \(\mathcal{Q}\) and cosolenoidal groups}\label{app:classcoss}

In this appendix, we present the classes of groups discussed in this work, namely class \(\mathcal{Q}\) groups and cosolenoidal groups, along with their structural properties, proved in full detail.

\subsection{Class \(\mathcal{Q}\)}

Recall that we have defined \hyperlink{classQ}{class \(\mathcal{Q}\)} as the class of lcsc abelian groups satisfying the following properties:
\begin{enumerate}
    \item The group is totally disconnected, non-discrete and non-compact.
    \item The (Pontryagin) dual of the group is torsion-free.\footnote{For instance, divisible group has torsion-free dual, but the converse is not always true; see~\cite[Theorem (24.23)]{HewittRoss}.}
\end{enumerate}

We will now formulate and prove Propositions~\ref{prop:classqgrp} and~\ref{prop:classqclass}.

\begin{prop}\label{prop:classqgrp-App}
For every \hyperlink{classQ}{class \(\mathcal{Q}\)} group \(G\) the following properties hold:
\begin{enumerate}
    \item \(G\) admits a local base at the identity consisting of compact open subgroups.
    \item \(G\) admits no closed cocompact proper subgroup (and in particular no lattices).
    \item Every compactification of \(G\) is connected.
\end{enumerate}
\end{prop}

\begin{proof}
Property (1) is van Dantzig's theorem; see~\cite[Theorem (7.7)]{HewittRoss}. For Property (2), if \(\Gamma<G\) is closed, cocompact proper subgroup, then \(\Gamma \backslash G\) is a compact group, which is totally disconnected because \(G\) has a basis of compact open subgroups, whose images in the quotient \(\Gamma\backslash G\) is a base at the identity of compact open subgroups. Under the Pontryagin duality, \(\widehat{\Gamma\backslash G}\cong\operatorname{Ann}_{G}\left(\Gamma\right)\) is the annihilator of \(\Gamma\) in \(\widehat{G}\), and it is a torsion group since \(\Gamma\backslash G\) is totally disconnected~\cite[Theorem (24.26)]{HewittRoss}. As \(\widehat{G}\) is torsion-free, \(\operatorname{Ann}_{G}\left(\Gamma\right)\) is trivial, hence \(\Gamma\) is trivial so that \(G\) is compact, which contradicts our assumption. For Property (3), recall that a compact abelian group \(K\) is connected if and only if its dual \(\widehat{K}\) is torsion-free~\cite[Theorem (24.25)]{HewittRoss}. Now if \(\left(K,\tau\right)\) is a compactification of a class \(\mathcal{Q}\) group \(G\) and \(\chi\in\widehat{K}\) is a torsion element, \(\chi\circ\tau\in\widehat{G}\) is also a torsion element which is nontrivial since \(\tau\left(G\right)\) is dense in \(K\), contradicting torsion-freeness of \(\widehat{G}\).
\end{proof}

\begin{prop}
\hyperlink{classQ}{Class \(\mathcal{Q}\)} is closed under restricted direct products.
\end{prop}

\begin{proof}[Proof of Proposition~\ref{prop:classqclass}]
It is easy to see that taking restricted direct products preserves the properties of being totally disconnected, non-discrete and non-compact. For the torsion-freeness of the dual, one only needs that the dual group of a restricted direct product \(\prod\nolimits_{n\in\mathbb{N}}^{\prime}\left(G_{n},K_{n}\right)\) is given by the restricted direct product
\[\prod\nolimits _{n\in\mathbb{N}}^{\prime}\big(\widehat{G_{n}},\operatorname{Ann}_{G_{n}}\left(K_{n}\right)\big),\]
where in general \(\operatorname{Ann}_{G}\left(K\right)\coloneqq\{ \chi\in\widehat{G}:\chi\mid_{K}=1\}\) denotes the annihilator of \(K\) in \(\widehat{G}\)~\cite[(23.33)]{HewittRoss}.
\end{proof}

\subsection{Cosolenoidal groups}

Recall that we call an lcsc abelian group \(G\) \textbf{cosolenoidal} if \(G\times\mathbb{R}\) admits a lattice while \(G\) itself does not. More generally, \(G\) is \textbf{$d$-cosolenoidal}, for \(d\in\mathbb{N}\), if \(G \times \mathbb{R}^{d}\) admits a lattice but \(G \times \mathbb{R}^{d-1}\) does not. Let us formulate and prove Proposition~\ref{prop:cosoH}.

\begin{prop}\label{prop:cosoH-App}
Let \(G\) be an lcsc abelian group. If there exists an lcsc abelian group \(H\) such that \(G\times H\) admits a lattice, then \(G\) is \(d\)-cosolenoidal for some \(d\in\mathbb{N}\).
\end{prop}

In order to prove this proposition we will need the following two basic lemmas.

\begin{lem}\label{lem:latquot}
Let \(G\) be an lcsc abelian group, \(K<G\) be a compact subgroup, and let \(\Gamma<G\) be a lattice. Then \(K\backslash\Gamma<K\backslash G\) is a lattice.
\end{lem}

\begin{proof}[Proof of Lemma~\ref{lem:latquot}]
We will show that \(K\backslash\Gamma<K\backslash G\) is discrete and cocompact. Let \(\pi:G\to K\backslash G\) be the quotient map. For the discreteness, assume \(\pi\left(\gamma_{n}\right)\to e_{K\backslash G}\) with \(\gamma_{n}\in\Gamma\).
Pick representatives \(g_{n}\coloneqq k_{n}\gamma_{n}\in K\gamma_{n}\) such that \(g_{n}\to e_{G}\). Since \(K\) is compact, after passing to a subsequence we may assume that \(k_{n}\to k\in K\). Then \(\gamma_{n}=g_{n}k_{n}^{-1}\to k^{-1}\). As \(\Gamma\) is discrete, \(\gamma_{n}\) is eventually constant, and hence so is \(\pi\left(\gamma_{n}\right)\). This shows that \(\Gamma_{K}\) is discrete. For the cocompactness, since \(\Gamma<G\) is cocompact, there is a compact \(Q\subseteq G\) with \(\Gamma Q=G\). Set \(Q_{o}\coloneqq \pi\left(Q\right)\), which is compact in \(K\backslash G\), and then
\[\left(K\backslash\Gamma\right)Q_{o}=\pi\left(\Gamma\right)\,\pi\left(Q\right)=\pi\left(\Gamma Q\right)=\pi\left(G\right)=K\backslash G.\qedhere\]
\end{proof}

\begin{lem}\label{lem:cosoH}
Let \(A\) be an lcsc abelian group, \(D\) a discrete countable abelian group, and \(\Gamma<A\times D\) a lattice. Then \(\operatorname{proj}_{A}\left(\Gamma\cap\left(A\times\left\{e_{D}\right\}\right)\right)<A\) is a lattice as well.
\end{lem}

\begin{proof}[Proof of Lemma~\ref{lem:cosoH}]
Put
\[\Gamma_{o}\coloneqq\operatorname{proj}_{A}\left(\Gamma\cap\left(A\times\left\{e_{D}\right\}\right)\right).\]
Since \(\operatorname{proj}_{A}\) restricted to \(\Gamma\cap\left(A\times\{e_{D}\}\right)\) is a homeomorphism, \(\Gamma_{o}\) is discrete. Then we will prove that \(\Gamma_{o}<A\) is cocompact. Since \(\Gamma<A\times D\) is cocompact, there exist a compact set \(Q\subseteq A\) and a finite set \(F\subseteq D\) such that
\[\Gamma\left(Q\times F\right)=A\times D.\]
Let \(\Delta\coloneqq\operatorname{proj}_{D}\left(\Gamma\right)\leq D\). For every \(d\in F\cap\Delta\) choose an element \(\gamma_{d}=\left(a_{d},d\right)\in\Gamma\), and set
\[Q_{o}\coloneqq\bigcup\nolimits_{d\in F\cap\Delta}\left(Q-a_{d}\right)\subseteq A.\]
Then \(Q_{o}\) is compact, and we claim that \(\Gamma_{o}Q_{o}=A\). Indeed, fix an arbitrary \(a\in A\). By \(\Gamma\left(Q\times F\right)=A\times D\) there exist
\(\gamma=\left(b,\delta\right)\in\Gamma\) and \(\left(q,d\right)\in Q\times F\) with
\[\left(a,e_{D}\right)=\gamma\left(q,d\right)=\left(b+q,\delta d\right).\]
Then \(d=\delta^{-1}\in\Delta\), so \(d\in F\cap\Delta\) and the chosen
\(\gamma_{d}=\left(a_{d},d\right)\) exists. Multiply inside \(\Gamma\),
\[\gamma\gamma_{d}=\left(b,\delta\right)\left(a_{d},d\right)=\left(b+a_{d},e_{D}\right)\in\Gamma\cap\left(A\times\left\{e_{D}\right\}\right),\]
hence \(b+a_{d}\in\Gamma_{o}\). Since \(q\in Q\), we have \(q-a_{d}\in Q-a_{d}\subseteq Q_{o}\), and therefore
\[a=\left(b+a_{d}\right)+\left(q-a_{d}\right)\in\Gamma_{o}+Q_{o}.\]
We conclude that \(A=\Gamma_{o}Q_{o}\), thus \(\Gamma_{o}\) is cocompact in \(A\), and since it is discrete, it is a lattice.
\end{proof}

We can now deduce Proposition~\ref{prop:cosoH-App}:

\begin{proof}[Proof of Proposition~\ref{prop:cosoH-App}]
By the structure theorem of lcsc abelian groups~\cite[Theorem 24.30]{HewittRoss}, every lcsc abelian group \(H\) is isomorphic to \(\mathbb{R}^{d_{o}}\times H_{o}\) for some \(d_{o}\in\mathbb{Z}_{\geq 0}\), where \(H_{o}\) admits a compact open subgroup \(U_{o}\leq H_{o}\), and therefore \(D_{o}\coloneqq U_{o}\backslash H_{o}\) is a discrete countable abelian group. Letting \(K_{o}\coloneqq\{e_{G}\}\times\{0\}\times U_{o}\), it follows from Lemma~\ref{lem:latquot} that if \(\Gamma<G\times H\cong G\times\mathbb{R}^{d_{o}}\times H_{o}\) is a lattice, then also
\[K_{o}\backslash\Gamma\,<\,K_{o}\backslash\left(G\times\mathbb{R}^{d_{o}}\times H_{o}\right)\,\cong\,G\times\mathbb{R}^{d_{o}}\times D_{o}\]
is a lattice. By Lemma~\ref{lem:cosoH} (with \(A=G\times\mathbb{R}^{d_{o}}\) and \(D=D_{o}\)), \(G\) is \(d\)-cosolenoidal for some \(d\leq d_{o}\).
\end{proof}

Let us formulate and prove Proposition~\ref{prop:cyclicity}. Recall that a finitely generated group is {\bf \(d\)-generated}, for some \(d\in\mathbb{N}\), if it admits a generating set of cardinality at most \(d\).

\begin{prop}
Let \(G\) be a \hyperlink{classQ}{class \(\mathcal{Q}\)} group and \(d\in\mathbb{N}\) such that \(G\times\mathbb{R}^{d}\) admits a lattice. Then for all compact open subgroups \(V<U<G\), either \(U\) admits a nontrivial finite subgroup, or \(V\backslash U\) is \(d\)-generated.
\end{prop}

The proof of this proposition will be based on the following two lemmas.

\begin{lem}\label{lem:cyclicity1}
Let \(K\) be a compact abelian group, such that \(K\times\mathbb{R}^{d}\) admits a lattice for some \(d\in\mathbb{N}\) which projects densely to \(K\) and injectively to \(\mathbb{R}^{d}\). Then for every compact open subgroup \(L<K\), the finite group \(L\backslash K\) is \(d\)-generated.
\end{lem}

\begin{proof}
Let \(\Lambda<K\times\mathbb{R}\) be a lattice such that \(\operatorname{proj}_{K}\left(\Lambda\right)\) is dense in \(K\) and the map \(\operatorname{proj}_{\mathbb{R}^{d}}\mid_{\Gamma}\) is injective. For an open subgroup \(L<K\), put \(\Lambda_{L}\coloneqq\Lambda\cap\left(L\times\mathbb{R}^{d}\right)\), and we have the isomorphisms of finite groups
\begin{equation}\label{eq:isoms}
L\backslash K\cong\Lambda_{L}\backslash\Lambda.
\end{equation}
Indeed, look at the homomorphism \(\varphi:\Lambda\to K\to L\backslash K\), \(\varphi:\left(k,t\right)\mapsto k \mapsto k+L\). The kernel of \(\varphi\) is \(\Lambda_{L}\) and, since \(\mathrm{proj}_{K}\left(\Lambda\right)\) is dense in \(K\) while \(L\backslash K\) is finite, \(\varphi\) is surjective, justifying \eqref{eq:isoms}. Since a quotient of a \(d\)-generated group is \(d\)-generated, and in light of \eqref{eq:isoms}, in order to deduce that \(L\backslash K\) is \(d\)-generated it suffices to show that \(\Lambda\) is \(d\)-generated. To this end we claim that \(\operatorname{proj}_{\mathbb{R}^{d}}\left(\Lambda\right)\) is a nontrivial discrete subgroup of \(\mathbb{R}^{d}\), hence is a lattice in \(\mathbb{R}^{d}\) and in particular \(d\)-generated. The nontriviality follows from the injectivity of \(\operatorname{proj}_{\mathbb{R}^{d}}\mid_{\Lambda}\). To prove the discreteness, it suffices to show that every sequence in \(\operatorname{proj}_{\mathbb{R}^{d}}\left(\Lambda\right)\) converging to \(0\), admits a subsequence which is constant \(0\). Indeed, let \(\left(t_{n}\right)=\left(\operatorname{proj}_{\mathbb{R}^{d}}\left(k_{n},t_{n}\right)\right)\subset\operatorname{proj}_{\mathbb{R}^{d}}\left(\Lambda\right)\) be a sequence converging to \(0\), and since \(K\) is compact, by passing to a subsequence we may assume that \(\left(k_{n}\right)\) converges to some \(k\in K\), and thus \(\left(k_{n},t_{n}\right)\) converges to \(\left(k,0\right)\in\Lambda\). Since \(\Lambda\) is discrete, for all sufficiently large \(n\) we have \(\left(k_{n},t_{n}\right)=\left(k,0\right)\), hence \(t_{n}=0\).
\end{proof}

\begin{lem}\label{lem:cyclicity2}
Let \(G\) be an lcsc abelian group such that \(G\times\mathbb{R}^{d}\) admits a lattice for some \(d\in\mathbb{N}\). Then there exists a cocompact closed subgroup \(G_{o}\leq G\) (which is in fact \(\overline{\operatorname{proj}_{G}\left(\Gamma\right)}\)) with the following property: for every compact open subgroup \(U<G_{o}\), there is a finite group \(F<U\), such that for every compact open subgroup \(V<F\backslash U\), the finite group \(V\backslash\left(F\backslash U\right)\) is \(d\)-generated.
\end{lem}

\begin{proof}
Let \(\Gamma<G\times\mathbb{R}^{d}\) be a lattice and define \(G_{o}\coloneqq\overline{\operatorname{proj}_{G}\left(\Gamma\right)}\). Form the map
\[\Gamma\backslash\left(G\times\mathbb{R}^{d}\right)\to G_{o}\backslash G,\quad\Gamma\left(g,t\right)\mapsto G_{o}g,\]
which is well-defined for if \(\Gamma\left(g,t\right)=\Gamma\left(g^{\prime},t^{\prime}\right)\) then \(g^{\prime}=\gamma_{G}g\) with \(\gamma_{G}\in\operatorname{proj}_{G}\left(\Gamma\right)<G_{o}\), so \(G_{o}g^{\prime}=G_{o}g\). Since this map is a continuous surjection and \(\Gamma\backslash\left(G\times\mathbb{R}^{d}\right)\) is compact, \(G_{o}\backslash G\) is compact as well, and therefore \(G_{o}<G\) is cocompact. Let us then show that \(G_{o}\) satisfies the desired property.

Let \(U<G_{o}\) be a compact open subgroup, and put \(\Gamma_{U}\coloneqq\Gamma\cap\left(U\times\mathbb{R}^{d}\right)\). Since \(\operatorname{proj}_{G_{o}}\left(\Gamma\right)\) is dense in \(G_{o}\) and \(U\) is open, it follows that \(\operatorname{proj}_{G_{o}}\left(\Gamma_{U}\right)\) is dense in \(U\). While \(\operatorname{proj}_{\mathbb{R}^{d}}\mid_{\Gamma_{U}}\) may not be injective, \(F\coloneqq\ker\operatorname{proj}_{\mathbb{R}^{d}}\mid_{\Gamma_{U}}=\Gamma\cap\left(U\times\left\{0\right\} \right)<U\) is a discrete subgroup of \(U\times\left\{ 0\right\}\), and since \(U\) is compact, \(F\) is finite. Put the compact group \(U_{o}\coloneqq F\backslash U\) and consider the lattice \(\Gamma_{o}\coloneqq p_{o}\left(\Gamma_{U}\right)<U_{o}\times\mathbb{R}^{d}\), where \(p_{o}:U\times\mathbb{R}^{d}\to U_{o}\times\mathbb{R}^{d}\) is the map \(p_{o}\left(u,t\right)=\left(Fu,t\right)\). Altogether, we obtain the desired conclusion from Lemma~\ref{lem:cyclicity1} applied to \(K=U_{o}\) and \(\Gamma_{o}\) and with any compact open subgroup \(L=V<K=U_{o}\).
\end{proof}

We can now prove the above proposition:

\begin{proof}
Let \(\Gamma<G\times\mathbb{R}^{d}\) be a lattice. Since \(G\) is class \(\mathcal{Q}\), by Proposition~\ref{prop:classqgrp-App} necessarily \(G_{o}=G\), where \(G_{o}\) is as in Lemma~\ref{lem:cyclicity2}. Let \(V<U<G\) be compact open subgroups, and assume that \(U\) admits no finite subgroups. Then necessarily \(F\) is trivial, where \(F\) is as in Lemma~\ref{lem:cyclicity2}, and therefore by the conclusion of Lemma~\ref{lem:cyclicity2}, \(V\backslash U\) is \(d\)-generated.
\end{proof}

Next, we formulate and prove Proposition~\ref{prop:adlat}:

\begin{prop}
Let \(G\) be a \hyperlink{classQ}{class \(\mathcal{Q}\)} group satisfying the following properties:
\begin{enumerate}
    \item \(G\) is torsion-free and divisible.
    \item There is a compact open subgroup \(U<G\) with no nontrivial finite subgroups, and \(U\backslash G\) is torsion.
\end{enumerate}
Let \(Q\) be the set of orders of elements in \(U\backslash G\). Then every lattice \(\Gamma< G\times\mathbb{R}\) is of the form
\[\Gamma=\mathbb{Z}\big[1/Q\big]\left(u,t\right)=\big\{ \big({\textstyle \frac{m}{q}}u,{\textstyle\frac{m}{q}}t\big):m\in\mathbb{Z},q\in Q\big\}<G\times\mathbb{R},\footnote{The writing \(g/q\) for \(g\in G\) and \(q\in\mathbb{Z}\backslash\{0\}\) is justified since \(G\) is divisible and torsion free.}\]
for some \(\left(u,t\right)\in U\times\mathbb{R}\).
\end{prop}

\begin{proof}
Let \(\Gamma<G\times\mathbb{R}\) be a lattice and put \(\Gamma_{U}\coloneqq\Gamma\cap\left(U\times\mathbb{R}\right)\). We claim that the projection map \(\operatorname{proj}_{\mathbb{R}}:\Gamma_{U}\to\operatorname{proj}_{\mathbb{R}}\left(\Gamma_{U}\right)\) forms an isomorphism of \(\Gamma_{U}\) onto a discrete subgroup of \(\mathbb{R}\). Indeed, \(\ker\operatorname{proj}_{\mathbb{R}}\mid_{\Gamma_{U}}=\Gamma\cap\left(U\times\left\{ 0\right\} \right)\) is a finite subgroup of \(U\) (since \(\Gamma\) is discrete and \(U\) is compact), and since \(U\) admits no nontrivial finite subgroups this kernel is trivial. We have left to show that \(\mathrm{proj}_{\mathbb{R}}\left(\Gamma_{U}\right)\) is a discrete subgroup of \(\mathbb{R}\). Suppose \(\left(t_{n}\right)_{n\in\mathbb{N}}\subset\mathrm{proj}_{\mathbb{R}}\left(\Gamma_{U}\right)\) is a sequence with \(t_{n}\to0\) as \(n\to\infty\). Then there exists a sequence \(\left(u_{n}\right)_{n\in\mathbb{N}}\subset U\) such that \(\gamma_{n}\coloneqq\left(u_{n},t_{n}\right)\in\Gamma\) for every \(n\in\mathbb{N}\), and since \(U\) is compact, we may assume (after passing to a subsequence) that \(u_{n}\to u\) as \(n\to\infty\) for some \(u\in U\), and consequently \(\gamma_{n}\to\left(u,0\right)\) as \(n\to\infty\). Because \(\Gamma\) is discrete and \(\left(\gamma_{n}\right)_{n\in\mathbb{N}}\subset\Gamma\), it follows that \(\gamma_{n}=\left(u,0\right)\) for all sufficiently large \(n\), so \(t_{n}=0\) for all sufficiently large \(n\). Thus every sequence in \(\mathrm{proj}_{\mathbb{R}}\left(\Gamma_{U}\right)\) that converges to \(0\) is eventually constant, hence \(\mathrm{proj}_{\mathbb{R}}\left(\Gamma_{U}\right)\) is a discrete subgroup of \(\mathbb{R}\).

We conclude that \(\mathrm{proj}_{\mathbb{R}}\mid_{\Gamma_{U}}:\Gamma_{U}\to\mathrm{proj}_{\mathbb{R}}\left(\Gamma_{U}\right)\) is an isomorphism of discrete groups. Hence, by the structure of discrete subgroups of \(\mathbb{R}\), there is \(\left(u,t\right)\in U\times\mathbb{R}\) such that
\[
\Gamma_{U}=\mathbb{Z}\left(u,t\right)=\bigl\{(mu,mt):m\in\mathbb{Z}\bigr\}.
\]
We now claim that the following identity holds:
\begin{equation}\label{eq:gammaZQ}
\Gamma=\mathbb{Z}\big[1/Q^{\prime}\big]\left(u,t\right),
\end{equation}
where \(Q^{\prime}\) denotes the set of orders of elements in \(\left\{Ug:g\in\mathrm{proj}_{G}\left(\Gamma\right)\right\}\subseteq U\backslash G\). Let \((\gamma_{1},\gamma_{2})\in\Gamma\) be arbitrary. Because \(U\backslash G\) is torsion, there is \(q\in Q^{\prime}\) with \(q\gamma_{1}\in U\) (note \(\gamma_{1}\in\mathrm{proj}_{G}\left(\Gamma\right)\)), and then
\[\left(q\gamma_{1},q\gamma_{2}\right)\in\Gamma_{U}=\mathbb{Z}\left(u,t\right),\]
so \(\left(q\gamma_{1},q\gamma_{2}\right)=\left(mu,mt\right)\) for some \(m\in\mathbb{Z}\), and using the divisibility of \(G\) we get
\[\left(\gamma_{1},\gamma_{2}\right)=\frac{m}{q}\left(u,t\right)\in\mathbb{Z}\big[1/q\big]\left(u,t\right)\subseteq\mathbb{Z}\big[1/Q^{\prime}\big]\left(u,t\right).\]
Conversely, let \(q\in Q^{\prime}\) be arbitrary and pick \(g\in\mathrm{proj}_{G}\left(\Gamma\right)\) such that \(Ug\in U\backslash G\) has order \(q\), thus \(qg\in U\). Pick some \(s\in\mathbb{R}\) with \(\left(g,s\right)\in\Gamma\), and as above we may write
\[\left(qg,qs\right)=m\left(u,t\right)\quad\text{for some }m\in\mathbb{Z},\]
so \(qg=mu\in U\). The minimality of \(q\) forces \(\gcd\left(m,q\right)=1\) (if \(d\coloneqq\gcd\left(m,q\right)>1\) then from \(qg=mu\) we get \(\left(q/d\right)g=\left(m/d\right)u\in U\), contradicting the minimality of \(q\)). Then pick \(a,b\in\mathbb{Z}\) with \(am+bq=1\) (B\'{e}zout's identity), and using the divisibility of \(G\) we get
\[\frac{1}{q}\left(u,t\right)=a\left(g,s\right)+b\left(u,t\right)\in\Gamma,\]
so each generator \(\frac{1}{q}\left(u,t\right)\in\mathbb{Z}\left[1/Q^{\prime}\right]\left(u,t\right)\) lies in \(\Gamma\). This establishes \eqref{eq:gammaZQ}.

Finally, let \(Q\) be the set of orders of elements in \(U\backslash G\), and we will prove that \(Q=Q^{\prime}\). The inclusion \(Q^{\prime}\subseteq Q\) is clear, and for the reverse inclusion it suffices to show that \(U\mathrm{proj}_{G}\left(\Gamma\right)=G\). Since \(\Gamma\) is a lattice in the abelian group \(G\times\mathbb{R}\), there is a compact set \(C\subset G\times\mathbb{R}\) with \(\Gamma C=G\times\mathbb{R}\), hence
\[G=\mathrm{proj}_{G}\left(\Gamma\right)K,\text{ where }K\coloneqq\mathrm{proj}_{G}\left(C\right).\]
Since \(U\) is compact open and \(K\) is compact, there is a finite set \(F\subseteq K\) such that \(UK=UF\), hence
\[G=\mathrm{proj}_{G}\left(\Gamma\right)K\subseteq\left(U\mathrm{proj}_{G}\left(\Gamma\right)\right)\left(UK\right)=\left(U\mathrm{proj}_{G}\left(\Gamma\right)\right)F,\]
which means that \(U\mathrm{proj}_{G}\left(\Gamma\right)\) is a finite index subgroup of \(G\). Since \(G\) is torsion-free and divisible, it has no proper finite index subgroups, and therefore \(U\mathrm{proj}_{G}\left(\Gamma\right)=G\). This completes the proof.
\end{proof}

\nocite{*}
\bibliographystyle{amsplain}
\bibliography{references}

\providecommand{\bysame}{\leavevmode\hbox to3em{\hrulefill}\thinspace}
\providecommand{\MR}{\relax\ifhmode\unskip\space\fi MR }
\providecommand{\MRhref}[2]{%
  \href{http://www.ams.org/mathscinet-getitem?mr=#1}{#2}
}
\providecommand{\href}[2]{#2}
\begin{thebibliography}{10}

\bibitem{AvBjCuI}
N.~Avraham-Re'em, M.~Bj\"{o}rklund, and R.~Cullman, \emph{Locally integrable cross sections and their intersection covolume}, arXiv:2509.20836 (2025).

\bibitem{BaakeGrimm2013}
Michael Baake and Uwe Grimm, \emph{Aperiodic order, volume 1}, Cambridge University Press, 2013.

\bibitem{BjFi2019}
M.~Bj\"{o}rklund and A.~Fish, \emph{Approximate invariance for ergodic actions of amenable groups}, Discrete Analysis (2019).

\bibitem{BjFi2024}
\bysame, \emph{A {S}zemer\'{e}di type theorem for sets of positive density in approximate lattices}, Ergodic Theory and Dynamical Systems (2024), 1--31.

\bibitem{BjHaAL}
M.~Bj\"{o}rklund and T.~Hartnick, \emph{Approximate lattices}, Duke Mathematical Journal \textbf{167} (2018), no.~15, 2903--2964.

\bibitem{BjHa2022}
\bysame, \emph{Spectral theory of approximate lattices in nilpotent lie groups}, Mathematische Annalen \textbf{384} (2022), no.~3, 1675--1745.

\bibitem{BjHaPoI}
M.~Bj\"{o}rklund, T.~Hartnick, and F.~Pogorzelski, \emph{Aperiodic order and spherical diffraction, {I}: auto-correlation of regular model sets}, Proceedings of the London Mathematical Society \textbf{116} (2018), no.~4, 957--996.

\bibitem{BjFi2021}
Michael Bj\"orklund, Alexander Fish, and Ilya Shkredov, \emph{Sets of transfer times with small densities}, Journal de l'\'{E}cole polytechnique -- Math\'{e}matiques \textbf{8} (2021), 311--329.

\bibitem{bjorklund2025int}
Michael Bj\"{o}rklund, Tobias Hartnick, and Yakov Karasik, \emph{Intersection spaces and multiple transverse recurrence}, Journal d'Analyse Math\'{e}matique (2025), 1--54.

\bibitem{EinsiedlerWard}
M.~Einsiedler and T.~Ward, \emph{Ergodic theory: with a view towards number theory}, Graduate Texts in Mathematics, vol. 259, Springer, 2011.

\bibitem{gardner2002brunn}
R.~Gardner, \emph{The {B}runn--{M}inkowski inequality}, Bulletin of the American Mathematical Society \textbf{39} (2002), no.~3, 355--405.

\bibitem{HewittRoss}
E.~Hewitt and K.~Ross, \emph{Abstract harmonic analysis, volume {I}}, 2nd ed., Grundlehren der Mathematischen Wissenschaften, vol. 115, Springer, 1979.

\bibitem{hofmann2006structure}
K.~H. Hofmann and S.~A. Morris, \emph{The structure of compact groups: A primer for students -- a handbook for the expert}, Walter de Gruyter, 2006.

\bibitem{Kneser}
M.~Kneser, \emph{Summenmengen in lokalkompakten abelschen gruppen}, Mathematische Zeitschrift \textbf{66} (1956), 88--110.

\bibitem{lang1994algebraic}
S.~Lang, \emph{Algebraic number theory}, Graduate Texts in Mathematics, vol. 110, Springer, 1994.

\bibitem{meyer1972algebraic}
Y.~Meyer, \emph{Algebraic numbers and harmonic analysis}, vol.~2, Elsevier, 1972.

\bibitem{shkredov2015}
I.~D. Shkredov, \emph{An introduction to higher energies and sumsets}, arXiv preprint arXiv:1512.00627 (2015).

\bibitem{szabo1997simple}
L\'{a}szl\'{o} Szab\'{o}, \emph{A simple proof for the {J}ordan measurability of convex sets}, Elemente der Mathematik \textbf{52} (1997), no.~2, 84--86.

\end{thebibliography}

\end{document}